\documentclass{article}

\usepackage{microtype}
\usepackage{graphicx}
\usepackage{subfigure}
\usepackage{booktabs} 

\usepackage{hyperref}

\usepackage[accepted]{icml2023}

\usepackage{amsmath}
\usepackage{amssymb}
\usepackage{mathtools}
\usepackage{amsthm}

\usepackage[capitalize,noabbrev]{cleveref}

\theoremstyle{plain}
\newtheorem{theorem}{Theorem}[section]

\newtheorem{lemma}[theorem]{Lemma}
\newtheorem{corollary}[theorem]{Corollary}
\theoremstyle{definition}
\newtheorem{defi}[theorem]{Definition}
\newtheorem{assumption}[theorem]{Assumption}
\theoremstyle{remark}

\allowdisplaybreaks

\usepackage[textsize=tiny]{todonotes}

\usepackage{enumitem}

\DeclareMathOperator*{\argmin}{arg\,min}

\newcommand{\R}{\mathbb{R}}
\newcommand{\E}{\mathbb{E}}
\newcommand{\T}{^{\top}}
\newcommand{\G}{\mathcal{G}}
\newcommand{\F}{\mathcal{F}}

\newcommand{\cD}{\mathcal{D}}
\newcommand{\bS}{\mathbb{S}}

\newcommand{\cO}{\mathcal{O}}
\newcommand{\bfone}{\mathbf{1}}
\newcommand{\bfonet}{\mathbf{1}^{\top}}

\usepackage{color, colortbl}
\definecolor{LightCyan}{rgb}{0.88,1,1}
\def\<#1,#2>{\langle #1,#2\rangle}
\usepackage{tikz}
\usetikzlibrary{shapes.geometric, arrows}
\tikzstyle{lemma} = [rectangle, rounded corners, minimum width=3cm, minimum height=1cm,text centered, draw=black]
\tikzstyle{arrow} = [thick,->,>=stealth]

\icmltitlerunning{Decentralized Stochastic Bilevel Optimization}

\begin{document}

\twocolumn[
\icmltitle{Decentralized Stochastic Bilevel Optimization with Improved per-Iteration Complexity}



\icmlsetsymbol{equal}{*}

\begin{icmlauthorlist}
\icmlauthor{Xuxing Chen}{yyy1}
\icmlauthor{Minhui Huang}{yyy2}
\icmlauthor{Shiqian Ma}{yyy3}
\icmlauthor{Krishnakumar Balasubramanian}{yyy4}

\end{icmlauthorlist}

\icmlaffiliation{yyy1}{Department of Mathematics, University of California, Davis, USA}
\icmlaffiliation{yyy2}{Department of Electrical and Computer Engineering, University of California, Davis, USA}
\icmlaffiliation{yyy3}{Department of Computational Applied Mathematics and Operations Research, Rice University, Houston, USA}
\icmlaffiliation{yyy4}{Department of Statistics, University of California, Davis, USA}

\icmlcorrespondingauthor{Xuxing Chen}{xuxchen@ucdavis.edu}

\icmlkeywords{Machine Learning, Bilevel Optimization, Decentralized Optimization, ICML}

\vskip 0.3in
]



\printAffiliationsAndNotice{}  

\begin{abstract}
    Bilevel optimization recently has received tremendous attention due to its great success in solving important machine learning problems like meta learning, reinforcement learning, and hyperparameter optimization. Extending single-agent training on bilevel problems to the decentralized setting is a natural generalization, and there has been a flurry of work studying decentralized bilevel optimization algorithms. However, it remains unknown how to design the distributed algorithm with sample complexity and convergence rate comparable to SGD for stochastic optimization, and at the same time without directly computing the exact Hessian or Jacobian matrices. In this paper we propose such an algorithm. More specifically, we propose a novel decentralized stochastic bilevel optimization (DSBO) algorithm that only requires first order stochastic oracle, Hessian-vector product and Jacobian-vector product oracle. The sample complexity of our algorithm matches the currently best known results for DSBO, while our algorithm does not require estimating the full Hessian and Jacobian matrices, thereby possessing to improved per-iteration complexity.
\end{abstract}

\section{Introduction}
Many machines learning problems can be formulated as a bilevel optimization problem of the form, \\
\begin{equation}\label{eq: bilevel_opt}
    \begin{aligned}
        &\min_{x\in\R^p}\quad \Phi(x) = f(x,y^*(x)) \\
        &\text{s.t.}\quad y^*(x) = \argmin_{y\in\mathbb{R}^{q}}g(x,y),
    \end{aligned}
\end{equation}
where we minimize the upper level function $f$ with respect to $x$ subject to the constraint that $y^*(x)$ is the minimizer of the lower level function. Its applications can range from classical optimization problems like compositional optimization \citep{chen2021closing} to modern machine learning problems such as reinforcement learning \citep{hong2020two}, meta learning \citep{snell2017prototypical, bertinetto2018meta, rajeswaran2019meta, ji2020convergence}, hyperparameter optimization \citep{pedregosa2016hyperparameter, franceschi2018bilevel}, etc. State-of-the-art bilevel optimization algorithms with non-asymptotic analyses include BSA \citep{ghadimi2018approximation}, TTSA \citep{hong2020two}, StocBiO \citep{ji2020convergence}, ALSET \citep{chen2021closing}, to name a few. 

Decentralized bilevel optimization aims at solving bilevel problems in a decentralized setting, which provides additional benefits such as faster convergence, data privacy preservation and robustness to low network bandwidth compared to the centralized setting and the single-agent training \citep{lian2017can}. For example, decentralized meta learning, which is a special case of decentralized bilevel optimization, arise naturally in the context of medical data analysis in the context of protecting patient privacy; see, for example, \citet{altae2017low, zhang2019metapred, kayaalp2022dif}. Motivated by such applications, the works of \citet{lu2022decentralized, chen2022decentralized, yang2022decentralized, gao2022stochastic} proposed and analyzed various decentralized stochastic bilevel optimization (DSBO) algorithms.

From a mathematical perspective, DSBO aims at solving the following problem in a distributed setting:
\begin{equation}\label{eq: dsbo_opt}
    \begin{aligned}
        &\min_{x\in\R^p}\quad \Phi(x) = \frac{1}{n}\sum_{i=1}^{n}f_i(x,y^*(x)) \\
        &\text{s.t.}\quad y^*(x) = \argmin_{y\in\R^{q}}g(x,y):= \frac{1}{n}\sum_{i=1}^{n}g_i(x,y),
    \end{aligned}
\end{equation}
where $x\in \mathbb{R}^{p}, y\in\mathbb{R}^{q}$. $f_i$ is possibly nonconvex and $g_i$ is strongly convex in $y$. Here $n$ denotes the number of agents, and agent $i$ only has access to stochastic oracles of $f_i,\ g_i$. The local objectives $f_i$ and $g_i$ are defined as:
\begin{align*}
    f_i(x,y) &= \mathbb{E}_{\phi\sim \cD_{f_i}}\left[F(x,y;\phi)\right],\\
    g_i(x,y) &= \mathbb{E}_{\xi\sim \cD_{g_i}}\left[G(x,y;\xi)\right].
\end{align*}
$\cD_{f_i}$ and $\cD_{g_i}$ represent the data distributions used to generate the objectives for agent $i$, and each agent only has access to $f_i$ and $g_i$. In practice we can replace the expectation by empirical loss, and then use samples to approximate the gradients in the updates. Existing works on DSBO require computing the full Hessian (or Jacobian) matrices in the hypergradient estimation, whose per-iteration complexity is $\cO(q^2)$ (or $\cO(pq)$). In problems like hyperparameter estimation, the lower level corresponds to learning the parameters of a model. When considering modern overparametrized models, the order of $q$ is hence extremely large. Hence, to reduce the per-iteration complexity, it is of great interest to have each iteration based only on Hessian-vector (or Jacobian-vector) products, whose complexity is $\cO(q)$ (or $\cO(p)$); see, for example, \citet{pearlmutter1994fast}.

\subsection{Our contributions}
Our contributions in this work are as follows. 
\begin{itemize}[noitemsep,leftmargin=*]
    \item We propose a novel method to estimate the global hypergradient. Our method estimates the product of the inverse of the Hessian and vectors directly, without computing the full Hessian or Jacobian matrices, and thus improves the previous overall (both computational and communication) complexity on hypergradient estimation from $\cO(Nq^2)$ to $\cO(Nq)$, where $N$ is the total steps of the hypergradient estimation subroutine.
    
    \item We design a DSBO algorithm (see Algorithm \ref{algo:DSBO}), and in Theorem~\ref{thm: main} and Corollary \ref{cor: samples} we show the sample complexity is of order $\cO(\epsilon^{-2}\log \frac{1}{\epsilon})$, which matches the currently well-known results of the single-agent bilevel optimization \citep{chen2021closing}. Our proof relies on weaker assumptions comparing to \citet{yang2022decentralized}, and is based on carefully combining moving average stochastic gradient estimation analyses with the decentralized bilevel algorithm analyses.

\item We conduct experiments on several machine learning problems. Our numerical results show the efficiency of our algorithm in both the synthetic and the real-world problems. Moreover, since our algorithm does not store the full Hessian or Jacobian matrices, both the space complexity and the communication complexity are improved comparing to \citet{chen2022decentralized, yang2022decentralized}. 
\end{itemize}

\subsection{Related work}
\textbf{Bilevel optimization.} Different from classical constrained optimization, bilevel optimization restricts certain variables to be the minimizer of the lower level function, which is more applicable in modern machine learning problems like meta learning \citep{snell2017prototypical, bertinetto2018meta, rajeswaran2019meta} and hyperparameter optimization \citep{pedregosa2016hyperparameter, franceschi2018bilevel}. In recent years, \citet{ghadimi2018approximation} gave the first non-asymptotic analysis of the bilevel stochastic approximation methods, which attracted much attention to study more efficient bilevel optimization algorithms including AID-based \citep{domke2012generic, pedregosa2016hyperparameter, gould2016differentiating, ghadimi2018approximation, grazzi2020iteration, ji2021bilevel}, ITD-based \citep{domke2012generic, maclaurin2015gradient, franceschi2018bilevel, grazzi2020iteration, ji2021bilevel}, and Neumann series-based \citep{ chen2021closing,hong2020two,ji2021bilevel} methods. These methods only require access to first order stochastic oracles and matrix-vector product (Hessian-vector and Jacobian-vector) oracles, which demonstrate great potential in solving bilevel optimization problems and achieve $\tilde \cO(\epsilon^{-2})$ sample complexity \citep{chen2021closing, arbel2021amortized} that matches the result of SGD for single level stochastic optimization ignoring the log factors. Moreover, under stronger assumptions and variance reduction techniques, better complexity bounds are obtained \citep{guo2021randomized, khanduri2021near, yang2021provably, chen2022single}.

\textbf{Decentralized optimization.} Extending optimization algorithms from a single-agent setting to a multi-agent setting has been studied extensively in recent years thanks to the modern parallel computing. Decentralized optimization, which does not require a central node, serves as an important part of distributed optimization. Because of data heterogeneity and the absence of a central node, decentralized optimization is more challenging and each node communicates with neighbors to exchange information and solve a finite-sum optimization problem. Under certain scenarios, decentralized algorithms are more preferable comparing to centralized ones since the former preserve data privacy \citep{ram2009asynchronous, yan2012distributed, wu2017decentralized, koloskova2020unified} and have been proved useful when the network bandwidth is low \citep{lian2017can}.

\textbf{Decentralized stochastic bilevel optimization.} To make bilevel optimization applicable in parallel computing, recent work started to focus on distributed stochastic bilevel optimization. FEDNEST \citep{tarzanagh2022fednest} and FedBiO \citep{fedbilevel} impose federated learning, which is essentially a centralized setting, on stochastic bilevel optimization. 
Existing work on DSBO can be classified to two categories: global DSBO and personalized DSBO. Problem \eqref{eq: dsbo_opt} that we consider in this paper is a global DSBO, where both lower-level and upper-level functions are not directly accessible to any local agent. Other works on global DSBO include \citet{chen2022decentralized, yang2022decentralized, gao2022stochastic}\footnote{Here we point out that although \citet{gao2022stochastic} claim that they solve the global DSBO, based on equations (2) and (3) in their paper (\url{https://arxiv.org/abs/2206.15025v1}), it is clear that they are only solving a special case of global DSBO problem. See appendix \ref{sec: MDBO_discussion} for detailed discussion.}. The personalized DSBO \citep{lu2022decentralized} replaces $y^*(x)$ by the local one $y_i^*(x) =\argmin_{y\in\R^q}g_i(x,y)$ in \eqref{eq: dsbo_opt}, which leads to
\begin{equation}\label{eq: pdsbo_opt}
    \begin{aligned}
        &\min_{x\in\R^p}\quad \Phi(x) = \frac{1}{n}\sum_{i=1}^{n}f_i(x,y_i^*(x)) \\
        &\text{s.t.}\quad y_i^*(x) = \argmin_{y\in\R^{q}}g_i(x,y), i=1,\ldots,n.
    \end{aligned}
\end{equation}
To solve global DSBO \eqref{eq: dsbo_opt}, \citet{chen2022decentralized} proposes a JHIP oracle to estimate the Jacobian-Hessian-inverse product while \citet{yang2022decentralized} introduces a Hessian-inverse estimation subroutine based on Neumann series approach which can be dated back to \citet{ghadimi2018approximation}. However, they both require computing the full Jacobian or Hessian matrices, which is extremely time-consuming when $q$ is large. In comparison, computing a Hessian-vector or Jacobian-vector product is more efficient in large-scale machine learning problems \citep{bottou2018optimization}, and is commonly used in vanilla bilevel optimization \citep{ghadimi2018approximation, ji2021bilevel, chen2021closing} to avoid computing the Hessian inverse.  In personalized DSBO \eqref{eq: pdsbo_opt}, local computation is sufficient to approximate $\nabla f_i(x,y_i^*(x))$, and thus does not require computing the  Hessian or Jacobian matrices and single-agent bilevel optimization methods can be directly incorporated in the distributed regime. In our paper we propose a novel algorithm that estimates the global hypergradient using only first-order oracle and matrix-vector products oracle. Based on this we further design our algorithm for solving DSBO that does not require to compute the full Jacobian or Hessian matrices. We summarize the results of aforementioned works and our results in Table \ref{table: compare}.

{\bf Notation.}
We denote by $\nabla f(x,y)$ and $\nabla^2f(x,y)$ the gradient and Hessian matrix of $f$, respectively. We use $\nabla_x f(x,y)$ and $\nabla_y f(x,y)$ to represent the gradients of $f$ with respect to $x$ and $y$, respectively. Denote by $\nabla_{xy}^2f(x,y)\in\mathbb{R}^{p\times q}$ the Jacobian matrix of $f$ and $\nabla_y^2f(x,y)$ the Hessian matrix of $f$ with respect to $y$. $\|\cdot\|$ denotes the $\ell_2$ norm for vectors and Frobenius norm for matrices, unless specified. $\mathbf{1}_n$ is the all one vector in $\mathbb{R}^{n}$, and $J_n = \mathbf{1}_n\mathbf{1}_n^{\top}$ is the $n\times n$ all one matrix. We use uppercase letters to represent the matrix that collecting all the variables (corresponding lowercase) as columns. For example $X_k = \left(x_{1,k}, ..., x_{n,k}\right),\ Y_k^{(t)} = \left(y_{1,k}^{(t)},...,y_{n,k}^{(t)}\right)$. We add an overbar to a letter to denote the average over all nodes. For example, $\bar{x}_k = \frac{1}{n}\sum_{i=1}^{n}x_{i,k},\ \bar{y}_{k}^{(t)} = \frac{1}{n}\sum_{i=1}^{n}y_{i,k}^{(t)}$.

\begin{table*}[t]
\caption{We compare our Algorithm \ref{algo:DSBO} (MA-DSBO) with existing distributed bilevel optimization algorithms: FEDNEST \citep{tarzanagh2022fednest}, SPDB \citep{lu2022decentralized}, DSBO-JHIP \citep{chen2022decentralized}, and GBDSBO \citep{yang2022decentralized}. The problem types include Federated Bilevel Optimization (FBO), Personalized-Decentralized Stochastic Bilevel Optimization (P-DSBO), and Global-Decentralized Stochastic Bilevel Optimization (G-DSBO). In the table we define $d = \max\left(p, q\right)$. 'Computation' (See Section \ref{sec: comp} for details) and 'Samples' represent the computational and sample complexity of finding an $\epsilon$-stationary point, respectively. $\tilde{\cO}$ hides the $\log(\frac{1}{\epsilon})$ factor. 
'Jacobian' refers to whether the algorithm requires computing full Hessian or Jacobian matrix. 'Mini-batch' refers to whether the algorithm requires their batch size depending on $\epsilon^{-1}$.}
\label{table: compare}
\vskip 0.15in
\begin{center}
\begin{small}
\begin{sc}
\begin{tabular}{ccccccc}
\toprule
\textbf{Algorithm}& \textbf{Problem} & \textbf{Computation}       & \textbf{Samples} & \textbf{Jacobian} & \textbf{Mini-batch} & \textbf{Network} \\ 
\midrule
\textbf{FEDNEST} & FBO  & $\tilde\cO(d\epsilon^{-2})$ & $\tilde{\cO}(\epsilon^{-2})$    & No &No &Centralized \\ 
\textbf{SPDB}      & P-DSBO  & $\tilde\cO(dn^{-1}\epsilon^{-2})$ & $\tilde{\cO}(n^{-1}\epsilon^{-2})$   & No  & Yes &Decentralized\\ 
\textbf{DSBO-JHIP}    & G-DSBO  &$\tilde\cO(pq\epsilon^{-3})$ & $\tilde\cO(\epsilon^{-3})$ & Yes & No &Decentralized\\
\textbf{GBDSBO}    & G-DSBO  & $\cO((q^2\log(\frac{1}{\epsilon}) + pq)n^{-1}\epsilon^{-2})$ & $\tilde{\cO}(n^{-1}\epsilon^{-2})$ & Yes & No &Decentralized\\ 
\textbf{MA-DSBO}    & G-DSBO  & $\tilde\cO(d \epsilon^{-2})$ & $\tilde{\cO}(\epsilon^{-2}) $ & No  & No &Decentralized\\
\bottomrule
\end{tabular}
\end{sc}
\end{small}
\end{center}
\vskip -0.1in
\end{table*}

\section{Preliminaries}
The following assumptions are used throughout this paper. They are standard assumptions that are made in the literature on bilevel optimization \citep{ghadimi2018approximation,hong2020two,chen2021closing,ji2021bilevel, huang2022efficiently} and decentralized optimization \citep{qu2017harnessing, nedic2017achieving, lian2017can, tang2018d}.

\begin{assumption}[Smoothness]\label{assump: lip_convexity}
    There exist positive constants $\mu_g$, $L_{f,0}$, $L_{f,1}$, $L_{g,1}$, $L_{g,2}$ such that for any $i$, functions $f_i$, $\nabla f_i$, $\nabla g_i$, $\nabla^2g_i$ are $L_{f,0}$, $L_{f,1}$, $L_{g,1}$, $L_{g,2}$ Lipschitz continuous respectively, and function $g_i$ is $\mu_g$-strongly convex in $y$.
\end{assumption}

\begin{assumption}[Network topology]\label{assump: W}
    The weight matrix $W = (w_{ij})\in \mathbb{R}^{n\times n}$ is symmetric and doubly stochastic, i.e.:
    \[
        W = W^{\top},\quad W\mathbf{1_n} = \mathbf{1_n},\quad w_{ij}\geq 0,\forall i,j,
    \]
    and its eigenvalues satisfy $1 = \lambda_1> \lambda_2\geq ...\geq \lambda_n$ and $\rho := \max\{|\lambda_2|, |\lambda_n|\} < 1$.
\end{assumption}
The weight matrix given in Assumption \ref{assump: W} characterizes the network topology by setting the weight parameter between agent $i$ and agent $j$ to be $w_{ij}$. The condition $\rho < 1$ is termed as 'spectral gap' \citep{lian2017can}, and is used in distributed optimization to ensure the decay of the consensus error, i.e., $\frac{\E\left[\|X_k - \bar{x}_k\bfonet_n\|^2\right]}{n}$, among the agents, which eventually guarantees the consensus among agents.

\begin{assumption}[Gradient heterogeneity]\label{assump: similarity_of_g}
    There exists a constant $\delta \geq 0$ such that for all $1\leq i\leq n, x\in\R^p, y\in \R^q,$
    \[
        \|\nabla_yg_i(x,y) - \frac{1}{n}\sum_{l=1}^{n}\nabla_y g_l(x,y)\| \leq \delta.
    \]
\end{assumption}
The above assumption is commonly used in distributed optimization literature (see, e.g., \citet{lian2017can}), and it indicates the level of  similarity between the local gradient and the global gradient. Moreover, it is weaker than the Assumption 3.4 (iv) of \citet{yang2022decentralized} which assumes that $\nabla_yg_i(x,y;\xi)$ has a bounded second moment. This is because the bounded second moment implies the boundedness of $\nabla_yg(x,y)$, as we have
\begin{align*}
    &\|\nabla_yg(x,y)\|^2 \\
    \leq &\E\left[\|\nabla_yg(x,y) - \nabla_yg(x,y;\xi)\|^2\right] + \|\nabla_yg(x,y)\|^2 \\
    = &\E\left[\|\nabla_yg(x,y;\xi)\|^2\right] \text{ -- uniformly bounded,}
\end{align*}
where the equality holds since we have $\E\left[\|X\|^2\right] = \E\left[\|X-\E\left[X\right]\|^2\right] + \|\E\left[X\right]\|^2$ for any random vector $X$. 
It directly gives the inequality in Assumption \ref{assump: similarity_of_g}. However Assumption \ref{assump: similarity_of_g} does not imply the boundedness of $\nabla_yg(x,y)$ (e.g., $g_i(x,y) = y^{\top}y$ for all $i$ satisfies Assumption \ref{assump: similarity_of_g} but does not have bounded gradient.)
\begin{assumption}[Bounded variance]\label{assump: stoc_derivatives}
     The stochastic derivatives, $\nabla f_i(x,y;\phi)$, $\nabla g_i(x,y;\xi)$, and $\nabla^2 g_i(x,y;\xi)$, are unbiased with bounded variances $\sigma_f^2$, $\sigma_{g,1}^2$, $\sigma_{g,2}^2$, respectively.
\end{assumption}
Note that we do not make any assumptions on whether the data distributions are heterogeneous or identically distributed.

\section{DSBO Algorithm with Improved Per-Iteration Complexity}\label{sec: algorithms}
We start with following standard result in the bilevel optimization literature \citep{ghadimi2018approximation, hong2020two,ji2020convergence,chen2021closing} that gives a closed form expression of the hypergradient $\nabla\Phi(x)$, making gradient-based bilevel optimization tractable.
\begin{lemma}
    Suppose Assumption \ref{assump: lip_convexity} holds. The hypergradient $\nabla\Phi(x)$ of \eqref{eq: dsbo_opt} takes the form:
    \begin{equation}\label{eq: hypergrad}
        \begin{aligned}
            \nabla\Phi(x) = &\frac{1}{n}\left(\sum_{i=1}^{n}\nabla_xf_i(\tilde{x})\right) \\
            &- \nabla_{xy}^2g(\tilde{x})\left(\nabla_y^2g(\tilde{x} )\right)^{-1}\left[\frac{1}{n}\left(\sum_{i=1}^{n}\nabla_yf_i(\tilde{x})\right)\right],
        \end{aligned}
    \end{equation}
    where $\tilde{x} = (x, y^*(x))$.
\end{lemma}
We also include smoothness properties of $\nabla \Phi(x)$ and $y^*(x)$ in Section \ref{sec: analysis} in the appendix.

\subsection{Main challenge}

As discussed in \citet{chen2022decentralized} and \citet{yang2022decentralized}, the main challenge in designing DSBO algorithms is to estimate the global hypergradient. This is challenging because of the data heterogeneity across agents, which leads to
\begin{align}\label{neq: mismatch}
    \begin{aligned}
        &\nabla_{xy}^2g(x, y^*(x))\left(\nabla_y^2g(x, y^*(x) )\right)^{-1}\\
        \neq& \frac{1}{n}\sum_{i=1}^{n}\nabla_{xy}^2g_i(x, y_i^*(x))\left(\nabla_y^2g_i(x,y_i^*(x))\right)^{-1},
    \end{aligned}
\end{align}
where $y_i^*(x) =\argmin_{y\in\R^q}g_i(x,y)$. This shows that simply averaging the local hypergradients does not give a good approximation to the global hypergradient. A decentralized approach should be designed to estimate the global hypergradient $\nabla\Phi(x)$. 

To this end, the JHIP oracle in \citet{chen2022decentralized} manages to estimate
\[
    \left(\sum_{i=1}^{n}\nabla_{xy}^2 g_i(x,y^*(x))\right) \left(\sum_{i=1}^{n}\nabla_y^2g_i(x,y^*(x))\right)^{-1}
\]
using decentralized optimization approach, and \citet{yang2022decentralized} proposed to estimate the global Hessian-inverse, i.e., 
\[
    \left(\sum_{i=1}^{n}\nabla_y^2g_i(x,y^*(x))\right)^{-1}
\]
via a Neumann series based approach. Instead of focusing on full matrices computation, we consider approximating
\begin{equation}\label{eq: z_expression}
    z = \left(\sum_{i=1}^{n}\nabla_y^2g_i(x,y^*(x))\right)^{-1}\left(\sum_{i=1}^{n}\nabla_yf_i(x,y^*(x))\right).
\end{equation}
According to \eqref{eq: hypergrad}, the global hypergradient is given by
\begin{equation}\label{eq: hypergrad_new_decompose}
    \nabla\Phi(x)=\frac{1}{n}\sum_{i=1}^{n}(\nabla_x f_i(x,y^*(x)) - \nabla_{xy}^2 g_i(x,y^*(x))z).
\end{equation}
From the above expression we know that as long as node $i$ can have a good estimate of $\nabla_x f_i(x,y^*(x))$ and $\nabla_{xy}^2g_i(x,y^*(x))z$, then on average the update will be a good approximation to the global hypergradient. More importantly, the process of estimating $z$ can avoid computing the full Hessian or Jacobian matrices.

\subsection{Hessian-Inverse-Gradient-Product oracle}\label{sec: HIGP_oracle}

Solving \eqref{eq: z_expression} is essentially a decentralized optimization with a strongly convex quadratic objective function. Suppose each agent only has access to $H_i\in\bS_{++}^{q\times q}$ and $b_i\in\R^q$, and all the agents collectively solve for
\begin{equation}\label{eq: higp_eq}
    \sum_{i=1}^{n}H_iz = \sum_{i=1}^{n}b_i,\ \text{or } z = \left(\sum_{i=1}^{n}H_i\right)^{-1}\left(\sum_{i=1}^{n}b_i\right).
\end{equation}
From an optimization perspective, the above expression is the optimality condition of:
\begin{equation}\label{eq: higp_opt}
    \min_{z\in\mathbb{R}^{q}}\frac{1}{n}\sum_{i=1}^{n}h_i(z),\ \text{where  }h_i(z)=\frac{1}{2}z^{\top}H_iz - b_i\T z.
\end{equation}
Hence we can design a decentralized algorithm to solve for $z$ without the presence of a central server. Based on this observation and \eqref{eq: hypergrad_new_decompose}, we present our Hessian-Inverse-Gradient Product oracle in Algorithm \ref{algo: HIGP_oracle}.

\begin{algorithm}[ht]
    \caption{Hessian-Inverse-Gradient Product oracle}\label{algo: HIGP_oracle}
    \begin{algorithmic}[1]
        \STATE {\bfseries Input:} $(H_{i,t}^{(k)}, b_{i,t}^{(k)})$, for $0\leq t\leq N$ accessible only to agent $i$. Stepsize $\gamma$, iteration number $N$, $d_{i,0}^{(k)} = -b_{i,0}^{(k)},\ s_{i,0}^{(k)} = - b_{i,0}^{(k)},\ \text{and } z_{i,0}^{(k)} = 0$
        
        \FOR{$t=0,1,...,N-1$}
            \FOR{$i=1,...,n$}
                \STATE $z_{i,t+1}^{(k)} = \sum_{j=1}^{n}w_{ij}z_{j,t}^{(k)} - \gamma d_{i,t}^{(k)}$, \\
                \STATE $s_{i,t+1}^{(k)} = H_{i,t+1}^{(k)}z_{i,t+1}^{(k)}- b_{i,t+1}^{(k)}$, \\
                \STATE $d_{i,t+1}^{(k)} =\sum_{j=1}^{n}w_{ij}d_{j,t}^{(k)} + s_{i,t+1}^{(k)} - s_{i,t}^{(k)}$
            \ENDFOR
        \ENDFOR
        \STATE {\bfseries Output:}{$z_{i,N}^{(k)} $ on each node.}
    \end{algorithmic}
\end{algorithm}
It is known that vanilla decentralized gradient descent (DGD) with a constant stepsize only converges to a neighborhood of the optimal solution even under the deterministic setting \citep{yuan2016convergence}. Therefore, one must use diminishing stepsize in DGD, and this leads to the sublinear convergence rate even when the objective function is strongly convex. To resolve this issue, there are various decentralized algorithms with a fixed stepsize \citep{xu2015augmented, shi2015extra, di2016next, nedic2017achieving, qu2017harnessing} achieving linear convergence on a strongly convex function  in the deterministic setting. Among them, one widely used technique is the gradient tracking method \citep{xu2015augmented, qu2017harnessing, nedic2017achieving, pu2021distributed}, which is also incorporated in our Algorithm \ref{algo: HIGP_oracle}. Instead of using the local stochastic gradient in line 4 of Algorithm \ref{algo: HIGP_oracle}, we maintain another set of variables $d_{i,t+1}^{(k)}$ in line 6 as the gradient tracking step. We will utilize the linear convergence property of gradient tracking in our convergence analysis.

\begin{algorithm}[ht]
    \begin{algorithmic}[1]
        \caption{Hypergradient Estimation}\label{algo: Hypergrad}
        \STATE {\bfseries Input:} Samples $\phi = (\phi_{i,0},...,\phi_{i,N}),\ \xi = (\xi_{i,0}, ..., \xi_{i,N})$ on node $i$.
            \STATE Run Algorithm \ref{algo: HIGP_oracle} with \\
            \STATE $H_{i,t}^{(k)} = \nabla_y^2g_i(x_{i,k},y_{i,k}^{(T)};\xi_{i,t}),$ \\ \STATE $b_{i,t}^{(k)} = \nabla_{y}f_i(x_{i,k},y_{i,k}^{(T)};\phi_{i,t})$ \\
            \STATE to get $z_{i,N}^{(k)}$. \\
            \STATE Set $u_{i,k} = \nabla_x f_i(x_{i,k},y_{i,k}^{(T)};\phi_{i,0})$ \\
            \STATE \qquad\qquad $-\nabla_{xy}^2 g_i(x_{i,k},y_{i,k}^{(T)};\xi_{i,0} )z_{i,N}^{(k)}$. \\
        \STATE {\bfseries Output:} $u_{i,k}$ on node $i$.
    \end{algorithmic}
\end{algorithm}

Note that for simplicity we write $H_{i,t}^{(k)}=\nabla_y^2g_i(x_{i,k},y_{i,k}^{(T)};\xi_{i,t})$ in line 3 of Algorithm \ref{algo: Hypergrad}, however, the real implementation only requires Hessian-vector products, as shown in Algorithm \ref{algo: HIGP_oracle}, and we do not need to compute the full Hessian.

\subsection{Decentralized Stochastic Bilevel Optimization}

Now we are ready to present our DSBO algorithm with the moving average technique, which we refer to as the MA-DSBO algorithm. 
\begin{algorithm}[ht]
    \begin{algorithmic}[1]
        \caption{MA-DSBO Algorithm}\label{algo:DSBO}
    	\STATE {\bfseries Input:} Stepsizes $\alpha_k, \beta_k$, iteration numbers $K, T, N$, $y_{i,k}^{(0)} = 0$, and $x_{i,0} = r_{i,0} = 0$. 
    	\FOR{$k=0,1,...,K-1$}
    	    \STATE $y_{i,k}^{(0)} = y_{i,k-1}^{(T)}$. \\
    		\FOR{$t=0,1,..., T-1$}
    		    \FOR{$i=1,...,n$}
        		    \STATE $y_{i,k}^{(t+1)} = \sum_{j=1}^{n}w_{ij}y_{j,k}^{(t)} -\beta_k v_{i,k}^{(t)}$ \text{ with } $v_{i,k}^{(t)} = \nabla_yg_i(x_{i,k}, y_{i,k}^{(t)}; \tilde{\xi}_{i,k}^{(t)})$
    		    \ENDFOR
            \ENDFOR
            \STATE Run Algorithm \ref{algo: Hypergrad} and set the output as $u_{i,k}$. \\
            \FOR{$i=1,...,n$}
                \STATE $x_{i,k+1} = \sum_{j=1}^{n}w_{ij}x_{j,k} - \alpha_k r_{i,k}$. \\
        		\STATE $r_{i,k+1} = (1-\alpha_k)r_{i,k} + \alpha_ku_{i,k}$.
            \ENDFOR
        \ENDFOR
	\STATE {\bfseries Output:} $\bar{x}_K = \frac{1}{n}\sum_{i=1}^{n}x_{i,K}.$
    \end{algorithmic}
\end{algorithm}
In Algorithm \ref{algo:DSBO} we adopt the basic structure of double-loop bilevel optimization algorithm \citep{ghadimi2018approximation, ji2021bilevel, chen2021closing} -- we first run $T$-step inner loop (line 4-8) to obtain a good approximation of $y^*$. Next, we run Algorithm \ref{algo: Hypergrad} to estimate the hypergradient. To reduce the order of the bias in hypergradient estimation error (see Section \ref{sec: MA} for details), we introduce the moving average update to maintain another set of variables $r_{i,k}$ as the update direction of $x$. The using of the moving average update helps reduce the order of bias in the stochastic gradient estimate. It is worth noting that similar techniques have been used in the context of nested stochastic composition optimization in~\citet{ghadimi2020single, balasubramanian2022stochastic}. 
Note that all communication steps of our Algorithms (lines 4 and 6 of Algorithm \ref{algo: HIGP_oracle}, lines 6 and 11 of Algorithm \ref{algo:DSBO}) only include sending (resp. receiving) vectors to (resp. from) neighbors, which greatly reduce the per-iteration communication complexity from $\max\{pq, q^2\}$ of GBDSBO (see line 8 and 11 of Algorithm 1 in \citet{yang2022decentralized}).) to $\max\{p, q\}$.

We now introduce our notion of convergence. Specifically, the $\epsilon$-stationary point of \eqref{eq: pdsbo_opt} is defined as follows. 
\begin{defi}
    For a sequence $\{\bar x_k\}_{k=0}^{K}$ generated by Algorithm \ref{algo:DSBO}, if  $\min_{0\leq k\leq K} \E\left[\|\nabla\Phi(\bar{x}_k)\|^2\right]\leq \epsilon$ for some positive integer $K$, then we say that we find an $\epsilon$-stationary point of \eqref{eq: pdsbo_opt}.
\end{defi}
The above notion of stationary point is commonly used in decentralized non-convex stochastic optimization \citep{lian2017can}. When $\epsilon=0$, it indicates that the hypergradient at some iterate $\bar{x}_k$ is zero. 
The convergence result of Algorithm \ref{algo:DSBO} is given in Theorem \ref{thm: main}.

\begin{theorem}\label{thm: main}
Suppose Assumptions \ref{assump: lip_convexity}, \ref{assump: W}, \ref{assump: similarity_of_g}, and \ref{assump: stoc_derivatives} hold. There exist constants \footnote{The constants are independent of $K$ and the details are given in the appendix.} $0<c_1<c_2$ such that in Algorithm \ref{algo:DSBO} if we set $\gamma\in (c_1, c_2)$, $T\geq 1$, and
    \begin{align*}
        \alpha_k \equiv \Theta\left(\frac{1}{\sqrt{K}}\right),\ \beta_k \equiv \Theta\left(\frac{1}{\sqrt{K}}\right),\ N = \Theta\left(\log K\right), 
    \end{align*}
    then we have
    \[
        \begin{aligned}
            \min_{0\leq k\leq K}\E\left[\|\nabla\Phi(\bar{x}_k)\|^2\right] &= \cO\left(\frac{1}{\sqrt{K}}\right), \\
            \min_{0\leq k\leq K}\frac{\E\left[\|X_k - \bar{x}_k\bfonet_n\|^2\right]}{n} &= \cO\left(\frac{1}{K}\right).
        \end{aligned}
    \]
\end{theorem}

Note that this theorem indicates that the consensus error is of order $\cO\left(\frac{1}{K}\right)$, and for any positive constant $\epsilon$, the iteration complexity of Algorithm \ref{algo:DSBO} for obtaining an $\epsilon$-stationary point of \eqref{eq: dsbo_opt} is $\cO(\epsilon^{-2})$. Moreover, we have the following corollary that gives the sample complexity of our algorithm.
\begin{corollary}\label{cor: samples}
    Suppose the conditions of Theorem \ref{thm: main} hold. For any $\epsilon>0$, if we set $K=\cO\left(\epsilon^{-2}\right),\ N = \Theta(\log \frac{1}{\epsilon}),$ and $T=1$, then in Algorithm \ref{algo:DSBO} the sample complexity to find an $\epsilon$-stationary point is $\cO(\epsilon^{-2}\log(\frac{1}{\epsilon}))$. 
\end{corollary}
It is worth noting that $T\geq 1$ in Theorem \ref{thm: main} implies, to some extent, that by setting a single timescale, more inner loop iterations will not help improve the convergence result in terms of $K$. This observation partially answers the decentralized version of the question `Will Bilevel Optimizers Benefit from Loops?' mentioned in the title of \citet{ji2022will}. It is interesting to study how setting $T$ dependent on other problem parameters will improve the dependency on problem parameters in the final convergence rate. The hypergradient estimation algorithms (i.e., HIGP oracle and Algorithm \ref{algo: Hypergrad}) provide an additional $\cO(\log \frac{1}{\epsilon})$ factor in the sample complexity, which matches \citet{chen2021closing}. To further remove the log factor, \citet{arbel2021amortized} applies warm start to hypergradient estimation and uses mini-batch method (whose batch sizes are dependent on $\epsilon^{-1}$) to reduce this complexity and eventually obtain $\cO(\epsilon^{-2})$. It would be interesting to study how to apply the warm start strategy to remove the log factor in our complexity bound without using mini-batch method. One restriction of Theorem \ref{thm: main} is that we do not obtain the convergence rate $\cO(\frac{1}{\sqrt{nK}})$, i.e., the linear speedup in terms of the number of the agents. The recent work of \citet{yang2022decentralized} achieves linear speedup. However, some of their assumptions are restrictive (see Section \ref{sec: discussion} for a detailed discussion). Besides, according to Table \ref{table: compare}, our Algorithm is more efficient and preferable when $\min\{p, q\} > n$ since we improve the per-iteration computational and communication complexity from $\max\{pq, q^2\}$ in \citet{yang2022decentralized} to $\max\{p, q\}$. It would be interesting to study how to incorporate Jacobian-computing-free algorithm in DSBO under the mild assumptions without affecting linear speedup.

\subsection{Consequences for Decentralized Stochastic Compositional Optimization}
Note that our algorithm can be used to solve Decentralized Stochastic Compositional Optimization (DSCO) problem:
\begin{equation}\label{eq: dco_opt}
    \min_{x\in\R^p}\quad \Phi(x) = \frac{1}{n}\sum_{i=1}^{n}f_i\left(\frac{1}{n}\sum_{j=1}^{n}g_j(x)\right),
\end{equation}
which can be written in a bilevel formulation:
\begin{equation}\label{eq: dco_opt_bi}
    \begin{aligned}
        \min_{x\in\R^p}\quad \Phi(x) &= \frac{1}{n}\sum_{i=1}^{n}f_i(y^*(x)) \\
        \text{s.t.}\quad y^*(x) &= \argmin_{y\in\R^{q}}\frac{1}{n}\sum_{i=1}^{n}\left(\frac{1}{2}y\T y - g_i(x)\T y\right),
    \end{aligned}
\end{equation}
To solve DSCO, \citet{zhao2022numerical} proposes D-ASCGD and its compressed version. Both of them have $\cO(\epsilon^{-2})$ sample complexity. However, their algorithm requires stronger assumptions (see Assumption 1 (a) in \citet{zhao2022numerical}) and needs to compute full Jacobians (i.e., $\nabla g_i(x;\xi)$), which lead to $\cO(pq\epsilon^{-2})$ computational complexity. By using our Algorithm \ref{algo:DSBO}, we can obtain $\tilde\cO(\max(p, q)\epsilon^{-2})$ computational complexity, which is preferable in high dimensional problems. We state the result formally in the corollary below; the proof is immediate. 
\begin{corollary}\label{cor: DSCO}
    Suppose the conditions of Theorem \ref{thm: main} hold. For any $\epsilon>0$, if we set $K=\cO\left(\epsilon^{-2}\right),\ N = \Theta(\log \frac{1}{\epsilon}),$ and $T=1$, then the sample complexity of using Algorithm \ref{algo:DSBO} to find an $\epsilon$-stationary point of Problem \eqref{eq: dco_opt_bi} is $\cO(\epsilon^{-2}\log(\frac{1}{\epsilon}))$, and the computational complexity is $\tilde\cO(\max(p, q)\epsilon^{-2})$.
\end{corollary}

\subsection{Proof sketch}\label{sec: proof_sketch}
In this section we briefly introduce a sketch of our proof for Theorem \ref{thm: main} as well as the ideas of the algorithm design. Throughout our analysis, we define the filtration as 
\[
    \F_k = \sigma\left(\bigcup_{i=1}^{n}\{y_{i,0}^{(T)}, ..., y_{i,k}^{(T)}, x_{i,0},...,x_{i,k}, r_{i,0},...,r_{i,k}\}\right).
\]

\subsubsection{Moving average method}\label{sec: MA}
The moving average method used in line 12 of Algorithm \ref{algo:DSBO} serves as a key step in setting up the convergence analysis framework. We focus on estimating
\[
    \frac{1}{K}\sum_{k=0}^{K}\E\left[\|\bar{r}_k\|^2 + \|\bar{r}_k - \nabla\Phi(\bar{x}_k)\|^2\right],
\]
which provides another optimality measure for finding the $\epsilon$-stationary point since we know
\[
    \E\left[\|\bar{r}_k\|^2 + \|\bar{r}_k - \nabla\Phi(\bar{x}_k)\|^2\right]\geq \frac{1}{2}\E\left[\|\nabla\Phi(\bar{x}_k)\|^2\right].
\]
It can then be shown that by appropriately choosing parameters (see Lemma \ref{lem: r_bar_sum} and \ref{lem: r_and_grad} for details), we obtain
\begin{align*}
    &\frac{1}{K}\sum_{k=0}^{K}\E\left[\|\bar{r}_k\|^2 + \|\bar{r}_k - \nabla\Phi(\bar{x}_k)\|^2\right] \\
    = &\cO\left(\frac{1}{\sqrt{K}} + \frac{1}{K}\sum_{k=0}^{K}\E\left[\|\E\left[\bar{u}_k|\F_k\right] - \nabla\Phi(\bar{x}_k)\|^2\right]\right),
\end{align*}
which implies that it suffices to bound the hypergradient estimation error, namely, the second term on the right hand side of the above equality. The moving average technique reduces the bias in the hypergradient estimate so that we can directly bound $\E\left[\|\E\left[\bar{u}_k|\F_k\right] - \nabla\Phi(\bar{x}_k)\|^2\right]$ instead of $\E\left[\|\bar{u}_k - \nabla\Phi(\bar{x}_k)\|^2\right]$, and the former one makes use of the linear convergence property of the gradient tracking methods, which is elaborated in the next section.

\subsubsection{Convergence of HIGP}
Define
\begin{align*}
    y_k^* &= y^*(\bar{x}_k), \\
    z_{*}^{(k)} &= \left(\sum_{i=1}^{n}\nabla_y^2g_i(\bar{x}_k,y_k^*)\right)^{-1}\left(\sum_{i=1}^{n}\nabla_y f_i(\bar{x}_k, y_k^*)\right).
\end{align*}
To bound the hypergradient estimation error, a rough analysis (see Lemma \ref{lem: last}) shows that $\E\left[\|\E\left[\bar{u}_k|\F_k\right] - \nabla\Phi(\bar{x}_k)\|^2\right] =$
\[
    \begin{aligned}
        &\resizebox{\hsize}{!}{$\cO\Big(\E\left[\|X_k - \bar{x}_k\bfonet\|^2 + \|Y_k^{(T)} - \bar{y}_k^{(T)}\bfonet\|^2 + \|\bar{y}_k^{(T)} - y_k^*\|^2\right]$} \\ 
        &\resizebox{\hsize}{!}{$+ \E\left[\|\E\left[z_{i,N}^{(k)} - \bar{z}_{N}^{(k)}|\F_k\right]\|^2 + \|\E\left[\bar{z}_N^{(k)}|\F_k\right] - z_*^{(k)}\|^2\right]\Big),$}
    \end{aligned}
\]
where the first two terms on the right hand side denote the consensus error among agents, and can be bounded via techniques in decentralized optimization (Lemma \ref{lem: XY_cons}). The third term represents the inner loop estimation error, which can be bounded by considering its decrease as $k$ increases (Lemma \ref{lem: y_error}). Our novelty lies in bounding the last two terms -- the consensus and convergence analysis of the HIGP oracle. Observe that
by setting 
\[
    \resizebox{\hsize}{!}{$\dot z_{i,t}^{(k)} = \E\left[z_{i,t}^{(k)}|\F_k\right], \dot d_{j,t}^{(k)} =\E\left[d_{j,t}^{(k)}|\F_k\right], \dot s_{i,t}^{(k)} = \E\left[s_{i,t}^{(k)}|\F_k\right],$}
\]
we know from Algorithm \ref{algo: HIGP_oracle}
\begin{align*}
    &\dot z_{i,t+1}^{(k)} = \sum_{j=1}^{n}w_{ij}\dot z_{j,t}^{(k)} - \gamma \dot d_{i,t}^{(k)},\  Z_0^{(k)} = 0,\\
    &\dot d_{i,t+1}^{(k)} =\sum_{i=1}^{n}w_{ij}\dot d_{j,t}^{(k)} + \dot s_{i,t+1}^{(k)} - \dot s_{i,t}^{(k)}, \\
    &\dot s_{i,t}^{(k)} = \nabla_y^2g_i(x_{i,k},y_{i,k}^{(T)})\dot z_{i,t}^{(k)}- \nabla_yf_i(x_{i,k},y_{i,k}^{(T)}),
\end{align*}
which is exactly a deterministic gradient descent scheme with gradient tracking on a strongly convex and smooth quadratic function. Hence the linear convergence results in gradient tracking methods can be applied, and this also explains why $\gamma$ can be chosen as a constant that is independent of $K$. Mathematically, in Lemmas \ref{lem: higp_main} and \ref{lem: last} we explicitly characterize the error and eventually obtain the final convergence result in Theorem \ref{thm: main}.

\begin{figure*}[ht]
	\centering  
	\subfigure[]{\label{figure:syn_1.0}\includegraphics[width=0.3\textwidth]{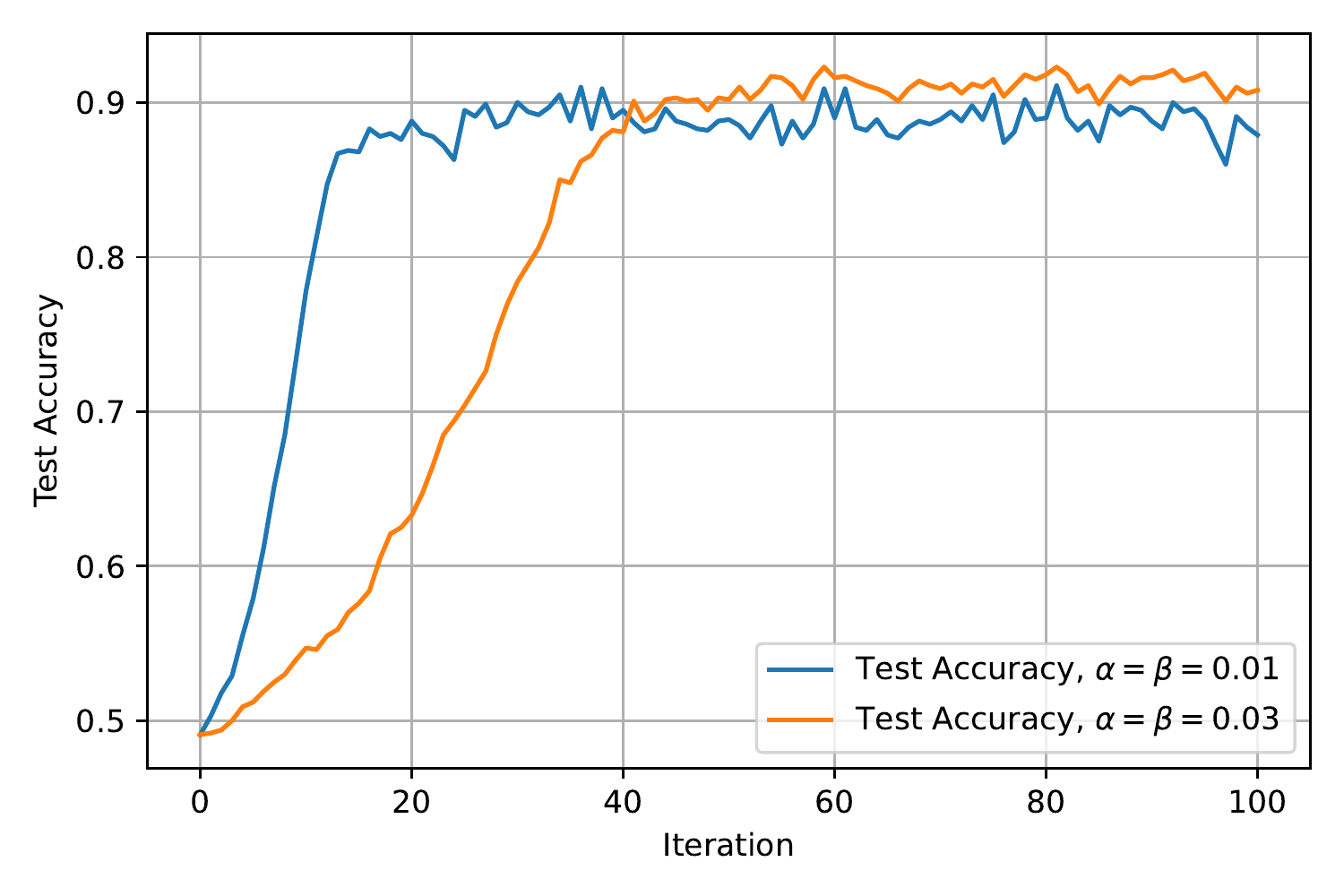}}
        \subfigure[]{\label{figure:comp_time}\includegraphics[width=0.3\textwidth]{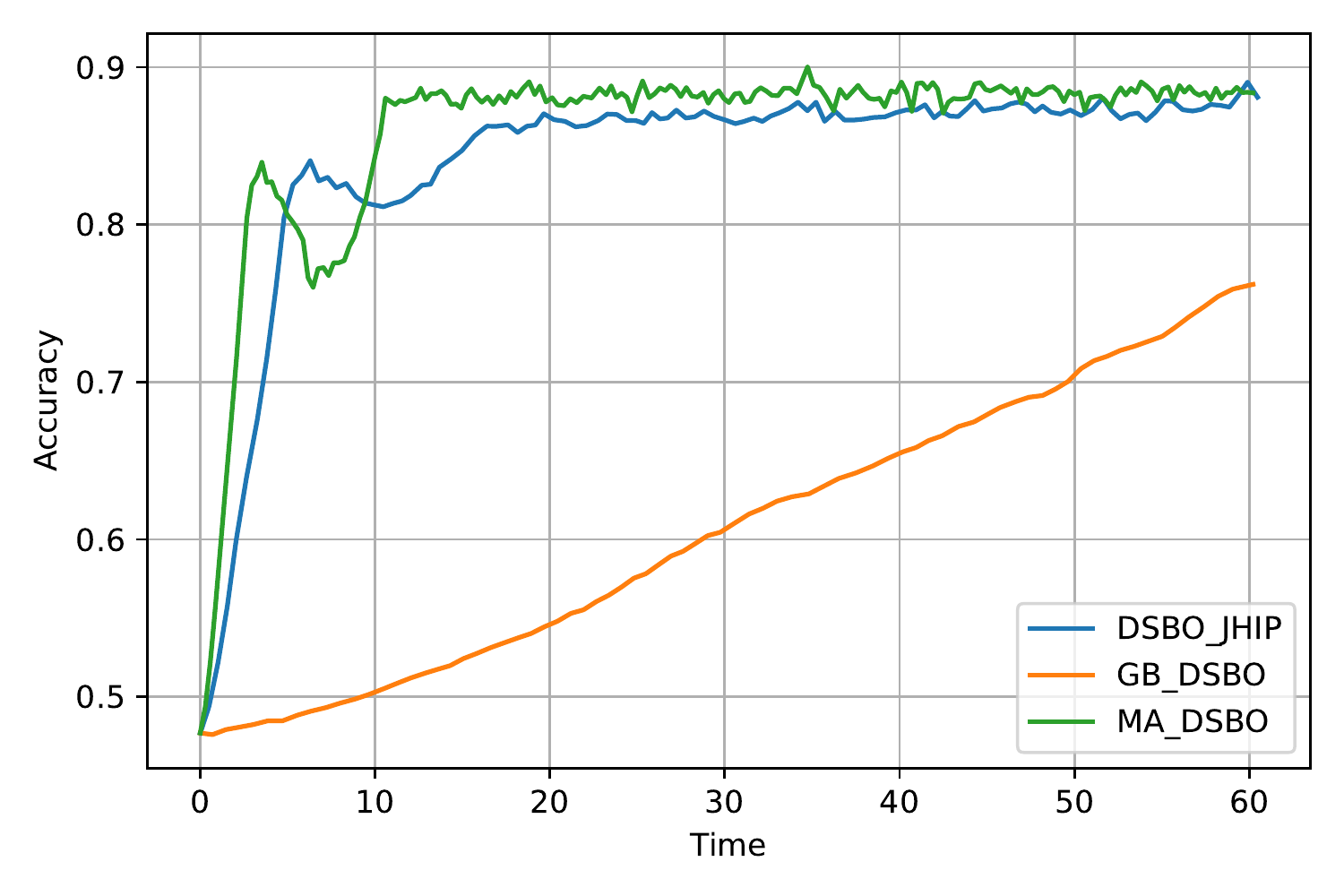}}
        \subfigure[]{\label{figure:comp_cc}\includegraphics[width=0.3\textwidth]{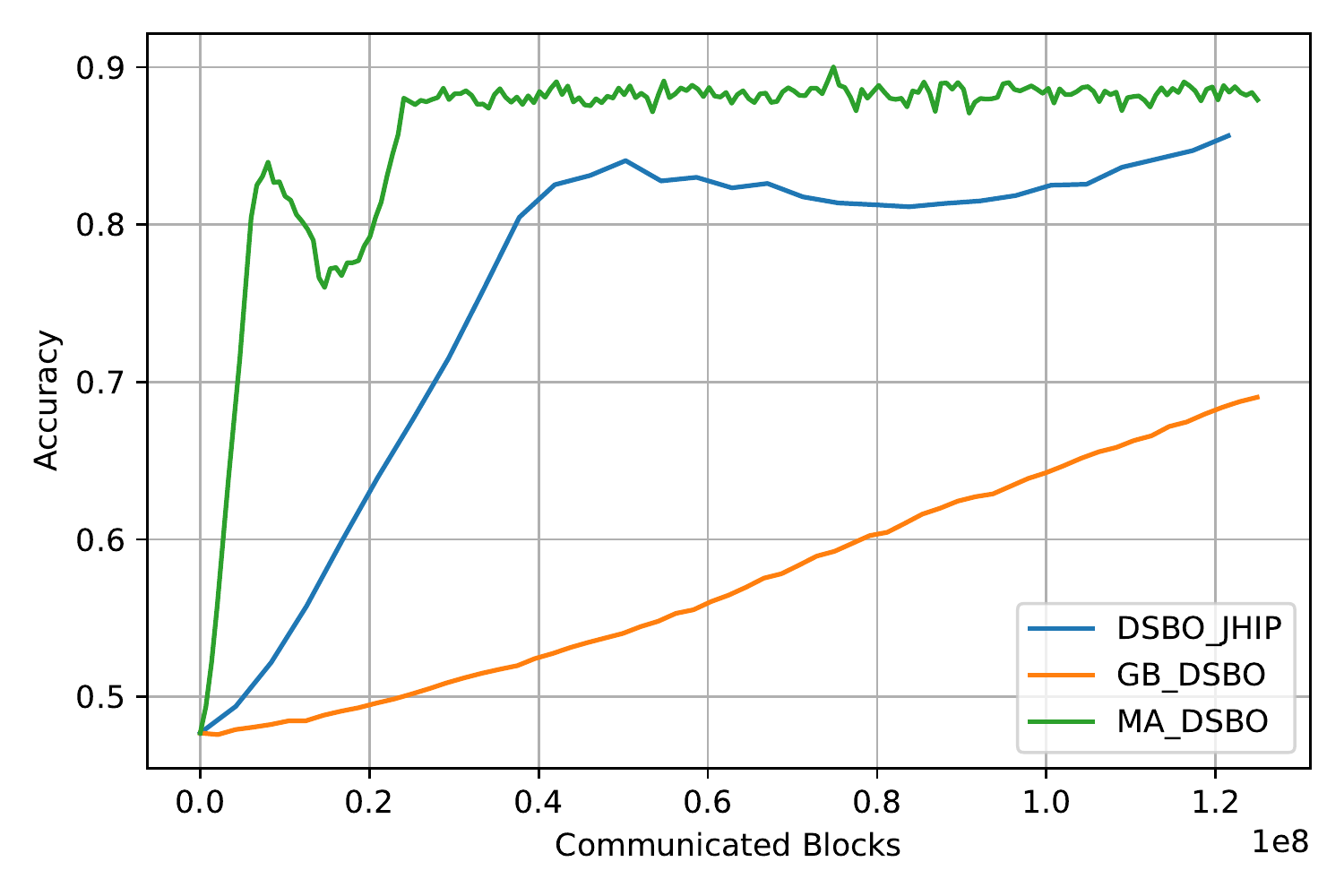}}
	\vspace{-0.2cm}
	\caption{$\ell^2$-regularized logistic regression on synthetic data.}\label{Figure: syn_acc}
	\vspace{-0.3cm}
\end{figure*}

\begin{figure*}[ht]
	\centering  
	\subfigure[]{\label{figure:mnist_acc}\includegraphics[width=0.3\textwidth]{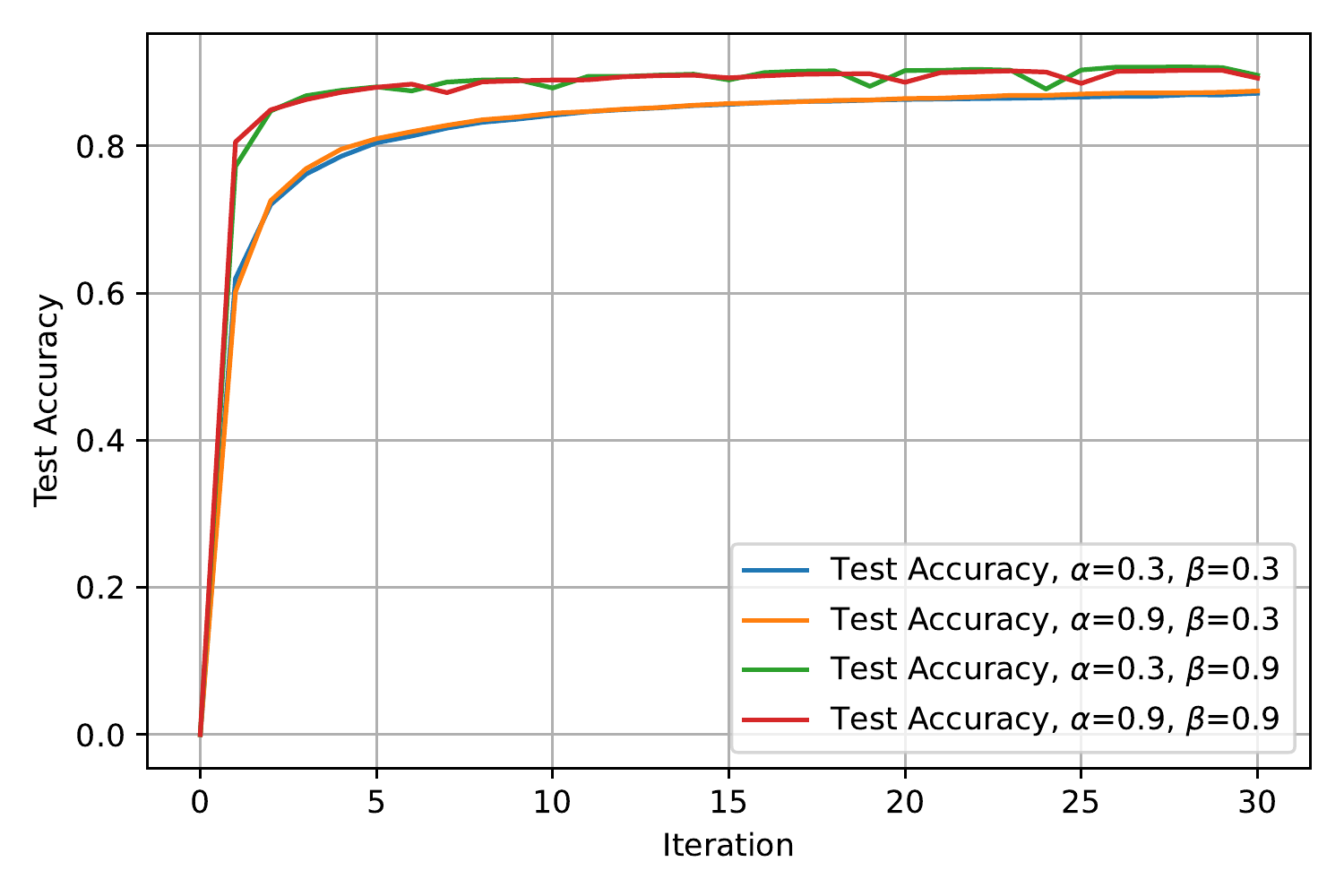} }
	\subfigure[]{\label{fig:mnist_train}\includegraphics[width=0.3\textwidth]{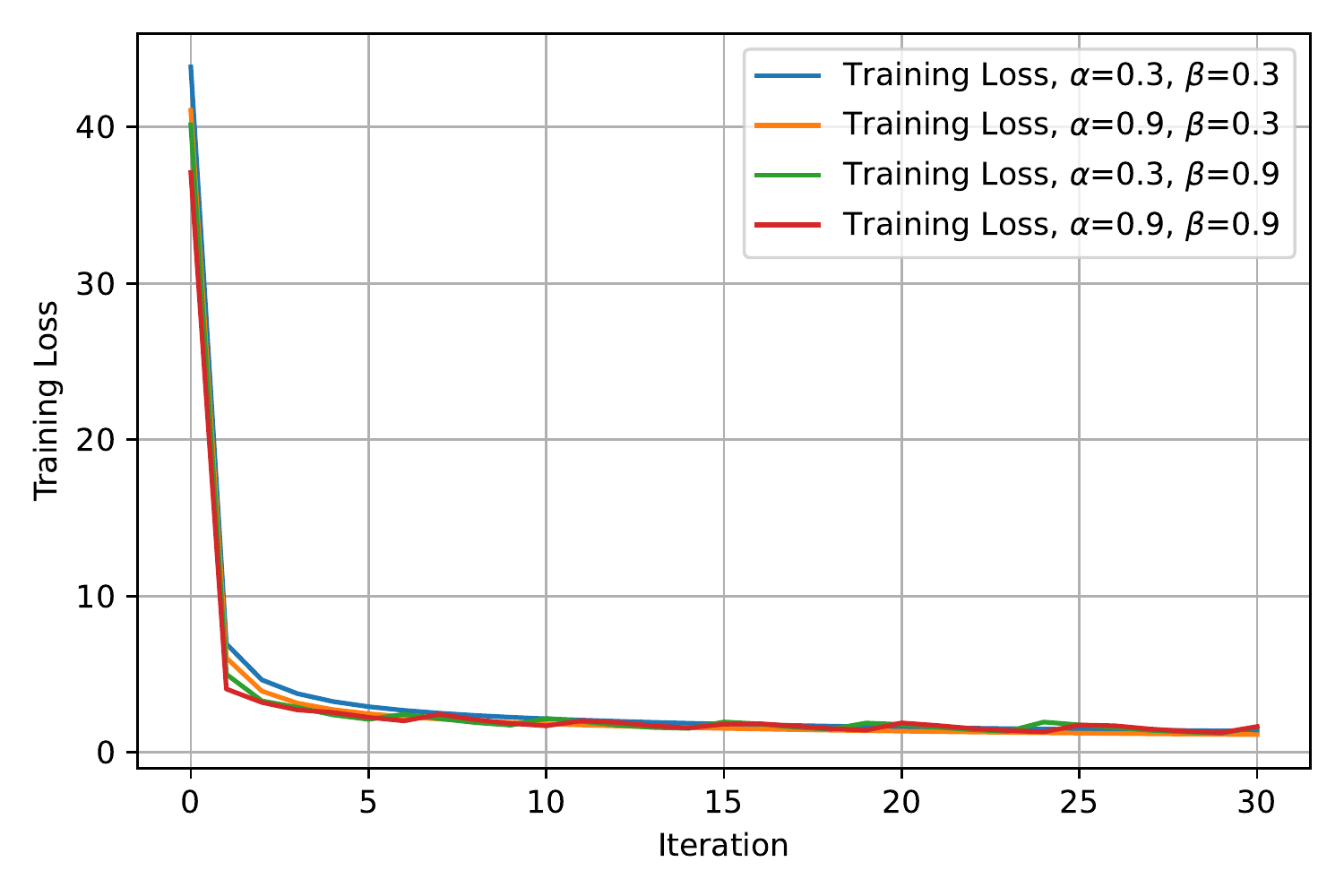}}
	\subfigure[]{\label{fig:mnist_test}\includegraphics[width=0.3\textwidth]{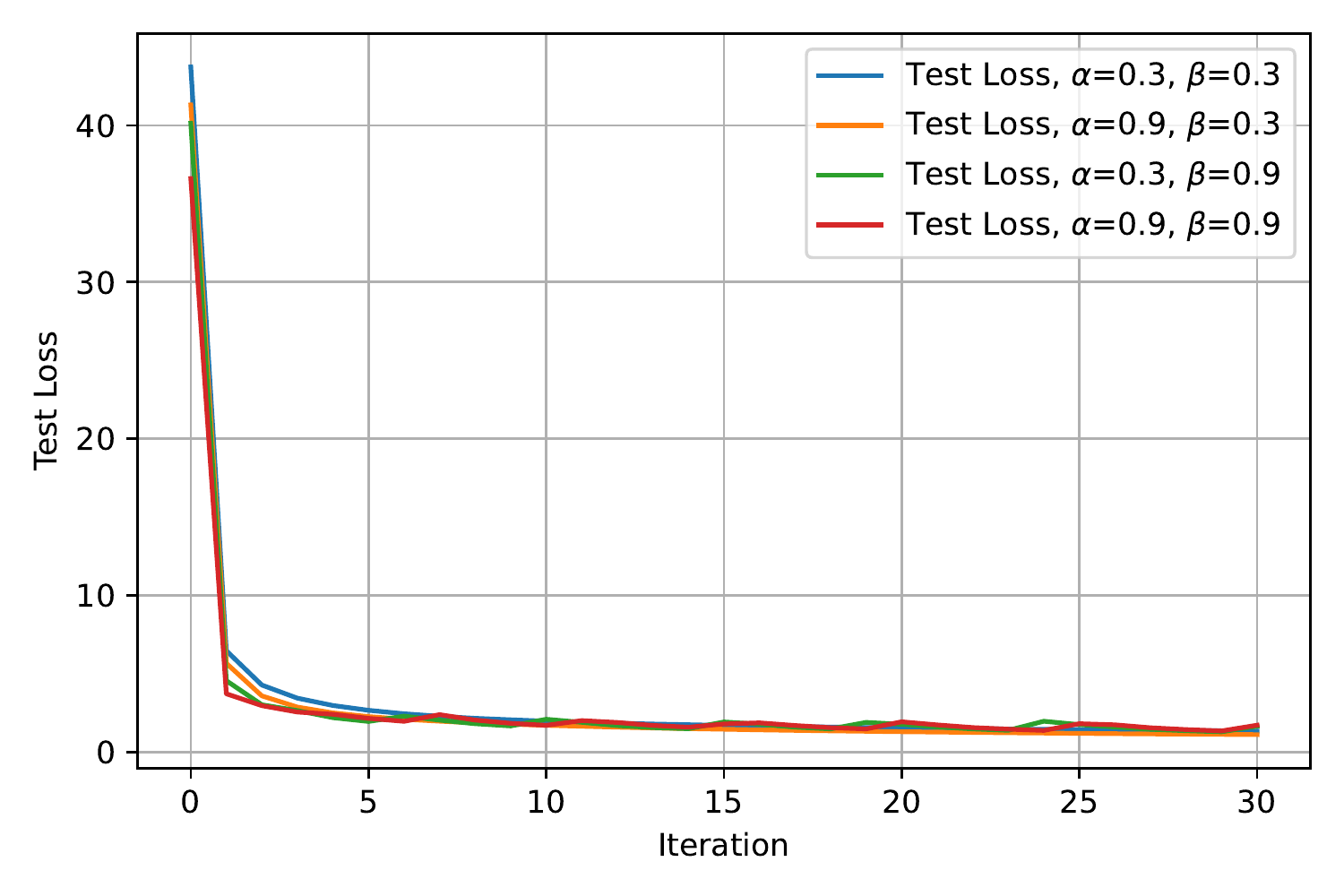}}
	\vspace{-0.2cm}
	\caption{$\ell^2$-regularized logistic regression on MNIST. }\label{Figure: MNIST}
	\vspace{-0.3cm}
\end{figure*}

\section{Numerical experiments}\label{sec: exp}
In this section we study the applications of Algorithm \ref{algo:DSBO} on hyperparameter optimization:
\begin{align*}
    &\min_{\lambda\in\R^p}\quad  \frac{1}{n}\sum_{i=1}^{n}f_i(\lambda,\omega^*(\lambda)), \\
    &\text{s.t.}\quad \omega^*(\lambda) =\argmin_{w\in\R^q} \frac{1}{n}\sum_{i=1}^{n}g_i(\lambda,\omega),
\end{align*}
where we aim at finding the optimal hyperparameter $\lambda$ under the constraint that $\omega^*(\lambda)$ is the optimal model parameter given $\lambda$. We consider both the synthetic and real world data. Comparing to hypergradient estimation algorithms in \citet{chen2022decentralized} and \citet{yang2022decentralized}, our HIGP oracle (Algorithm \ref{algo: HIGP_oracle}) reduces both the per-iteration complexity and storage from $\cO(q^2)$ to $\cO(q)$. All the experiments are performed on a local device with $8$ cores ($n = 8$) using mpi4py \citep{dalcin2021mpi4py} for parallel computing and PyTorch \citep{NEURIPS2019_9015} for computing stochastic oracles. The network topology is set to be the ring topology with the weight matrix $W = (w_{ij})$ given by
\[
    w_{ii} = w,\ w_{i,i+1} = w_{i, i-1} = \frac{1-w}{2},\ \text{for some } w\in (0,1).
\]
Here $w_{1,0} = w_{1,n}$ and $w_{n, n+1} = w_{n, 1}$. In other words, the neighbors of agent $i$ only include $i-1$ and $i+1$ for $i=1,2,...,n$ with $0$ and $n+1$ representing $n$ and $1$ respectively.
\subsection{Heterogeneous and normally distributed data}\label{sec: syn_exp}
Following \citet{pedregosa2016hyperparameter, grazzi2020iteration, chen2022decentralized}, $f_i$ and $g_i$ are defined as:

\begin{align*}
    &f_i(\lambda, \omega) = \sum_{(x_e,y_e)\in\cD_i'}\psi(y_ex_e^{\top}\omega), \\
    &g_i(\lambda, \omega) = \sum_{(x_e,y_e)\in\cD_i}\psi(y_ex_e^{\top}\omega) + \frac{1}{2}\sum_{i=1}^{p}e^{\lambda_i}\omega_i^2,
\end{align*}
where $\psi(x) = \log(1 + e^{-x})$ and $p=200$ denotes the dimension parameter. A ground truth vector $w^*$ is generated in the beginning, and each $x_e\in\R^p$ is generated according to the normal distribution. The data distribution of $x_e$ on node $i$ is $\mathcal{N}(0, i^2)$. Then we set $y_e = x_e\T w + \varepsilon\cdot z$, where $\varepsilon = 0.1$ denotes the noise rate and $z\in\R^p$ is the noise vector sampled from standard normal distribution. The task is to learn the optimal regularization parameter $\lambda\in\R^p$. We also compare our Algorithm \ref{algo:DSBO} with GBDSBO \citep{yang2022decentralized} and DSBO-JHIP \citep{chen2022decentralized} under this setting with dimension parameter $p=100$. Figures \ref{figure:syn_1.0}, \ref{figure:comp_time} and \ref{figure:comp_cc}\footnote{The word "block" is a term used in tracemalloc module in Python (see \url{https://docs.python.org/3/library/tracemalloc.html}) to measure the memory usage, and we keep track of the number of the communicated blocks between different agents as a direct measure for communication cost.} demonstrate the efficiency of our algorithm in both time and space complexity. Due to space limit, we include our additional experiments in Section \ref{sec: add_exp}.
\subsection{MNIST}
Now we consider hyperparameter optimization on MNIST dataset \citep{lecun1998gradient}. Following \citet{grazzi2020iteration}, we have
\[
    \begin{aligned}
        &f_i(\lambda, \omega)= \frac{1}{|\cD_i'|}\sum_{(x_e,y_e)\in\cD_i' }L(x_e^{\top}\omega, y_e), \\
        &g_i(\lambda, \omega)= \frac{1}{|D_i|}\sum_{(x_e,y_e)\in \cD_i}L(x_e^{\top}\omega, y_e) + \frac{1}{cp}\sum_{i=1}^{c}\sum_{j=1}^{p}e^{\lambda_j}\omega_{ij}^2,
    \end{aligned}
\]
where $c=10,\ p = 784$ denote the number of classes and the number of features, $\omega \in\R^{c\times p}$ is the model parameter, and $L$ denotes the cross entropy loss. $\cD_i$ and $\cD_i'$ denote the training and validation set respectively. The batch size is $1000$ in each stochastic oracle. We include the numerical results of different stepsize choices in Figure \ref{Figure: MNIST}. Note that in previous algorithms \citep{chen2022decentralized, yang2022decentralized} one Hessian matrix of the lower level function requires $\cO(c^2p^2)$ storage, while in our algorithm a Hessian-vector product only requires $\cO(cp)$ storage, which improves both the space and the communication complexity. The accuracy and the loss curves indicate that our MA-DSBO Algorithm \ref{algo:DSBO} has a considerably good performance on real world dataset. Note that this problem has larger dimension, and the other algorithms took more time so we do not do the comparison.
\section{Conclusion}
In this paper, we propose a DSBO algorithm that does not require computing full Hessian and Jacobian matrices, thereby improving the per-iteration complexity of currently known DSBO algorithms, under mild assumptions. Moreover, we prove that our algorithm achieves $\tilde{\cO}(\epsilon^{-2})$ sample complexity, which matches the result in state-of-the-art single-agent bilevel optimization algorithms. We would like to point out that Assumption \ref{assump: similarity_of_g} (or bounded second moment condition in \citet{yang2022decentralized}) requires certain types of upper bounds on $\|\nabla_yg(x,y)\|$, which may not hold in decentralized optimization (see, e.g., \citet{pu2021distributed}). It is interesting to study decentralized stochastic bilevel optimization without this type of conditions, and one promising direction is to apply variance reduction techniques like in \citet{tang2018d}. It is also interesting to incorporate Hessian-free methods \citep{sowconvergence} in DSBO, and we leave it as future work.
\subsection*{Acknowledgments}
XC acknowledges the support by UC Davis Dean's Graduate Summer Fellowship. Research of SM was supported in part by National Science Foundation (NSF) grants DMS-2243650, CCF-2308597, UC Davis CeDAR (Center for Data Science and Artificial Intelligence Research) Innovative Data Science Seed Funding Program, and a startup fund from Rice University. KB acknowledges the support by National Science Foundation (NSF) via the grant NSF DMS-2053918.

\bibliography{bibfile}
\bibliographystyle{icml2023}

\newpage
\appendix
\onecolumn

\begin{center}
    \noindent\rule{\textwidth}{4pt} \vspace{-0.2cm}
    
    \LARGE \textbf{Appendix} 
    
    \noindent\rule{\textwidth}{1.2pt}
\end{center}
\section{Additional experiments on heterogeneous data}\label{sec: add_exp}
To introduce heterogeneity, we set $r$ as the heterogeneity rate, and the data distribution of $x_e$ in Section \ref{sec: syn_exp} on node $i$ is $\mathcal{N}(0, i^2\cdot r^2)$. In Figure \ref{fig:syn_0.5}, \ref{fig:syn_1.0} and \ref{fig:syn_1.5} (and similarly for \ref{fig:syn_loss_0.5}, \ref{fig:syn_loss_1.0}, and \ref{fig:syn_loss_1.5}) we set $r$ as $0.5, 1.0,$ and $1.5$ respectively. The accuracy and loss results demonstrate that our algorithm works well under different heterogeneity rates.
\begin{figure*}[ht]
	\centering  
	\subfigure[]{\label{fig:syn_0.5}\includegraphics[width=0.3\textwidth]{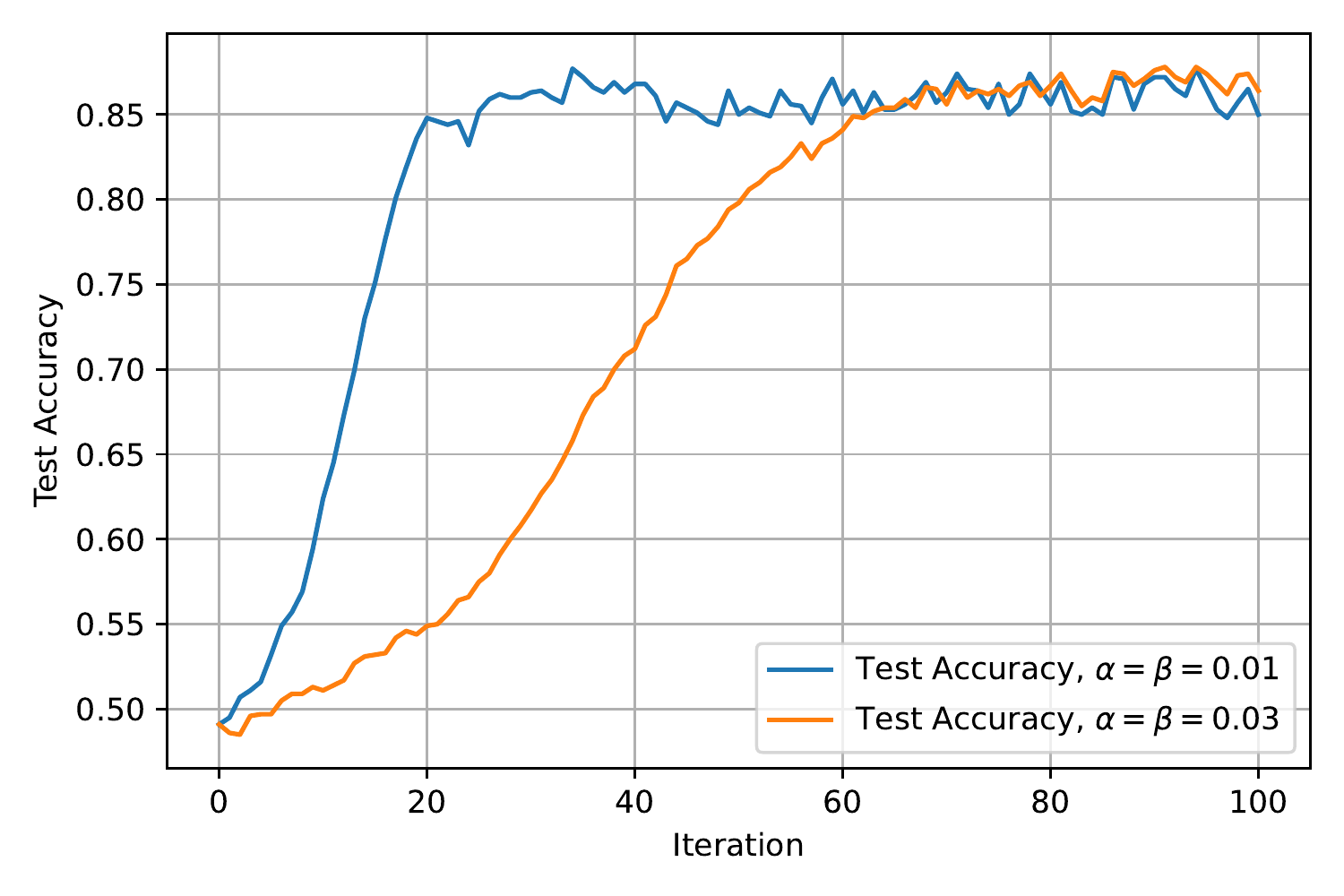}}
	\subfigure[]{\label{fig:syn_1.0}\includegraphics[width=0.3\textwidth]{figures/synthetic_acc_h_rate_1.0.pdf}}
	\subfigure[]{\label{fig:syn_1.5}\includegraphics[width=0.3\textwidth]{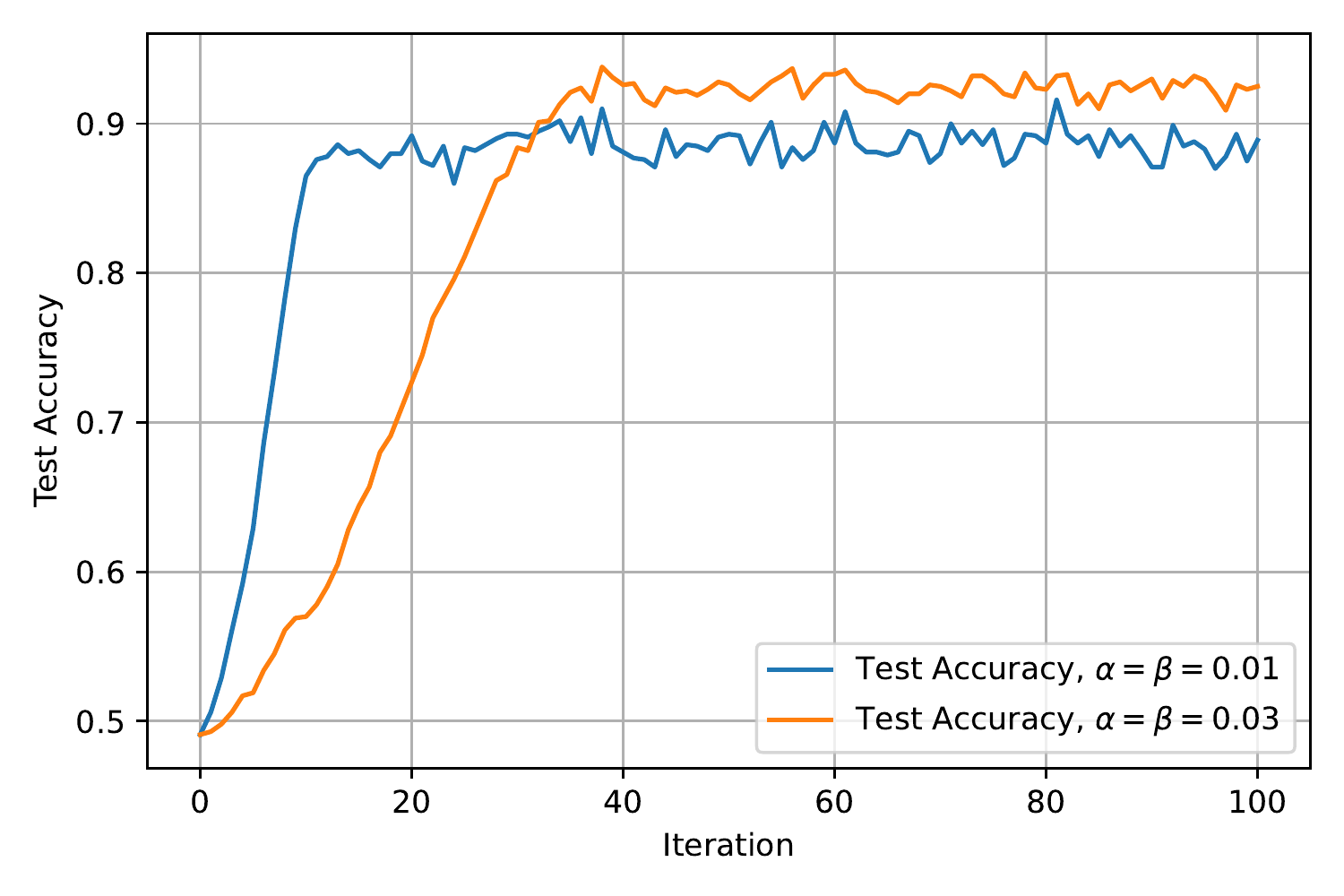}}
	\subfigure[]{\label{fig:syn_loss_0.5}\includegraphics[width=0.3\textwidth]{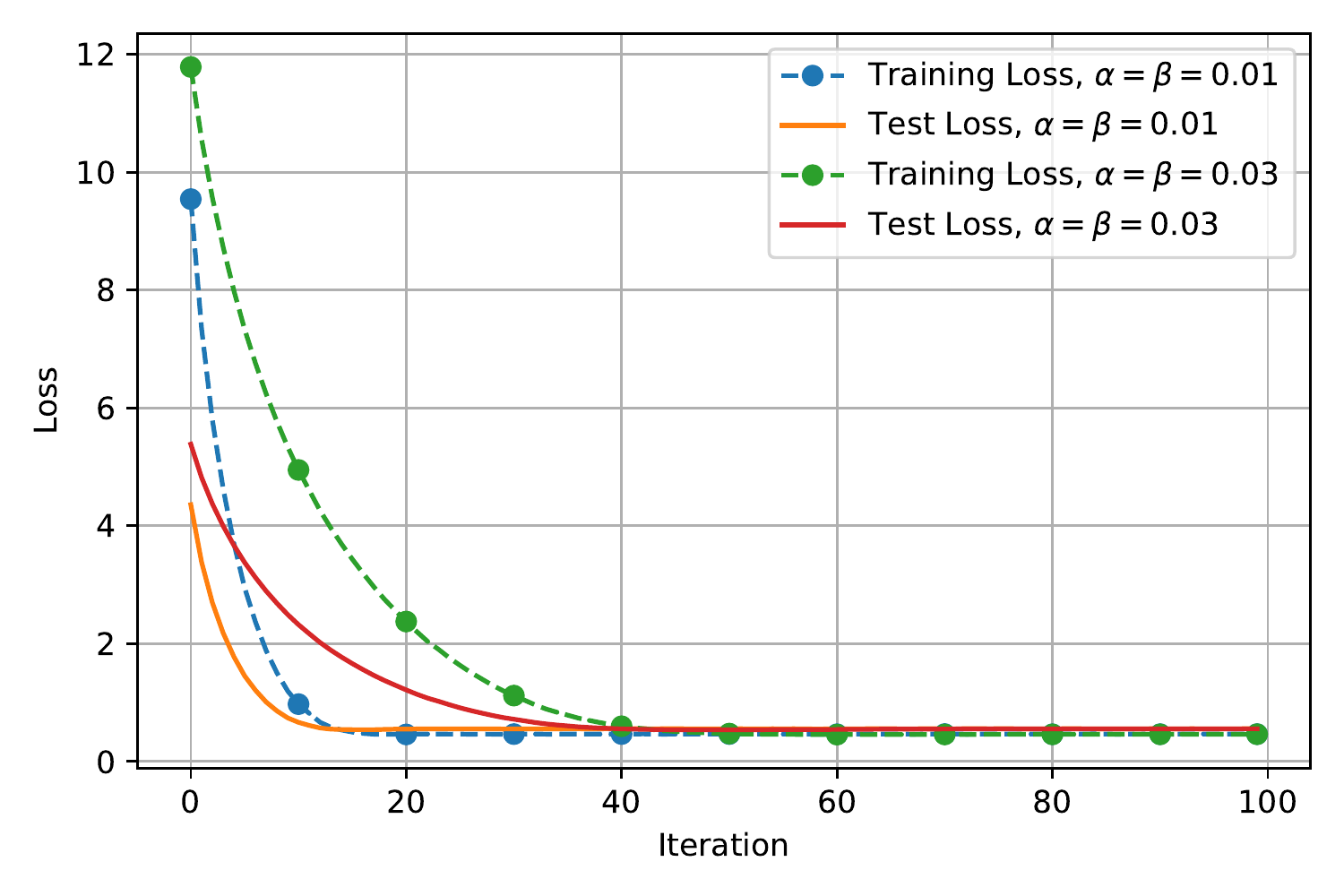} }
	\subfigure[]{\label{fig:syn_loss_1.0}\includegraphics[width=0.3\textwidth]{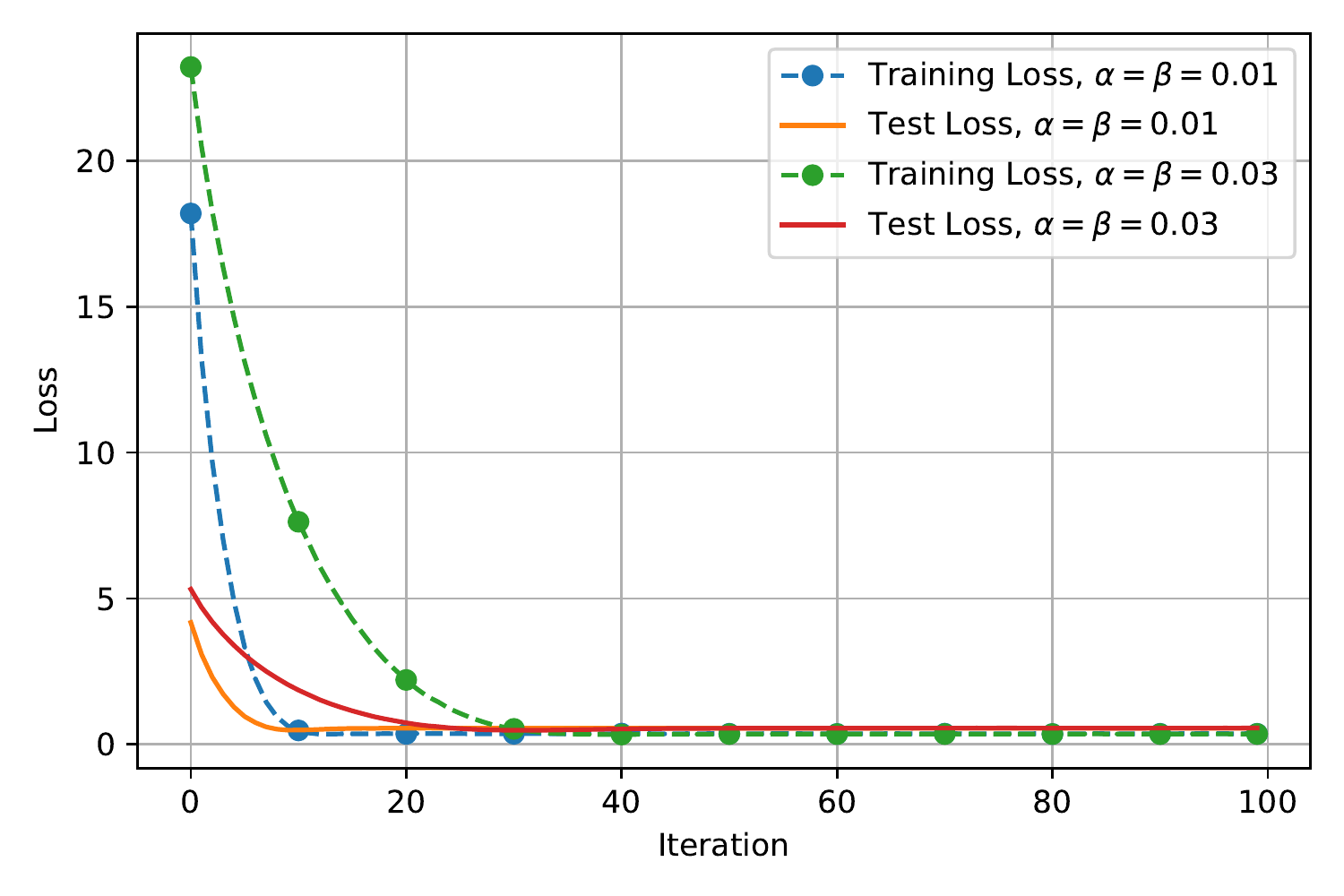}}
	\subfigure[]{\label{fig:syn_loss_1.5}\includegraphics[width=0.3\textwidth]{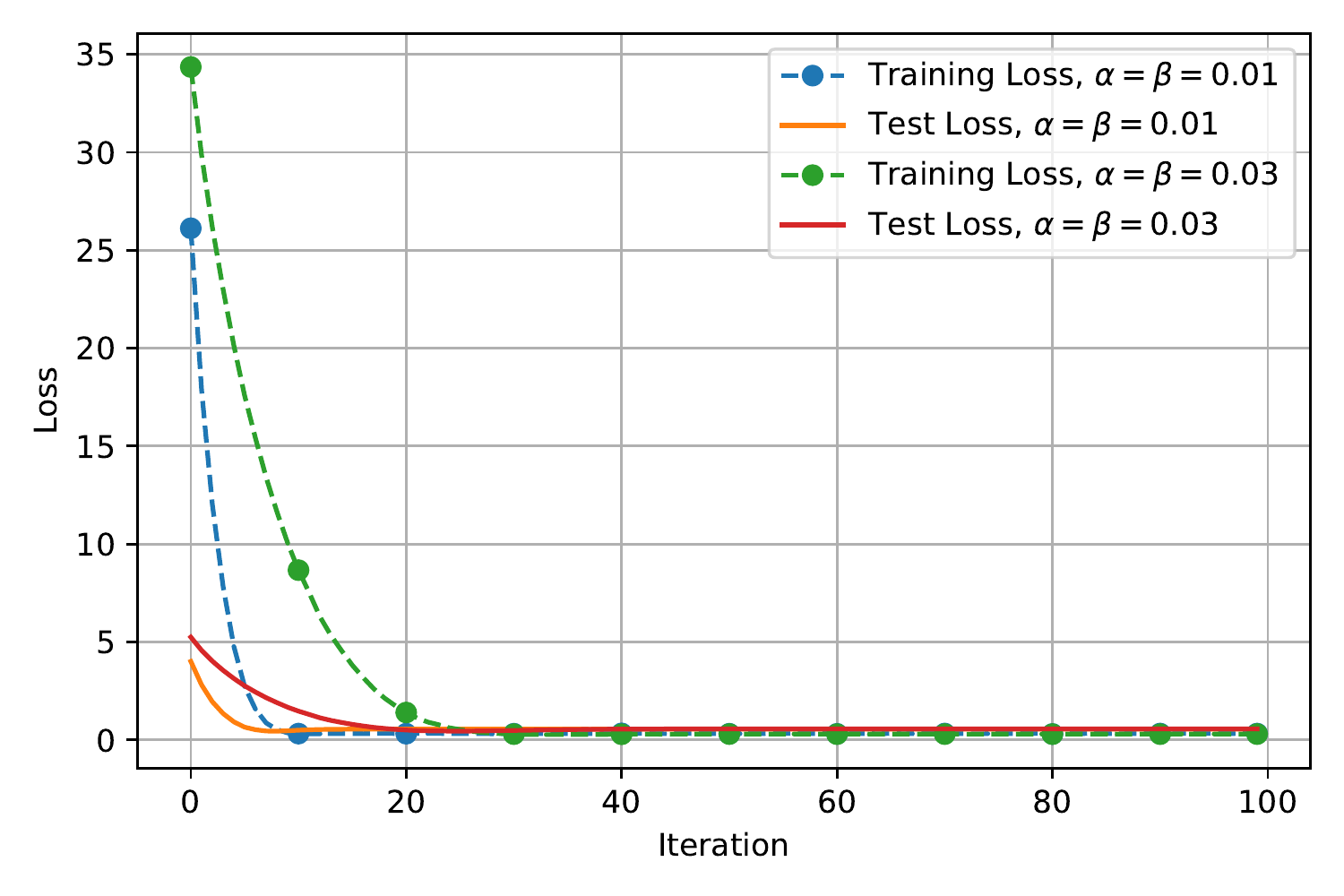}}
	\vspace{-0.2cm}
	\caption{$\ell^2$-regularized logistic regression on synthetic data.}\label{Fig: syn_acc}
	\vspace{-0.3cm}
\end{figure*}

\section{Analysis}\label{sec: analysis}
Figure \ref{fig: proof} represents the structure of the proof. For convenience we restate our notation convention here again:
\begin{itemize}
    \item We use the first subscript (usually denoted as $i$) to represent the agent number, and the second subscript (usually denoted as $k$ or $t$) to represent the iteration number. For example $x_{i,k}$ represents the $x$ variable of agent $i$ at $k$-th iteration. For the inner loop iterate like $y_{i,k}^{(t)}$, the superscript $t$ represents the iteration number of the inner loop.
    
    \item We use uppercase letters to represent the matrix that collecting all the variables (corresponding lowercase) as columns. For example $X_k = \left(x_{1,k}, ..., x_{n,k}\right),\ Y_k^{(t)} = \left(y_{1,k}^{(t)},...,y_{n,k}^{(t)}\right)$.
    
    \item We add an overbar to a letter to denote the average over all nodes. For example, $\bar{x}_k = \frac{1}{n}\sum_{i=1}^{n}x_{i,k},\ \bar{y}_{k}^{(t)} = \frac{1}{n}\sum_{i=1}^{n}y_{i,k}^{(t)}$.
    
    \item The filtration is defined as
    \[
        \F_k = \sigma\left(\bigcup_{i=1}^{n}\{y_{i,0}^{(T)}, ..., y_{i,k}^{(T)}, x_{i,0},...,x_{i,k}, r_{i,0},...,r_{i,k}\} \right).
    \]
\end{itemize}
\begin{figure}
    \centering
    \begin{tikzpicture}[node distance=2cm]
        \node (lem1) [lemma] {Lemma \ref{lem: smoothness}};
        \node (lem2) [lemma, right of=lem1, xshift=2cm] {Lemma \ref{lem: useful_ineq}};
        \node (lem3) [lemma, right of=lem2, xshift=2cm] {Lemma \ref{lem: gd_decrease}};
        \node (lem6) [lemma, right of=lem3, xshift=2cm] {Lemma \ref{ineq: stepsize_for_higp}};

        \node (lem4) [lemma, below of=lem3] {Lemma \ref{lem: dsgd_quadratic}};
        \node (lem7) [lemma, below of=lem4] {Lemma \ref{lem: XY_cons}};
        \node (lem5) [lemma, right of=lem7,xshift=2cm] {Lemma \ref{lem: dsgd_higp}};
        \node (lem9) [lemma, below of=lem7] {Lemma \ref{lem: higp_main}};
        
        \node (lem11) [lemma, below of=lem1] {Lemma \ref{lem: r_bar_sum}};
        \node (lem8) [lemma, right of=lem11, xshift=2cm] {Lemma \ref{lem: y_error}};
        
        \node (lem10) [lemma, below of=lem11] {Lemma \ref{lem: all_const_sum}};
        \node (lem12) [lemma, below of=lem10] {Lemma \ref{lem: r_and_grad}};
        
        \node (lem13) [lemma, right of=lem10, xshift=2cm] {Lemma \ref{lem: last}};
        
        \node (main) [lemma, below of=lem13] {\textbf{Theorem \ref{thm: main}}};
        
        \draw [arrow] (lem3) -- (lem4);
        \draw [arrow] (lem4) -- (lem5);
        \draw [arrow] (lem5) -- (lem7);
        \draw [arrow] (lem6) -- (lem4);
        \draw [arrow] (lem6) -- (lem5);
        \draw [arrow] (lem5) -- (lem9);
        \draw [arrow] (lem7) -- (lem9);
        \draw [arrow] (lem1) -- (lem8);
        \draw [arrow] (lem2) -- (lem8);
        \draw [arrow] (lem3) -- (lem8);
        \draw [arrow] (lem1) -- (lem11);
        \draw [arrow] (lem11) -- (lem13);
        \draw [arrow] (lem10) -- (lem13);
        \draw [arrow] (lem7) -- (lem13);
        \draw [arrow] (lem9) -- (lem13);
        \draw [arrow] (lem13) -- (main);
        \draw [arrow] (lem8) -- (lem13);
        \draw [arrow] (lem12) -- (lem13);
    \end{tikzpicture}
    \caption{Structure of the proof}
    \label{fig: proof}
\end{figure}
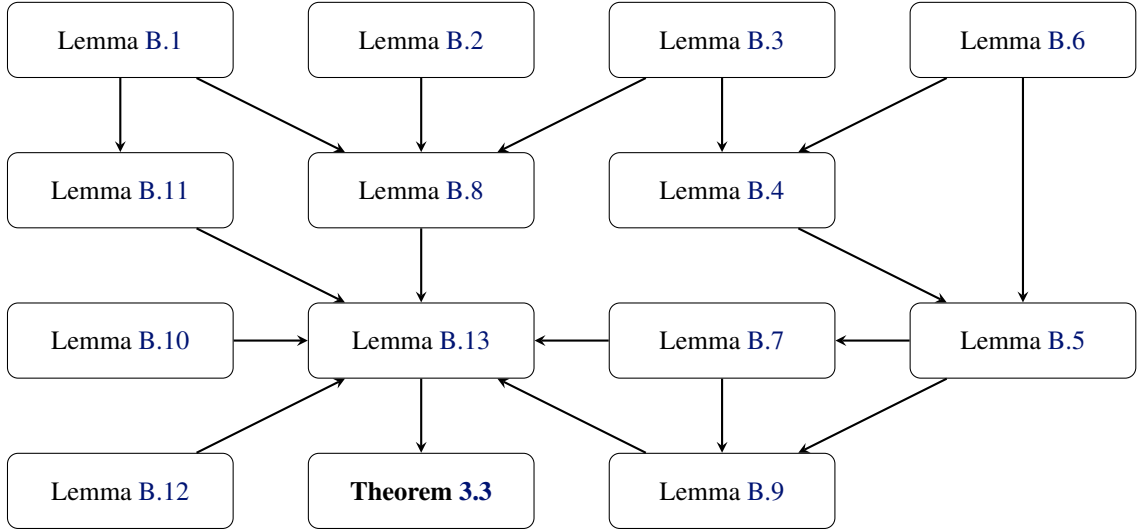

We first state several well-known results in bilevel optimization literature (see, e.g., Lemma 2.2 in \citet{ghadimi2018approximation}.).
\begin{lemma}\label{lem: smoothness}
    Suppose Assumptions \ref{assump: lip_convexity} and \ref{assump: stoc_derivatives} hold, we know $\nabla\Phi(x)$ and $y^*(x)$ defined in \eqref{eq: dsbo_opt} are $L_{\Phi}$ and $L_{y^*}$-Lipschitz continuous respectively with the constants given by
    \begin{equation}\label{eq: smoothness_const}
        L_{\Phi} = L_{f,1} + \frac{2L_{f,1}L_{g,1} + L_{g,2}L_{f,0}^2 }{\mu_g} + \frac{2L_{g,1}L_{f,0}L_{g,2}+L_{g,1}^2L_{f,1}}{\mu_g^2} + \frac{L_{g,2}L_{g,1}^2L_{f,0}}{\mu_g^3},\ L_{y^*} = \frac{L_{g,1}}{\mu_g}.
    \end{equation}
\end{lemma}

The following inequality is a standard result and will be used in our later analysis. We prove it here for completeness.
\begin{lemma}\label{lem: useful_ineq}
    Suppose we are given two sequences $\{a_k\}$ and $\{b_k\}$ that satisfy
    \[
        a_{k+1}\leq \delta a_k + b_k,\ a_k\geq 0,\ b_k\geq 0 \text{ for all } k\geq 0
    \]
    for some $\delta\in (0,1)$. Then we have
    \[
        a_{k+1} \leq \delta^{k+1}a_0 + \sum_{i=0}^{k}b_i\delta^{k-i}.
    \]
\end{lemma}
\begin{proof}[Proof of Lemma~\ref{lem: useful_ineq}]
    Setting $c_i = \frac{a_i}{\delta^i}$, we know
    \[
        c_{i+1}\leq c_i + b_i\cdot\delta^{-i-1}\text{ for all }i\geq 0.
    \]
    Taking summation on both sides ($i$ from $0$ to $k$) and multiplying $\delta^{k+1}$, we know for $k\geq 0$,
    \[
        a_{k+1}\leq \delta^{k+1}a_0 + \sum_{i=0}^{k}b_i\delta^{k-i},
    \]
    which completes the proof.
\end{proof}
The following lemma is standard in stochastic optimization (see, e.g., Lemma 10 in \citet{qu2017harnessing}).
\begin{lemma}\label{lem: gd_decrease}
    Suppose $f(x)$ is $\mu$-strongly convex and $L-smooth$. For any $x$ and $\eta<\frac{2}{\mu + L}$, define $x^+ = x - \eta\nabla f(x),\ x^*=\argmin f(x)$. Then we have
    \[
        \|x^+ - x^*\| \leq (1-\eta\mu)\|x-x^*\|
    \]
\end{lemma}
Next, we characterize the bounded second moment of the HIGP oracle. Note that Algorithm \ref{algo: HIGP_oracle} is essentially decentralized stochastic gradient descent with gradient tracking on a strongly convex quadratic function.
\begin{lemma}\label{lem: dsgd_quadratic}
    Suppose we are given matrices $A_i$ and vectors $b_i$ such that there exist $0<\mu<L$ such that $\mu I\preceq A_i\preceq LI$ for $1\leq i\leq n$. $W = (w_{ij})$ satisfies Assumption \ref{assump: W}. The sequences $\{x_{i,k}\},\ \{s_{i,k}\}$ and $\{v_{i,k}\}$ satisfy for any $k\geq 0$ and $1\leq i\leq n$,
    \begin{align*}
        &x_{i,k+1} = \sum_{j=1}^{n}w_{ij}x_{j,k} - \alpha s_{i,k},\ s_{i,k+1} = \sum_{j=1}^{n}w_{ij}s_{j,k} +v_{i,k+1} - v_{i,k},\ v_{i,k} = A_{i,k}x_{i,k} - b_{i,k},\ s_{i,0} = v_{i,0}, \\
        &\E\left[A_{i,k}\right] = A_i,\ \E\left[b_{i,k}\right] = b_i,\ \E\left[\|A_{i,k} - A_i\|^2\right]\leq \sigma_1^2,\ \E\left[\|b_{i,k} - b_i\|^2\right]\leq \sigma_2^2.
    \end{align*}
    Moreover, we assume $A_{i,k}, x_{j,k}, b_{i,k}$ are independent for any $i,j\in\left\{1,...,n\right\}$, $\left\{A_{i,k}\right\}_{i=1}^{n}$ are independent and $\left\{b_{i,k}\right\}_{i=1}^{n}$ are independent. Define
    \begin{align*}
        &\tilde{\sigma}_1^2 = \sigma_1^2 + L^2,\ \tilde{\sigma}_2^2 = \sigma_2^2 + \max_i\|b_i\|^2,\ x^* := \left(\frac{1}{n}\sum_{i=1}^{n}A_i\right)^{-1}\left(\frac{1}{n}\sum_{i=1}^{n}b_i\right), \\
        &C_1 = 9\sigma_1^2 + 6\alpha^2\tilde{\sigma}_1^2 + \frac{18\alpha^2\sigma_1^2\tilde{\sigma}_1^2}{n},\ C_2 = 12\tilde{\sigma}_1^2 + 9\sigma_1^2 + 12\alpha^2L^2\tilde{\sigma}_1^2  + \frac{18\alpha^2\sigma_1^2\tilde{\sigma}_1^2}{n}, \\
        &C_3 = 6\rho^2\tilde{\sigma}_1^2,\ C_4 = 2\sigma_2^2 + \frac{6\alpha^2\sigma_2^2\tilde{\sigma}_1^2}{n} + \left(9\sigma_1^2 + \frac{18\alpha^2\sigma_1^2\tilde{\sigma}_1^2}{n}\right)\|x^*\|^2, \\
        &c = \left(\frac{\alpha^2}{n}(3\sigma_1^2\|x^*\|^2 + \sigma_2^2),\ 0,\ \frac{(1+\rho^2)}{1-\rho^2}C_4\right)\T,\ M = 
        \begin{pmatrix}
            M_{11} &M_{12} &0 \\
            0 &M_{22} &M_{23} \\
            M_{31} &M_{32} &M_{33}
        \end{pmatrix}, \\
        & M_{11} = 1-\alpha\mu,\ M_{12} = \left(\frac{2\alpha}{\mu} + 2\alpha^2\right)\tilde{\sigma}_1^2,\ M_{22} = \frac{1+\rho^2}{2},\ M_{23} = \alpha^2\frac{1+\rho^2}{1-\rho^2} \\
        &M_{31} = \frac{1+\rho^2}{1-\rho^2}C_1,\ M_{32} = \frac{1+\rho^2}{1-\rho^2}C_2,\ M_{33} = \frac{1+\rho^2}{2} + \frac{1+\rho^2}{1-\rho^2}C_3\alpha^2.
    \end{align*}
    If $\alpha$ satisfies
    \begin{equation}\label{ineq: stepsize_for_quadratics}
        \begin{aligned}
            &\left(1+\frac{\alpha\mu}{2}\right)(1-\alpha\mu)^2 + \frac{3\alpha^2\sigma_1^2}{n}<1-\alpha\mu,\ 0< \alpha_1\leq \alpha \leq \alpha_2 \text{ for some } 0<\alpha_1<\alpha_2, \\
            &\rho(M) < 1 - \frac{2\alpha\mu}{3},\ \text{ and }~M~\text{has 3 different positive eigenvalues,}  
        \end{aligned}
    \end{equation} 
    then we have
    \begin{align}\label{ineq: dsgd_quadratic}
    \begin{aligned}
            \E\left[\|\bar{x}_{k+1} - x^*\|^2\right] &\leq (1-\alpha\mu)\E\left[\|\bar{x}_k - x^*\|^2\right] + \left(\frac{2\alpha}{\mu} + 2\alpha^2\right)\frac{\tilde{\sigma}_1^2}{n}\E\left[\|X_k - \bar{x}_k\bfonet\|^2\right]\\&\qquad + \frac{\alpha^2}{n}(3\sigma_1^2\|x^*\|^2 + \sigma_2^2), \\
        \|X_{k+1} - \bar{x}_{k+1}\bfonet\|^2 &\leq \frac{(1+\rho^2)}{2}\|X_k - \bar{x}_k\bfonet\|^2 + \alpha^2\frac{1+\rho^2}{1-\rho^2}\|S_k - \bar{s}_k\bfonet\|^2, \\
        \E\left[\frac{\|S_{k+1} - \bar{s}_{k+1}\bfonet\|^2}{n}\right] &\leq  \frac{1+\rho^2}{1-\rho^2}C_1\E\left[\|\bar{x}_k - x^*\|^2\right] + \frac{1+\rho^2}{1-\rho^2}C_2\E\left[\frac{\|X_k - \bar{x}_k\bfonet\|^2}{n}\right]  \\&\qquad \qquad + \left(\frac{1+\rho^2}{2}  + \frac{1+\rho^2}{1-\rho^2}C_3\alpha^2\right)\E\left[\frac{\|S_k-\bar{s}_k\bfonet\|^2}{n}\right] +\frac{1+\rho^2}{1-\rho^2}C_4.
        \end{aligned}
    \end{align}
    Moreover, we set $P$ such that $M = P\cdot\text{diag}(\lambda_1, \lambda_2, \lambda_3)P^{-1}$ with $0<\lambda_3<\lambda_2<\lambda_1$ being eigenvalues and each column of $P$ is a unit vector. Define $C_M:= \|P\|_2\|P^{-1}\|_2$, we have
    \begin{align}
        &\max\left(\frac{1}{n}\E\left[\|X_k - x^*\bfonet\|^2\right], \frac{1}{n}\E\left[\|X_k-\bar{x}_k\bfonet\|^2\right]\right) \notag \\
        \leq &3C_M\left(1-\frac{2\alpha\mu}{3}\right)^k\left(\E\left[\|\bar{x}_0 - x^*\|^2\right] + \E\left[\frac{\|X_0\|^2 + \|S_0\|^2}{n}\right]\right) + \frac{5C_M\|c\|}{\alpha\mu},   \label{ineq: quadratic_conv} \\
        &\frac{1}{n}\E\left[\|X_k\|^2\right]\leq 6C_M\left(1-\frac{2\alpha\mu}{3}\right)^k\left(\E\left[\|\bar{x}_0 - x^*\|^2\right] + \E\left[\frac{\|X_0\|^2 + \|S_0\|^2}{n}\right]\right) + \frac{10C_M\|c\|}{\alpha\mu} + 2\|x^*\|^2. \label{ineq: bdd_2_m}
    \end{align}
\end{lemma}
\begin{proof}[Proof of Lemma~\ref{lem: dsgd_quadratic}]
    Note that by definition of $\tilde{\sigma}_1^2$ and $\tilde{\sigma}_2^2$ we have
    \begin{equation}\label{ineq: quadratic_bdd_2_m}
        \begin{aligned}
            &\E\left[\|A_{i,k}\|^2\right] = \E\left[\|A_{i,k}-A_i\|^2\right] + \|A_i\|_2^2\leq \sigma_1^2 + L^2 = \tilde{\sigma}_1^2,\\
            &\E\left[\|b_{i,k}\|^2\right] = \E\left[\|b_{i,k}-b_i\|^2\right] + \|b_i\|^2 \leq \sigma_2^2 + \max_i\|b_i\|^2 = \tilde{\sigma}_2^2.
        \end{aligned}
    \end{equation}
    By $s_{i,0} = v_{i,0}$ we know $\bar{s}_0 = \bar{v}_0$. From the recursion we know
    \[
        \bar{s}_{k+1} = \bar{s}_k + \bar{v}_{k+1} - \bar{v}_k,
    \]
    and hence $\bar{s}_k = \bar{v}_k$ by induction. For $\bar{x}_k$ we know
    \begin{align*}
        &\bar{x}_{k+1} - x^* \\
        = &\bar{x}_k - x^* - \frac{\alpha}{n}\sum_{i=1}^{n}(A_{i,k}x_{i,k} - b_{i,k}) \\
        = &\bar{x}_k - x^* - \frac{\alpha}{n}\sum_{i=1}^{n}(A_i\bar{x}_k - b_i) + \frac{\alpha}{n}\sum_{i=1}^{n}(A_i\bar{x}_k - b_i) - \frac{\alpha}{n}\sum_{i=1}^{n}(A_{i,k}x_{i,k} - b_{i,k}) \\
        = &\left(I - \frac{\alpha}{n}\sum_{i=1}^{n}A_i\right)(\bar{x}_k - x^*) + \frac{\alpha}{n}\sum_{i=1}^{n}A_{i,k}(\bar{x}_k - x_{i,k}) + \frac{\alpha}{n}\sum_{i=1}^{n}((A_i - A_{i,k})\bar{x}_k + b_{i,k} - b_i).
    \end{align*}
    Using the above equality, $\E\left[A_{i,k}\right] = A_i$ and $\E\left[b_{i,k}\right] = b_i$, we know
    \begin{align*}
        &\E\left[\|\bar{x}_{k+1} - x^*\|^2\right] \\
        = &\E\left[\|\left(I - \frac{\alpha}{n}\sum_{i=1}^{n}A_i\right)(\bar{x}_k - x^*) + \frac{\alpha}{n}\sum_{i=1}^{n}A_{i,k}(\bar{x}_k - x_{i,k})\|^2\right] + \frac{\alpha^2}{n^2}\E\left[\|\sum_{i=1}^{n}((A_i - A_{i,k})\bar{x}_k + b_{i,k} - b_i)\|^2\right] \\
        + &\E\left[\<\left(I - \frac{\alpha}{n}\sum_{i=1}^{n}A_i\right)(\bar{x}_k - x^*) + \frac{\alpha}{n}\sum_{i=1}^{n}A_{i,k}(\bar{x}_k - x_{i,k}), \frac{\alpha}{n}\sum_{i=1}^{n}((A_i - A_{i,k})\bar{x}_k + b_{i,k} - b_i)> \right] \\
        \leq &\left(1+\frac{\alpha\mu}{2}\right)(1-\alpha\mu)^2\E\left[\|\bar{x}_k - x^*\|^2\right] + \left(1+\frac{2}{\alpha\mu}\right)\frac{\alpha^2\tilde{\sigma}_1^2}{n}\sum_{i=1}^{n}\E\left[\|\bar{x}_k - x_{i,k}\|^2\right] + \frac{\alpha^2}{n^2}(n\sigma_1^2\E\left[\|\bar{x}_k\|^2\right] + n\sigma_2^2) \\
        + & \frac{\alpha^2}{2n^2}\sum_{i=1}^{n}\E\left[\sigma_1^2\|\bar{x}_k\|^2 + \tilde{\sigma}_1^2\|\bar{x}_k - x_{i,k}\|^2\right] \\
        = &\left(1+\frac{\alpha\mu}{2}\right)(1-\alpha\mu)^2\E\left[\|\bar{x}_k - x^*\|^2\right] + \left(\frac{2\alpha}{\mu} + \alpha^2 + \frac{\alpha^2}{2n}\right)\frac{\tilde{\sigma}_1^2}{n}\E\left[\|X_k - \bar{x}_k\bfonet\|^2\right] + \frac{\alpha^2}{n}\left(\frac{3\sigma_1^2}{2}\E\left[\|\bar{x}_k\|^2\right] + \sigma_2^2\right) \\
        \leq &\left[\left(1+\frac{\alpha\mu}{2}\right)(1-\alpha\mu)^2 + \frac{3\alpha^2\sigma_1^2}{n}\right]\E\left[\|\bar{x}_k - x^*\|^2\right] + \left(\frac{2\alpha}{\mu} + 2\alpha^2\right)\frac{\tilde{\sigma}_1^2}{n}\E\left[\|X_k - \bar{x}_k\bfonet\|^2\right] + \frac{\alpha^2}{n}(3\sigma_1^2\|x^*\|^2 + \sigma_2^2) \\
        \leq & (1-\alpha\mu)\E\left[\|\bar{x}_k - x^*\|^2\right] + \left(\frac{2\alpha}{\mu} + 2\alpha^2\right)\frac{\tilde{\sigma}_1^2}{n}\E\left[\|X_k - \bar{x}_k\bfonet\|^2\right] + \frac{\alpha^2}{n}(3\sigma_1^2\|x^*\|^2 + \sigma_2^2).
    \end{align*}
    The first inequality holds because we have
    \begin{align*}
        &\E\left[\<\left(I - \frac{\alpha}{n}\sum_{i=1}^{n}A_i\right)(\bar{x}_k - x^*) + \frac{\alpha}{n}\sum_{i=1}^{n}A_{i,k}(\bar{x}_k - x_{i,k}), \frac{\alpha}{n}\sum_{i=1}^{n}((A_i - A_{i,k})\bar{x}_k + b_{i,k} - b_i)> \right] \\
        = &\E\left[\<\frac{\alpha}{n}\sum_{i=1}^{n}A_{i,k}(\bar{x}_k - x_{i,k}), \frac{\alpha}{n}\sum_{i=1}^{n}((A_i - A_{i,k})\bar{x}_k + b_{i,k} - b_i)> \right] \\
        =& \E\left[\< \frac{\alpha}{n}\sum_{i=1}^{n}A_{i,k}(\bar{x}_k - x_{i,k}), \frac{\alpha}{n}\sum_{i=1}^{n}(A_i - A_{i,k})\bar{x}_k> \right] \\
        = & \frac{\alpha^2}{n^2}\sum_{i=1}^{n}\E\left[(\bar{x}_k - x_{i,k})\T A_{i,k}\T(A_i-A_{i,k})\bar{x}_k\right] \leq \frac{\alpha^2}{2n^2}\sum_{i=1}^{n}\E\left[\sigma_1^2\|\bar{x}_k\|^2 + \tilde{\sigma}_1^2\|\bar{x}_k - x_{i,k}\|^2\right],
    \end{align*}
    the second inequality uses $\|\bar{x}_k\|^2\leq 2\|\bar{x}_k - x^*\|^2 + 2\|x^*\|^2$, and the third inequality uses \eqref{ineq: stepsize_for_quadratics}. For $\|X_{k+1} - \bar{x}_{k+1}\bfonet\|^2$ we know
    \begin{equation}\label{ineq: consensus_ineq}
        \begin{aligned}
            &\|X_{k+1} - \bar{x}_{k+1}\bfonet\|^2 = \|X_kW - \bar{x}_k\bfonet - \alpha (S_k - \bar{s}_k\bfonet)\|^2\\
        \leq &\left(1 + \frac{1-\rho^2}{2\rho^2}\right)\rho^2\|X_k-\bar{x}_k\bfonet\|^2 + \left(1 + \frac{2\rho^2}{1-\rho^2}\right)\alpha^2\|S_k-\bar{s}_k\bfonet\|^2.
        \end{aligned}
    \end{equation}
    The inequality uses Cauchy-Schwarz inequality and the fact that 
    \begin{align*}
        &\|X_kW - \bar{x}_k\bfonet\| = \|\left(X_k - \bar{x}_k\bfonet\right)\left(W - \frac{\bfone\bfonet}{n}\right)\| = \|\left(W - \frac{\bfone\bfonet}{n}\right)\left(X_k - \bar{x}_k\bfonet\right)\T\|\\
        \leq &\|W - \frac{\bfone\bfonet}{n}\|_2\|X_k - \bar{x}_k\bfonet\|\leq \rho\|X_k - \bar{x}_k\bfonet\|,
    \end{align*}
    where the last inequality uses Assumption \ref{assump: W}. For $\|S_{k} - \bar{s}_{k}\bfonet\|^2$ we know
    \begin{equation}\label{ineq: S_cons_onestep}
        \begin{aligned}
            & \|S_{k+1} - \bar{s}_{k+1}\bfonet\|^2 = \|S_kW - \bar{s}_k\bfonet + V_{k+1} - V_k - \bar{v}_{k+1}\bfonet + \bar{v}_k\bfonet\|^2 \\
        \leq &\left(1 + \frac{1-\rho^2}{2\rho^2}\right)\|S_k - \bar{s}_k\bfonet\|^2 + \left(1 + \frac{2\rho^2}{1-\rho^2}\right)\|\left(V_{k+1}-V_k\right)\left(I - \frac{\bfone\bfonet}{n}\right)\|^2 \\
        = &\frac{1+\rho^2}{2}\|S_k - \bar{s}_k\bfonet\|^2 + \frac{1+\rho^2}{1-\rho^2}\|V_{k+1} - V_k\|^2.
        \end{aligned}
    \end{equation}
    For $V_{k+1} - V_k$ we have
    \begin{align*}
        &\E\left[\|V_{k+1} - V_k\|^2\right] \\
        = &\sum_{i=1}^{n}\E\left[\|A_{i,k+1}(x_{i,k+1} - x_{i,k}) + (A_{i,k+1} - A_i + A_i - A_{i,k})x_{i,k} +(b_{i,k} - b_i + b_i - b_{i,k+1})\|^2\right] \\
        = &\sum_{i=1}^{n}\E\left[\|A_{i,k+1}(x_{i,k+1} - x_{i,k}) + (A_{i,k+1} - A_i)x_{i,k}\|^2 + \|(A_i - A_{i,k})x_{i,k}\|^2 + \|b_{i,k} - b_i\|^2 + \|b_i - b_{i,k+1}\|^2\right] \\
        \leq &\sum_{i=1}^{n}\E\left[2\|A_{i,k+1}(x_{i,k+1} - x_{i,k})\|^2 + 2\|(A_{i,k+1} - A_i)x_{i,k}\|^2 + \|(A_i - A_{i,k})x_{i,k}\|^2 + \|b_{i,k} - b_i\|^2 + \|b_i - b_{i,k+1}\|^2\right] \\
        \leq &2\tilde{\sigma}_1^2\E\left[\|X_{k+1}-X_k\|^2\right] + 3\sigma_1^2\E\left[\|X_k\|^2\right] + 2n\sigma_2^2.
    \end{align*}
    For $\|X_{k+1} - X_k\|$ we know
    \begin{align*}
        &\E\left[\|X_{k+1} - X_k\|^2\right] = \E\left[\|X_kW - X_k - \alpha S_kW\|^2\right]\\
        = &\E\left[\|\left(X_k - \bar{x}_k\bfonet\right)(W - I) - \alpha (S_kW - \bar{s}_k\bfonet) - \alpha\bar{s}_k\bfonet\|^2\right] \\
        \leq &3\|W - I\|_2^2\E\left[\|X_k-\bar{x}_k\bfonet\|^2\right] + 3\alpha^2\rho^2\E\left[\|S_k-\bar{s}_k\bfonet\|^2\right] + 3n\alpha^2\E\left[\|\bar{s}_k\|^2\right] \\
        \leq &6\E\left[\|X_k-\bar{x}_k\bfonet\|^2\right] + 3\alpha^2\rho^2\E\left[\|S_k-\bar{s}_k\bfonet\|^2\right] \\&\qquad\qquad + 3\alpha^2(\frac{\sigma_1^2}{n}\E\left[\|X_k\|^2\right] + \sigma_2^2 + 2L^2\E\left[\|X_k-\bar{x}_k\bfonet\|^2 + n\|\bar{x}_k-x^*\|^2\right]) \\
        = & (6 + 6\alpha^2L^2)\E\left[\|X_k-\bar{x}_k\bfonet\|^2\right] + 3\alpha^2\rho^2\E\left[\|S_k-\bar{s}_k\bfonet\|^2\right] + \frac{3\alpha^2\sigma_1^2}{n}\E\left[\|X_k\|^2\right] \\
        &\qquad \qquad+ 3n\alpha^2\E\left[\|\bar{x}_k-x^*\|^2\right] + 3\alpha^2\sigma_2^2,
    \end{align*}
    where the second inequality holds since
    \begin{align*}
        &\E\left[\|\bar{s}_k\|^2\right] = \E\left[\|\frac{1}{n}\sum_{i=1}^{n}(A_{i,k}x_{i,k} - b_{i,k})\|^2\right] \\
        = &\E\left[\|\frac{1}{n}\sum_{i=1}^{n}((A_{i,k} - A_i)x_{i,k} - (b_{i,k} - b_i)) + \frac{1}{n}\sum_{i=1}^{n}(A_ix_{i,k} - A_i\bar{x}_k) + \frac{1}{n}\sum_{i=1}^{n}A_i(\bar{x}_k - x^*)\|^2\right] \\
        = &\frac{1}{n^2}\sum_{i=1}^{n}\E\left[\|(A_{i,k} - A_i)x_{i,k}\|^2 + \|b_{i,k}-b_i\|^2\right] + \frac{1}{n^2}\E\left[\|\sum_{i=1}^{n}((A_ix_{i,k} - A_i\bar{x}_k) + A_i(\bar{x}_k - x^*))\|^2\right]\\
        \leq & \frac{\sigma_1^2}{n^2}\E\left[\|X_k\|^2\right] + \frac{\sigma_2^2}{n} + \frac{2L^2}{n}\E\left[\|X_k-\bar{x}_k\bfonet\|^2 + n\|\bar{x}_k-x^*\|^2\right].
    \end{align*}
    Hence we know
    \begin{align*}
        &\E\left[\|V_{k+1} - V_k\|^2\right] \leq 2\tilde{\sigma}_1^2\E\left[\|X_{k+1}-X_k\|^2\right] + 3\sigma_1^2\E\left[\|X_k\|^2\right] + 2n\sigma_2^2\\
        \leq &2\tilde{\sigma}_1^2\left\{(6 + 6\alpha^2L^2)\E\left[\|X_k-\bar{x}_k\bfonet\|^2\right] + 3\alpha^2\rho^2\E\left[\|S_k-\bar{s}_k\bfonet\|^2\right] + 3n\alpha^2\E\left[\|\bar{x}_k-x^*\|^2\right]\right\} \\
        + &\left(3\sigma_1^2 + \frac{6\alpha^2\sigma_1^2\tilde{\sigma}_1^2}{n}\right)\E\left[\|X_k\|^2\right] + (2n\sigma_2^2 + 6\alpha^2\sigma_2^2\tilde{\sigma}_1^2) \\
        \leq &nC_1\E\left[\|\bar{x}_k - x^*\|^2\right] + C_2\E\left[\|X_k-\bar{x}_k\bfonet\|^2\right] + \alpha^2C_3\E\left[\|S_k-\bar{s}_k\bfonet\|^2\right] + nC_4,
    \end{align*}
    where the second inequality uses 
    \begin{align*}
        \|X_k\|^2\leq 3\left[\|X_k - \bar{x}_k\bfonet\|^2 + n\|\bar{x}_k - x^*\|^2 + n\|x^*\|^2\right].
    \end{align*}
    The above inequalities and \eqref{ineq: S_cons_onestep} imply
    \begin{equation}
        \begin{aligned}
            &\frac{1}{n}\E\left[\|S_{k+1} - \bar{s}_{k+1}\bfonet\|^2\right]\\
            \leq &\frac{1+\rho^2}{2n}\|S_k - \bar{s}_k\bfonet\|^2 + \frac{1+\rho^2}{1-\rho^2}
            \left(C_1\E\left[\|\bar{x}_k - x^*\|^2\right] + C_2\E\left[\frac{\|X_k-\bar{x}_k\bfonet\|^2}{n}\right] + \alpha^2C_3\E\left[\frac{\|S_k-\bar{s}_k\bfonet\|^2}{n}\right] + C_4\right) \\
            \leq &\frac{1+\rho^2}{1-\rho^2}C_1\E\left[\|\bar{x}_k - x^*\|^2\right] + \frac{1+\rho^2}{1-\rho^2}C_2\E\left[\frac{\|X_k - \bar{x}_k\bfonet\|^2}{n}\right] + \left(\frac{1+\rho^2}{2} + \frac{1+\rho^2}{1-\rho^2}C_3\alpha^2\right)\E\left[\frac{\|S_k-\bar{s}_k\bfonet\|^2}{n}\right] +\frac{1+\rho^2}{1-\rho^2}C_4.
        \end{aligned}
    \end{equation}
    Now if we define
    \begin{align*}
        &\Gamma_k = \left(\E\left[\|\bar{x}_k - x^*\|^2\right],\ \E\left[\frac{\|X_k-\bar{x}_k\bfonet\|^2}{n}\right],\ \E\left[\frac{\|S_k-\bar{s}_k\bfonet\|^2}{n}\right]\right)\T,\ \\
        &c = \left(\frac{\alpha^2}{n}(3\sigma_1^2\|x^*\|^2 + \sigma_2^2),\ 0,\ \frac{(1+\rho^2)}{1-\rho^2}C_4\right)\T,\ M = 
        \begin{pmatrix}
            M_{11} &M_{12} &0 \\
            0 &M_{22} &M_{23} \\
            M_{31} &M_{32} &M_{33}
        \end{pmatrix}, \\
        & M_{11} = 1-\alpha\mu,\ M_{12} = \left(\frac{2\alpha}{\mu} + 2\alpha^2\right)\tilde{\sigma}_1^2,\ M_{22} = \frac{1+\rho^2}{2},\ M_{23} = \alpha^2\frac{1+\rho^2}{1-\rho^2} \\
        &M_{31} = \frac{1+\rho^2}{1-\rho^2}C_1,\ M_{32} = \frac{1+\rho^2}{1-\rho^2}C_2,\ M_{33} = \frac{1+\rho^2}{2} + \frac{1+\rho^2}{1-\rho^2}C_3\alpha^2,
    \end{align*}
    then by \eqref{ineq: dsgd_quadratic} we know $\Gamma_{i+1}\leq M\Gamma_i + c$ for any $i$, and thus 
    \[
        \Gamma_{k+1}\leq M\Gamma_k + c\leq...\leq M^{k+1}\Gamma_0 +\sum_{i=0}^{k}M^ic,
    \]
    where all the inequalities are element-wise. By \eqref{ineq: stepsize_for_quadratics} we know there exists an invertible matrix $P\in \R^{3\times 3}$ such that $M = P\cdot\text{diag}(\lambda_1, \lambda_2, \lambda_3)P^{-1}$, and $0<\lambda_3<\lambda_2<\lambda_1<1-\frac{2\alpha\mu}{3}$. Without loss of generality we may assume each column of $P$ (as an eigenvector) is a unit vector. Hence we know
    \begin{equation}\label{ineq: Mk}
        \|M^k\|_2 = \|P\cdot \text{diag}(\lambda_1^k, \lambda_2^k,\lambda_3^k)P^{-1}\|_2\leq \left(1-\frac{2\alpha\mu}{3}\right)^k\|P\|_2\|P^{-1}\|_2 = C_M\cdot \left(1-\frac{2\alpha\mu}{3}\right)^k,
    \end{equation}
    where we define $C_M := \|P\|_2\|P^{-1}\|_2$ in the last equality. Note that since we choose $P$ such that each column is a unit vector and $M = P\cdot\text{diag}(\lambda_1, \lambda_2, \lambda_3)P^{-1}$, $P$ is uniquely determined and $C_M$ is a continuous function of $\alpha$ and other constants $(\sigma_1,\ \sigma_2,\ \mu,\ L,\ \max_{i}\|b_i\|,\ \|x^*\|,\ n,\ \rho)$. On the other hand, observe that 
    \begin{align}\label{ineq: M_sum}
    \begin{aligned}
        \|\sum_{i=0}^{k}M^i\|_2 &= \|\sum_{i=0}^{k}P\cdot \text{diag}(\lambda_1^i, \lambda_2^i, \lambda_3^i)P^{-1}\|_2 = \|P\cdot \text{diag}\left(\sum_{i=0}^{k}\lambda_1^i, \sum_{i=0}^{k}\lambda_2^i, \sum_{i=0}^{k}\lambda_3^i\right)P^{-1}\|_2\\
        &\leq C_M\cdot \max_i \frac{1}{1-\lambda_i} < \frac{3C_M}{2\alpha\mu},
    \end{aligned}
    \end{align}
    where the last inequality uses the upper bound of the eigenvalues. For \eqref{ineq: quadratic_conv} we have
    \begin{align*}
        &\max\left(\frac{1}{n}\E\left[\|X_k - x^*\bfonet\|^2\right], \frac{1}{n}\E\left[\|X_k-\bar{x}_k\bfonet\|^2\right]\right) \leq \frac{2}{n}\E\left[\|X_k -\bar{x}_k\bfonet\|^2 + n\|\bar{x}_k - x^*\|^2\right]\leq 2\sqrt{2}\|\Gamma_k\| \\
        \leq &2\sqrt{2}\|M^k\Gamma_0 + \sum_{i=0}^{k-1}M^ic\|\leq 2\sqrt{2}(\|M^k\|_2\|\Gamma_0\| + \|\sum_{i=1}^{k-1}M^i\|_2\|c\|) \\
        \leq &2\sqrt{2}C_M\left(1-\frac{2\alpha\mu}{3}\right)^k\|\Gamma_0\| + 2\sqrt{2}\cdot\frac{3C_M}{2\alpha\mu}\|c\| \\
        \leq & 2\sqrt{2}C_M\left(1-\frac{2\alpha\mu}{3}\right)^k(\E\left[\|\bar{x}_0 - x^*\|^2\right] + \E\left[\frac{\|X_0-\bar{x}_0\bfonet\|^2}{n}\right] + \E\left[\frac{\|S_0 - \bar{s}_0\bfonet\|^2}{n}\right]) + \frac{3\sqrt{2}C_M\|c\|}{\alpha\mu} \\
        \leq &3C_M\left(1-\frac{2\alpha\mu}{3}\right)^k\left(\E\left[\|\bar{x}_0 - x^*\|^2\right] + \E\left[\frac{\|X_0\|^2 + \|S_0\|^2}{n}\right]\right) + \frac{5C_M\|c\|}{\alpha\mu},
    \end{align*}
    where the fifth inequality uses \eqref{ineq: Mk} and \eqref{ineq: M_sum}, and the seventh inequality uses the fact that $\|X_0-\bar{x}_0\bfonet\| = \|X_0\left(I - \frac{\bfone\bfonet}{n}\right)\|\leq \|X_0\|$ (same for $S_0$). \eqref{ineq: bdd_2_m} can be viewed as a corollary of the above inequality by noticing that
    \begin{align*}
        &\frac{1}{n}\E\left[\|X_k\|^2\right]\leq \frac{2}{n}\E\left[\|X_k - x^*\bfonet\|^2 + n\|x^*\|^2\right] \\
        \leq &6C_M\left(1-\frac{2\alpha\mu}{3}\right)^k\left(\E\left[\|\bar{x}_0 - x^*\|^2\right] + \E\left[\frac{\|X_0\|^2 + \|S_0\|^2}{n}\right]\right) + \frac{10C_M\|c\|}{\alpha\mu} + 2\|x^*\|^2.
    \end{align*}
\end{proof}
\noindent{\bf Remark:} 
\begin{itemize}
    \item Lemma~\ref{lem: dsgd_quadratic} characterizes convergence results of decentralized stochastic gradient descent with gradient tracking on strongly convex quadratic functions. Moreover, it also indicates that the second moment of $X_k$ can be bounded by using \eqref{ineq: bdd_2_m}, which will be used in proving the boundedness of second moment of $Z_t^{(k)}$ of our HIGP oracle.
    
    \item If we consider the same updates under the deterministic setting, then $\sigma_1 = \sigma_2 = 0$ and thus $\|c\|= 0$ by definition, which indicates the constant term in \eqref{ineq: quadratic_conv} vanishes (i.e., linear convergence). We will utilize this important observation in the next lemma.
\end{itemize}
 
\begin{lemma}\label{lem: dsgd_higp}
    Suppose Assumptions \ref{assump: lip_convexity}, \ref{assump: W} and \ref{assump: stoc_derivatives} hold. In Algorithm \ref{algo: HIGP_oracle} define
    \begin{align*}
        &C_1 = 9\sigma_{g,2}^2 + 6\gamma^2(\sigma_{g,2}^2 + L_{g,1}^2) + \frac{18\gamma^2\sigma_{g,2}^2(\sigma_{g,2}^2 + L_{g,1}^2)}{n},\\
        &C_2 = 12(\sigma_{g,2}^2 + L_{g,1}^2) + 9\sigma_{g,2}^2 + 12\gamma^2L_{g,1}^2(\sigma_{g,2}^2 + L_{g,1}^2)  + \frac{18\gamma^2\sigma_{g,2}^2(\sigma_{g,2}^2 + L_{g,1}^2)}{n}, \\
        &C_3 = 6\rho^2(\sigma_{g,2}^2 + L_{g,1}^2),\ C_4 = 2\sigma_{f}^2 + \frac{6\gamma^2\sigma_{f}^2(\sigma_{g,2}^2 + L_{g,1}^2)}{n} + (9\sigma_{g,2}^2 + \frac{18\gamma^2\sigma_{g,2}^2(\sigma_{g,2}^2 + L_{g,1}^2)}{n})\|x^*\|^2, \\
        &c = \left(\frac{\gamma^2}{n}\left(3\sigma_{g,2}^2\frac{L_{f,0}^2}{\mu_g^2} + \sigma_f^2\right),\ 0,\ \frac{(1+\rho^2)}{1-\rho^2}C_4\right)\T,\ M = 
        \begin{pmatrix}
            M_{11} &M_{12} &0 \\
            0 &M_{22} &M_{23} \\
            M_{31} &M_{32} &M_{33}
        \end{pmatrix}, \\
        & M_{11} = 1-\gamma\mu_g,\ M_{12} = \left(\frac{2\gamma}{\mu_g} + 2\gamma^2\right)(\sigma_{g,2}^2 + L_{g,1}^2),\ M_{22} = \frac{1+\rho^2}{2},\ M_{23} = \gamma^2\frac{1+\rho^2}{1-\rho^2}, \\
        &M_{31} = \frac{1+\rho^2}{1-\rho^2}C_1,\ M_{32} = \frac{1+\rho^2}{1-\rho^2}C_2,\ M_{33} = \frac{1+\rho^2}{2} + \frac{1+\rho^2}{1-\rho^2}C_3\gamma^2.
    \end{align*}
    Define $\tilde{M}$ to be matrix $M$ and $C_{\tilde{M}}$ to be $C_{M}$ when $\sigma_{g,2} = \sigma_{f} = 0$. If $\gamma$ satisfies
    \begin{equation}\label{ineq: gamma_for_quadratics}
        \begin{aligned}
            &\left(1+\frac{\gamma\mu_g}{2}\right)(1-\gamma\mu_g)^2 + \frac{3\gamma^2\sigma_{g,2}^2}{n}<1-\gamma\mu_g,\ 0< \gamma_1\leq \gamma \leq \gamma_2 \text{ for some } 0<\gamma_1<\gamma_2, \\
            &\max\left(\rho(\tilde{M}), \rho(M)\right) < 1 - \frac{2\gamma\mu_g}{3},\  \text{both } M \text{ and } \tilde{M} \text{have 3 different positive eigenvalues,}  
        \end{aligned}
    \end{equation}
    then for any $0\leq t\leq N$ we have
    \begin{align}
        &\E\left[\|\bar{z}_t^{(k)}\|^2|\F_k\right]\leq\frac{1}{n}\E\left[\|Z_t^{(k)}\|^2|\F_k\right]\leq \sigma_z^2 := 6C_M\left(\frac{L_{f,0}^2}{\mu_g^2} + \sigma_{f}^2 + L_{f,0}^2\right) + \frac{10C_M\|c\|}{\gamma\mu_g} + \frac{2L_{f,0}^2}{\mu_g^2}, \label{ineq: z_bdd_2_m} \\
        &\frac{1}{n}\|\E\left[Z_{t}^{(k)} - \bar{z}_{t}^{(k)}\bfonet|\F_k\right]\|^2\leq 3C_{\tilde{M}}\left(1-\frac{2\gamma\mu_g}{3}\right)^t\left(\frac{L_{f,0}^2}{\mu_g^2} + L_{f,0}^2\right). \label{ineq: z_exp_cons}
    \end{align}
\end{lemma}
\begin{proof}[Proof of Lemma~\ref{lem: dsgd_higp}]
    Note that \eqref{ineq: z_bdd_2_m} is a direct results of Lemma \ref{lem: dsgd_quadratic} by noticing that 
    \begin{align*}
        &z_{i,t+1}^{(k)} = \sum_{j=1}^{n}w_{ij}z_{j,t}^{(k)} - \gamma d_{i,t}^{(k)},\  Z_0^{(k)} = 0,\\
        &d_{i,t+1}^{(k)} =\sum_{i=1}^{n}w_{ij}d_{j,t}^{(k)} + s_{i,t+1}^{(k)} - s_{i,t}^{(k)},\ s_{i,t}^{(k)} = H_{i,t}^{(k)}z_{i,t}^{(k)}- b_{i,t}^{(k)}, \\
        &\E\left[H_{i,t}^{(k)}|\F_k\right] = \nabla_y^2g_i(x_{i,k},y_{i,k}^{(T)}) ,\ \E\left[\|H_{i,t}^{(k)}-\nabla_y^2g_i(x_{i,k},y_{i,k}^{(T)})\|^2|\F_k\right]\leq \sigma_{g,2}^2, \\
        &\E\left[b_{i,t}^{(k)}|\F_k\right] = \nabla_yf_i(x_{i,k},y_{i,k}^{(T)}),\ \E\left[\|b_{i,t}^{(k)} - \nabla_yf_i(x_{i,k}, y_{i,k}^{(T)})\|^2|\F_k\right]\leq \sigma_f^2,
    \end{align*}
    for any $k\geq 0, 1\leq i\leq n,$ and $t\geq 0$. Assumption \ref{assump: lip_convexity} also indicates that 
    \begin{align*}
        \mu_g I\preceq \nabla_y^2g_i(x_{i,k},y_{i,k}^{(T)})\preceq L_{g,1} I,\ \|\nabla_yf_i(x_{i,k},y_{i,k}^{(T)})\|\leq L_{f,0}.
    \end{align*}
    Hence we know by \eqref{ineq: bdd_2_m},
    \begin{align*}
        &\E\left[\|\bar{z}_t^{(k)}\|^2|\F_k\right]\leq \frac{1}{n}\E\left[\|Z_t^{(k)}\|^2|\F_k\right]\leq 6C_M\left(1-\frac{2\gamma\mu_g}{3}\right)^k\left(\frac{L_{f,0}^2}{\mu_g^2} + \sigma_{f}^2 + L_{f,0}^2\right) + \frac{10C_M\|c\|}{\gamma\mu_g} + \frac{2L_{f,0}^2}{\mu_g^2}\leq \sigma_z^2,
    \end{align*}
    which proves \eqref{ineq: z_bdd_2_m}. To prove \eqref{ineq: z_exp_cons}, we notice that in expectation, the updates can be written as
    \begin{align*}
        &\E\left[z_{i,t+1}^{(k)}|\F_k\right] = \sum_{j=1}^{n}w_{ij}\E\left[z_{j,t}^{(k)}|\F_k\right] - \gamma \E\left[d_{i,t}^{(k)}|\F_k\right],\  Z_0^{(k)} = 0,\\
        &\E\left[d_{i,t+1}^{(k)}|\F_k\right] =\sum_{i=1}^{n}w_{ij}\E\left[d_{j,t}^{(k)}|\F_k\right] + \E\left[s_{i,t+1}^{(k)}|\F_k\right] - \E\left[s_{i,t}^{(k)}|\F_k\right], \\
        &\E\left[s_{i,t}^{(k)}|\F_k\right] = \nabla_y^2g_i(x_{i,k},y_{i,k}^{(T)})\E\left[z_{i,t}^{(k)}|\F_k\right]- \nabla_yf_i(x_{i,k},y_{i,k}^{(T)}).
    \end{align*}
    The updates of $\E\left[z_{i,t}^{(k)}|\F_k\right]$ can be viewed as a noiseless case (i.e., $\sigma_{g,2} =\sigma_f = 0$) of Lemma \ref{lem: dsgd_quadratic}. Using this observation, \eqref{ineq: quadratic_conv}, and the definition of $\|c\|$ and $\tilde{M}$, we know \eqref{ineq: z_exp_cons} holds.
\end{proof}
Now we provide a technical lemma that guarantees \eqref{ineq: stepsize_for_quadratics} and \eqref{ineq: gamma_for_quadratics}. For simplicity we can just consider \eqref{ineq: stepsize_for_quadratics}.

\begin{lemma}\label{ineq: stepsize_for_higp}
    Let $M$ be the matrix defined in Lemma \ref{lem: dsgd_quadratic}. There exist $0<\alpha_1<\alpha_2$ such that $\alpha\in (\alpha_1, \alpha_2)$ and
    \begin{align}
        &\left(1+\frac{\alpha\mu}{2}\right)(1-\alpha\mu)^2 + \frac{3\alpha^2\sigma_1^2}{n}<1-\alpha\mu, \label{ineq: lr_first_constraint}\\
        &\rho(M) < 1 - \frac{2\alpha\mu}{3},\ \text{ and }~M~\text{has 3 different positive eigenvalues.}  \label{ineq: lr_second_constraint}
    \end{align}
\end{lemma}
\begin{proof}[Proof of Lemma~\ref{ineq: stepsize_for_higp}]
    Note that \eqref{ineq: lr_first_constraint} is equivalent to
    \[
        \mu^3\alpha^2 + \frac{6\alpha\sigma_1^2}{n} - \mu < 0,
    \]
    which implies any $\alpha_1,\ \alpha_2$ satisfying
    \begin{equation}\label{ineq: alpha_for_first}
        0<\alpha_1<\alpha_2<\frac{\sqrt{9\sigma_1^4 + n^2\mu^4} - 3\sigma_1^2}{n\mu^3}
    \end{equation}
    will ensure \eqref{ineq: lr_first_constraint}. Next we consider \eqref{ineq: lr_second_constraint}. Define
    \[
        \varphi(\lambda) := \det(\lambda I - M) = \prod_{i=1}^{3}(\lambda - M_{ii}) - M_{23}M_{32}(\lambda - M_{11}) - M_{12}M_{23}M_{31}.
    \]
    We know that a sufficient condition to guarantee \eqref{ineq: lr_second_constraint} is
    \begin{equation}\label{ineq: varphi_condition}
        \varphi\left(1-\frac{2\alpha\mu}{3}\right) > 0,\ \varphi(M_{11}) < 0,\ \varphi(M_{22}) > 0,\ \varphi(0)< 0,\ M_{11} > M_{22},
    \end{equation}
    since this implies $0< M_{22}< M_{11} = 1-\alpha\mu < 1 - \frac{2\alpha\mu}{3}$ and
    \[
        \varphi\left(1-\frac{2\alpha\mu}{3}\right)\cdot \varphi(M_{11}) < 0,\ \varphi(M_{11})\cdot \varphi(M_{22}) < 0,\ \varphi(M_{22})\cdot \varphi(0) < 0,
    \]
    which together with continuity of $\varphi$ indicate the roots of $\varphi(\lambda)= 0$ (i.e., the eigenvalues of $M$, denoted as $\lambda_1, \lambda_2,\ \lambda_3$ in descending order) satisfy
    \[
        0< \lambda_3 < M_{22} < \lambda_2 < M_{11} < \lambda_1 < 1 - \frac{2\alpha\mu}{3}.
    \]
    The condition $\varphi(M_{11})<0$ is automatically true by definition of $\varphi$ and $M$, and for the rest of the conditions in \eqref{ineq: varphi_condition} we have
    \begin{align*}
        &\varphi\left(1-\frac{2\alpha\mu}{3}\right) > 0 \\
        \Leftrightarrow &\alpha\cdot \varphi_1(\alpha):=\frac{\alpha\mu}{3}\left[\left(\frac{1-\rho^2}{2} - \frac{2\alpha\mu}{3}\right)\left(\frac{1-\rho^2}{2} - \frac{2\alpha\mu}{3} - \frac{1+\rho^2}{1-\rho^2}C_3\alpha^2\right) - \left(\frac{1+\rho^2}{1-\rho^2}\right)^2C_2\alpha^2\right] \\
        &\hspace{5em} - \left(\frac{1+\rho^2}{1-\rho^2}\right)^2C_1\alpha^2\left(\frac{2\alpha}{\mu} + 2\alpha^2\right)\tilde{\sigma}_1^2 > 0, \\
        &\varphi(M_{22}) > 0 \Leftrightarrow M_{23}((M_{11} - M_{22})M_{32} - M_{12}M_{31}) > 0 \Leftrightarrow (M_{11} - M_{22})M_{32} - M_{12}M_{31} > 0\\
        \Leftrightarrow&\varphi_2(\alpha):= \left(\frac{1-\rho^2}{2} - \alpha\mu\right)\frac{1+\rho^2}{1-\rho^2}C_2 - \frac{1+\rho^2}{1-\rho^2}C_1\left(\frac{2\alpha}{\mu} + 2\alpha^2\right)\tilde{\sigma}_1^2 > 0,\ \text{(by definition of } C_2, C_2 > 0 \text{ when }\alpha = 0 \text{)}\\
        &\varphi(0)<0 \Leftrightarrow -M_{11}(M_{22}M_{33} - M_{23}M_{32}) - M_{12}M_{23}M_{31} < 0 \Leftarrow M_{22}M_{33} - M_{23}M_{32} >0 \\
        \Leftrightarrow &\varphi_3(\alpha) := \frac{1+\rho^2}{2}\left(\frac{1+\rho^2}{2} + \frac{1+\rho^2}{1-\rho^2}C_3\alpha^2\right) - \left(\frac{1+\rho^2}{1-\rho^2}\right)^2C_2\alpha^2 > 0, \\
        &M_{11} > M_{22}\Leftrightarrow \alpha< \frac{1-\rho^2}{2\mu}.
    \end{align*}
    Hence a sufficient condition for \eqref{ineq: varphi_condition} is
    \[
        \varphi_1(\alpha) > 0,\ \varphi_2(\alpha) >0,\ \varphi_3(\alpha) > 0,\ \alpha < \frac{1-\rho^2}{2\mu}.
    \]
    Given the expressions of $\varphi_i(\alpha)$ above, we know they satisfy $\varphi_i(0) > 0$. Hence we can define $\beta$ to be the minimum positive constant such that $\varphi_1(\beta)\varphi_2(\beta)\varphi_3(\beta) = 0$, and 
    \[
        \alpha_2 = \min\left(\frac{\sqrt{9\sigma_1^4 + n^2\mu^4} - 3\sigma_1^2}{n\mu^3}, \frac{1-\rho^2}{2\mu}, \beta\right),\ \alpha_1 = \text{ any constant in } (0,\alpha_2),
    \]
    which implies that for any $\alpha\in (\alpha_1,\alpha_2)$, we always have
    \[
        \varphi_1(\alpha) > 0,\ \varphi_2(\alpha) >0,\ \varphi_3(\alpha) > 0,\ \alpha < \frac{\sqrt{9\sigma_1^4 + n^2\mu^4} - 3\sigma_1^2}{n\mu^3},\ \alpha < \frac{1-\rho^2}{2\mu},
    \]
    because of the definition of $\beta$, and $\varphi_i(0) = 0$ for all $1\leq i\leq 3$.  \eqref{ineq: alpha_for_first}. The above expression implies \eqref{ineq: alpha_for_first} and \eqref{ineq: varphi_condition}, and hence \eqref{ineq: lr_first_constraint} and \eqref{ineq: lr_second_constraint} are satisfied.
\end{proof}

\noindent {\bf Remark:} 
\begin{itemize}
    \item One can follow the proof of Corollary 1 in \citet{pu2021distributed} to obtain an explicit dependence between $\alpha_1, \alpha_2$ and other parameters, which is purely technical and we omit it in this lemma.
    
    \item Define $\tilde{\alpha}_2$ to be the constant $\alpha_2$ when $\sigma_1 = \sigma_2 = 0$ in the above lemma. We can check that the proof is still valid and thus for any $\alpha\in (\frac{\min\left(\alpha_2, \tilde{\alpha}_2\right)}{2}, \min\left(\alpha_2, \tilde{\alpha}_2\right) )$ we have
    \begin{align*}
        &\left(1+\frac{\alpha\mu}{2}\right)(1-\alpha\mu)^2 + \frac{3\alpha^2\sigma_1^2}{n}<1-\alpha\mu,\\
        &\max\left(\rho(\tilde{M}), \rho(M)\right) < 1 - \frac{2\alpha\mu}{3},\ \text{both } M~\text{and}~ \tilde{M}~\text{have 3 different positive eigenvalues}, 
    \end{align*}
    and thus the existence of $\gamma_1$ and $\gamma_2$ in \eqref{ineq: gamma_for_quadratics} is also guaranteed.
\end{itemize}

Using Lemma \ref{lem: dsgd_higp} we could directly bound $\|X_k - \bar{x}_k\bfonet\|^2$ and $\|Y_k^{(t+1)} - \bar{y}_k^{(t+1)}\bfonet\|^2$.
\begin{lemma}\label{lem: XY_cons}
    Suppose Assumptions \ref{assump: lip_convexity}, \ref{assump: W}, \ref{assump: similarity_of_g}, and \ref{assump: stoc_derivatives} hold. Define
    \begin{align*}
        &\sigma_u^2 = 2(L_{f,0}^2 + \sigma_f^2) + 2(L_{g,1}^2 + \sigma_{g,2}^2)\sigma_z^2,\ \sigma_x^2 = \frac{1+\rho^2}{1-\rho^2}\cdot \sigma_u^2,\ \tilde{\alpha}_{k+1}^2 = \sum_{i=0}^{k}\alpha_i^2\left(\frac{1+\rho^2}{2}\right)^{k-i}, \\
        &\tilde{\beta}_{k+1}^2 = \frac{1+\rho^2}{1-\rho^2}\sum_{i=0}^{k}\beta_i^2(2\sigma_{g,1}^2 + 6L_{g,1}^2\sigma_x^2\tilde{\alpha}_i^2 + 3\delta^2)\left(\frac{3+\rho^2}{4}\right)^{k-i},\ \tilde{\alpha}_0 = \tilde{\beta}_0 = 0.
    \end{align*}
    If $\beta_k$ satisfy
    \begin{equation}\label{ineq: beta_condition}
        \frac{(1+\rho^2)}{2} + \beta_k^2\frac{1+\rho^2}{1-\rho^2}\cdot 6L_{g,1}^2 \leq \frac{3+\rho^2}{4}<1,
    \end{equation}
    then in Algorithm \ref{algo:DSBO}, for any $k\geq 0$ and $0\leq t\leq T-1$ we have
    \begin{equation}\label{ineq: XY_consensus}
        \begin{aligned}
            &\E\left[\|U_k\|^2\right]\leq n\sigma_u^2,\ \E\left[\|X_k - \bar{x}_k\bfonet\|^2\right]\leq n\sigma_x^2\tilde{\alpha}_k^2, \\ 
            &\frac{1}{n}\E\left[\|Y_{k}^{(t)} - \bar{y}_{k}^{(t)}\bfonet\|^2\right]\leq \left[\left(\frac{3+\rho^2}{4}\right)^{t}T - t\left(\frac{3+\rho^2}{4}\right)\right]\tilde{\beta}_k^2 + t\tilde{\beta}_{k+1}^2.
        \end{aligned}
    \end{equation}
\end{lemma}
\begin{proof}[Proof of Lemma~\ref{lem: XY_cons}]
    Note that the inner and outer loop updates satisfy
    \begin{align*}
        &\bar{x}_{k+1} = \bar{x}_k - \alpha_k \bar{r}_k,\ X_{k+1} - \bar{x}_{k+1}\bfonet = X_kW - \bar{x}_k\bfonet - \alpha_k(R_k - \bar{r}_k\bfonet), \\
        &\bar{y}_k^{(t+1)} = \bar{y}_k^{(t)} - \beta_k\bar{v}_k^{(t)},\ Y_{k}^{(t+1)} - \bar{y}_k^{(t+1)} = Y_k^{(t)}W - \bar{y}_k^{(t)}\bfonet - \beta_k(V_k^{(t)} - \bar{v}_k^{(t)}\bfonet),
    \end{align*}
    which gives
    \begin{align}
        &\|X_{k+1} - \bar{x}_{k+1}\bfonet\|^2 \leq \frac{(1+\rho^2)}{2}\|X_k - \bar{x}_k\bfonet\|^2 + \alpha_k^2\frac{1+\rho^2}{1-\rho^2}\|R_k - \bar{r}_k\bfonet\|^2, \label{ineq: X_cons_recursion}\\
        &\|Y_{k}^{(t+1)} - \bar{y}_{k}^{(t+1)}\bfonet\|^2 \leq \frac{(1+\rho^2)}{2}\|Y_k^{(t)} - \bar{y}_k^{(t)}\bfonet\|^2 + \beta_k^2\frac{1+\rho^2}{1-\rho^2}\|V_k^{(t)} - \bar{v}_k^{(t)}\bfonet\|^2. \label{ineq: Y_cons_recursion}
    \end{align}
    The inequalities hold similarly as the inequality in \eqref{ineq: consensus_ineq}. Notice that we have
    \begin{align*}
        &\|R_k - \bar{r}_k\bfonet\| = \|R_k\left(I - \frac{\bfone\bfonet}{n}\right)\| = \|\left[(1-\alpha_k)R_{k-1} + \alpha_kU_{k-1}\right]\left(I-\frac{\bfone\bfonet}{n}\right)\|\\
        \leq &\max\left(\|R_{k-1}\left(I-\frac{\bfone\bfonet}{n}\right)\|,\ \|U_{k-1}\left(I-\frac{\bfone\bfonet}{n}\right)\|\right)\leq \max_{0\leq i\leq k-1}\left(\|U_{i}\left(I-\frac{\bfone\bfonet}{n}\right)\|\right).
    \end{align*}
    The second inequality holds by repeating the first inequality multiple times. For each $\|U_k - \bar{u}_k\bfonet\|$ we have
    \begin{align*}
        &\E\left[\|U_k - \bar{u}_k\bfonet\|^2\right] = \E\left[\|U_k\left(I - \frac{\bfone\bfonet}{n}\right)\|^2\right]\leq \E\left[\|U_k\|^2\right] = \sum_{i=1}^{n}\E\left[\|u_{i,k}\|^2\right] \\
        \leq &2\sum_{i=1}^{n}\left(\E\left[\|\nabla_xf_i(x_{i,k}, y_{i,k}^{(T)};\phi_{i,0})\|^2\right] + \E\left[\|\nabla_{xy}^2 g_i(x_{i,k},y_{i,k}^{(T)};\xi_{i,0} )z_{i,N}^{(k)}\|^2\right]\right)\\
        \leq &2\sum_{i=1}^{n}\left(L_{f,0}^2 + \sigma_f^2 + (L_{g,1}^2 + \sigma_{g,2}^2)\E\left[\|z_{i,N}^{(k)}\|^2\right]\right) \leq 2n(L_{f,0}^2 + \sigma_f^2) + 2n(L_{g,1}^2 + \sigma_{g,2}^2)\sigma_z^2 = n\sigma_u^2.
    \end{align*}
    The fourth inequality uses \eqref{ineq: z_bdd_2_m}. Using the above two inequaities in \eqref{ineq: X_cons_recursion} we know
    \[
        \|X_{k+1} - \bar{x}_{k+1}\bfonet\|^2 \leq \frac{(1+\rho^2)}{2}\|X_k - \bar{x}_k\bfonet\|^2 + n\alpha_k^2\sigma_x^2.
    \]
    Using Lemma \ref{lem: useful_ineq} and $X_0 = 0$, we can obtain the first two results of \eqref{ineq: XY_consensus}. To analyze $\|V_k^{(t)} - \bar{v}_k^{(t)}\bfonet\|$, we first notice that 
    \begin{align*}
        v_{i,k}^{(t)} - \bar{v}_k^{(t)} &= v_{i,k}^{(t)} - \nabla_y g_i(x_{i,k}, y_{i,k}^{(t)}) - (\bar{v}_k^{(t)} - \frac{1}{n}\sum_{l=1}^{n}\nabla_yg_l(x_{l,k}, y_{l,k}^{(t)})) + \nabla_y g_i(x_{i,k}, y_{i,k}^{(t)}) - \nabla_yg_i(\bar{x}_k, \bar{y}_k^{(t)}) \\
        &- \frac{1}{n}\sum_{l=1}^{n}(\nabla_yg_l(x_{l,k}, y_{l,k}^{(t)}) - \nabla_yg_l(\bar{x}_k, \bar{y}_k^{(t)}))
        +\nabla_yg_i(\bar{x}_k, \bar{y}_k^{(t)}) - \frac{1}{n}\sum_{l=1}^{n}\nabla_yg_l(\bar{x}_k, \bar{y}_k^{(t)}).
    \end{align*}
    Hence we know 
    \begin{align*}
        &\E\left[\|V_k^{(t)} - \bar{v}_k^{(t)}\bfonet\|^2\right] = \sum_{i=1}^{n}\E\left[\|v_{i,k}^{(t)} - \bar{v}_k^{(t)}\|^2\right] \\
        \leq & (n+1)\sigma_{g,1}^2 + 3\sum_{i=1}^{n}\E\left[L_{g,1}^2(\|x_{i,k} - \bar{x}_k\|^2 + \|y_{i,k}^{(t)} - \bar{y}_k^{(t)}\|^2) + \frac{L_{g,1}^2}{n}\sum_{l=1}^{n}(\|x_{l,k} - \bar{x}_k\|^2 + \|y_{l,k}^{(t)}-\bar{y}_k^{(t)}\|^2) + \delta^2\right]\\
        = &(n+1)\sigma_{g,1}^2 + 6L_{g,1}^2\E\left[\|X_k - \bar{x}_k\bfonet\|^2 + \|Y_{k} - \bar{y}_k^{(t)}\bfonet\|^2\right] + 3n\delta^2 \\
        \leq &6L_{g,1}^2\E\left[\|Y_k^{(t)} - \bar{y}_k^{(t)}\bfonet\|^2\right] + 2n\sigma_{g,1}^2 + 6nL_{g,1}^2\sigma_x^2\tilde{\alpha}_k^2 + 3n\delta^2,
    \end{align*}
    where the second inequality uses the first result of \eqref{ineq: XY_consensus}. The above inequality together with \eqref{ineq: Y_cons_recursion} imply
    \begin{equation}\label{ineq: Y_kt_cons_recursion}
        \begin{aligned}
            &\frac{1}{n}\E\left[\|Y_{k}^{(t+1)} - \bar{y}_{k}^{(t+1)}\bfonet\|^2\right] \\
        \leq &\left[\frac{(1+\rho^2)}{2} + \beta_k^2\frac{1+\rho^2}{1-\rho^2}\cdot 6L_{g,1}^2 \right]\cdot\frac{1}{n}\E\left[\|Y_k^{(t)} - \bar{y}_k^{(t)}\bfonet\|^2\right] + \beta_k^2\frac{1+\rho^2}{1-\rho^2}(2\sigma_{g,1}^2 + 6L_{g,1}^2\sigma_x^2\tilde{\alpha}_k^2 + 3\delta^2) \\
        \leq &\left(\frac{3+\rho^2}{4}\right)^{t+1}\cdot\frac{1}{n}\E\left[\|Y_k^{(0)} - \bar{y}_k^{(0)}\bfonet\|^2\right] + \beta_k^2\frac{1+\rho^2}{1-\rho^2}(2\sigma_{g,1}^2 + 6L_{g,1}^2\sigma_x^2\tilde{\alpha}_k^2 + 3\delta^2)\sum_{l=0}^{t}\left(\frac{3+\rho^2}{4}\right)^i \\
        \leq &\left(\frac{3+\rho^2}{4}\right)^{t+1}\cdot\frac{1}{n}\E\left[\|Y_k^{(0)} - \bar{y}_k^{(0)}\bfonet\|^2\right] + (t+1)\beta_k^2\frac{1+\rho^2}{1-\rho^2}(2\sigma_{g,1}^2 + 6L_{g,1}^2\sigma_x^2\tilde{\alpha}_k^2 + 3\delta^2),
        \end{aligned}
    \end{equation}
    where the second inequality uses Lemma \ref{lem: useful_ineq} and \eqref{ineq: beta_condition}. Notice that we use warm-start strategy (i.e., $Y_{k+1}^{(0)} = Y_{k}^{(T)}$), hence we know
    \begin{align*}
        &\frac{1}{n}\E\left[\|Y_{k+1}^{(0)} - \bar{y}_{k+1}^{(0)}\bfonet\|^2\right] = \frac{1}{n}\E\left[\|Y_{k}^{(T)} - \bar{y}_{k}^{(T)}\bfonet\|^2\right]\\
            \leq &\left(\frac{3+\rho^2}{4}\right)^{T}\cdot\frac{1}{n}\E\left[\|Y_k^{(0)} - \bar{y}_k^{(0)}\bfonet\|^2\right] + T\beta_k^2\frac{1+\rho^2}{1-\rho^2}(2\sigma_{g,1}^2 + 6L_{g,1}^2\sigma_x^2\tilde{\alpha}_k^2 + 3\delta^2) \\
            \leq & T\frac{1+\rho^2}{1-\rho^2}\sum_{i=0}^{k}\beta_i^2(2\sigma_{g,1}^2 + 6L_{g,1}^2\sigma_x^2\tilde{\alpha}_i^2 + 3\delta^2)\left(\frac{3+\rho^2}{4}\right)^{k-i} = T\tilde{\beta}_{k+1}^2,
    \end{align*}
    where the second inequality uses Lemma \ref{lem: useful_ineq}. Using the above estimation in \eqref{ineq: Y_kt_cons_recursion}, we know
    \begin{align*}
        &\frac{1}{n}\E\left[\|Y_{k}^{(t+1)} - \bar{y}_{k}^{(t+1)}\bfonet\|^2\right]\\
        \leq &\left(\frac{3+\rho^2}{4}\right)^{t+1}\cdot\frac{1}{n}\E\left[\|Y_k^{(0)} - \bar{y}_k^{(0)}\bfonet\|^2\right] + (t+1)\beta_k^2\frac{1+\rho^2}{1-\rho^2}(2\sigma_{g,1}^2 + 6L_{g,1}^2\sigma_x^2\tilde{\alpha}_k^2 + 3\delta^2) \\
        \leq & \left(\frac{3+\rho^2}{4}\right)^{t+1}T\tilde{\beta}_k^2 + (t+1)\left(\tilde{\beta}_{k+1}^2 - \left(\frac{3+\rho^2}{4}\right)\tilde{\beta}_k^2\right),
    \end{align*}
    and thus the proof is complete by rearranging the terms.
\end{proof}
Now we are ready to analyze the convergence of the inner loop of Algorithm \ref{algo:DSBO}.

\begin{lemma}\label{lem: y_error}
    Suppose Assumptions \ref{assump: lip_convexity} and \ref{assump: stoc_derivatives} hold. For any $0\leq t\leq T-1$ define
    \begin{equation}\label{eq: C_kt_const}
        C_{k,t+1} = \sum_{l=0}^{t}\left[\left(\frac{\beta_k}{\mu_g} + \beta_k^2\right)L_{g,1}^2\left(\sigma_x^2\tilde{\alpha}_k^2 + \left[\left(\frac{3+\rho^2}{4}\right)^{l}T - l\left(\frac{3+\rho^2}{4}\right)\right]\tilde{\beta}_k^2 + l\tilde{\beta}_{k+1}^2\right) + \frac{\beta_k^2\sigma_{g,1}^2}{n}\right].
    \end{equation}
    If $T\geq 1$ and $0<\beta_k\leq \min\{1, \frac{1}{\mu_g}\}$, then in Algorithm \ref{algo:DSBO}, we have
    \begin{equation}\label{ineq: bar_y_sum}
        \begin{aligned}
            \frac{\mu_g}{2}\sum_{k=1}^{K}\beta_k\E\left[\|\bar{y}_k^{(0)} - y_{k-1}^*\|^2\right]\leq \E\left[\|\bar{y}_{1}^{(0)}-y_{0}^*\|^2\right] + L_{y^*}^2\sum_{k=1}^{K}\left(\frac{2\alpha_{k-1}^2}{\beta_k\mu_g}+\alpha_{k-1}^2\right)\E\left[\|\bar{r}_{k-1}\|^2\right] + \sum_{k=1}^{K}C_{k,T},
        \end{aligned}
    \end{equation}
    where $y_k^* = y^*(\bar{x}_k) = \argmin_y \sum_{i=1}^{n}g_i(\bar{x}_k, y)$
\end{lemma}
\begin{proof}[Proof of Lemma~\ref{lem: y_error}]
    For any $k\geq 0,\ 1\leq t\leq T-1$, define $$\G_t^{(k)} = \sigma\left(\bigcup_{i=1}^{n}\{y_{i,0}^{(T)}, ..., y_{i,k-1}^{(T)}, y_{i,k}^{(t)}, x_{i,0},...,x_{i,k}, r_{i,0},...,r_{i,k}\}\right).$$ 
    We know
    \begin{align}\label{ineq: inner_1_step}
        \begin{aligned}
            &\E\left[\|\bar{y}_{k}^{(t+1)} - y_k^*\|^2|\G_t\right] \\
        = &\E\left[\|\bar{y}_k^{(t)} - \beta_k\nabla_yg(\bar{x}_k, \bar{y}_k^{(t)}) - y_k^* - \beta_k\left(\bar{v}_k^{(t)} - \E\left[\bar{v}_k^{(t)}|\G_t\right]\right) - \beta_k\left(\E\left[\bar{v}_k^{(t)}|\G_t\right] - \nabla_yg(\bar{x}_k, \bar{y}_k^{(t)})\right)\|^2|\G_t\right] \\
        = & \E\left[\|\bar{y}_k^{(t)} - \beta_k\nabla_yg(\bar{x}_k, \bar{y}_k^{(t)}) - y_k^* - \beta_k\left(\E\left[\bar{v}_k^{(t)}|\G_t\right] - \nabla_yg(\bar{x}_k, \bar{y}_k^{(t)})\right)\|^2|\G_t\right] + \frac{\beta_k^2\sigma_{g,1}^2}{n}\\
        \leq &(1+\beta_k\mu_g)\|\bar{y}_k^{(t)} - \beta_k\nabla_yg(\bar{x}_k, \bar{y}_k^{(t)}) - y_k^*\|^2 + \left(1+\frac{1}{\beta_k\mu_g}\right)\beta_k^2\E\left[\|\E\left[\bar{v}_k^{(t)}|\G_t\right] - \nabla_yg(\bar{x}_k, \bar{y}_k^{(t)})\|^2|\G_t\right] +\frac{\beta_k^2\sigma_{g,1}^2}{n}\\
        \leq & (1+\beta_k\mu_g)(1 - \beta_k\mu_g)^2\|\bar{y}_k^{(t)}-y_k^*\|^2 + \left(\frac{\beta_k}{\mu_g} + \beta_k^2\right)\|\frac{1}{n}\sum_{i=1}^{n}\left(\nabla_yg_i(x_{i,k},y_{i,k}^{(t)}) - \nabla_yg_i(\bar{x}_k,\bar{y}_k^{(t)})\right)\|^2 + \frac{\beta_k^2\sigma_{g,1}^2}{n} \\
        \leq & (1 - \beta_k\mu_g)\|\bar{y}_k^{(t)}-y_k^*\|^2 + \frac{\left(\frac{\beta_k}{\mu_g} + \beta_k^2\right)L_{g,1}^2}{n}\left(\|X_k - \bar{x}_k\bfonet\|^2 + \|Y_k^{(t)} - \bar{y}_k^{(t)}\bfonet\|^2\right) + \frac{\beta_k^2\sigma_{g,1}^2}{n},
        \end{aligned}
    \end{align}
    where the second equality holds since $\bar{v}_k^{(t)} - \E\left[\bar{v}_k^{(t)}|\G_t\right]$ has expectation $0$ and 
    \[
        \E\left[\|\bar{v}_k^{(t)} - \E\left[\bar{v}_k^{(t)}|\G_t\right]\|^2|\G_t\right] = \E\left[\|\frac{1}{n}\sum_{i=1}^{n}\left(v_{i,k}^{(t)} - \E\left[v_{i,k}^{(t)}|\G_t\right]\right)\|^2|\G_t\right] \leq \frac{\sigma_{g,1}^2}{n},
    \]
    due to independence, the second inequality holds due to Lemma \ref{lem: gd_decrease} and $\beta_k\leq 1$, and the third inequality holds due to Lipschitz continuity of $\nabla_y g$. Taking expectation on both sides and using \eqref{ineq: XY_consensus} we know
    \begin{align*}
        &\E\left[\|\bar{y}_{k}^{(t+1)} - y_k^*\|^2\right] \\
        \leq &(1 - \beta_k\mu_g)\E\left[\|\bar{y}_k^{(t)}-y_k^*\|^2\right] + \left(\frac{\beta_k}{\mu_g} + \beta_k^2\right)L_{g,1}^2\left(\sigma_x^2\tilde{\alpha}_k^2 + \left[\left(\frac{3+\rho^2}{4}\right)^{t}T - t\left(\frac{3+\rho^2}{4}\right)\right]\tilde{\beta}_k^2 + t\tilde{\beta}_{k+1}^2\right) + \frac{\beta_k^2\sigma_{g,1}^2}{n} \\
        \leq &(1 - \beta_k\mu_g)^{t+1}\E\left[\|\bar{y}_k^{(0)}-y_k^*\|^2\right] + C_{k,t+1},
    \end{align*}
    where the second inequality uses Lemma \ref{lem: useful_ineq}. Observe that we also have
    \begin{equation}
        \begin{aligned}
            &\E\left[\|\bar{y}_{k+1}^{(0)}-y_{k}^*\|^2\right] = \E\left[\|\bar{y}_{k}^{(T)}-y_{k}^*\|^2\right] \leq  (1-\beta_k\mu_g)^{T}\E\left[\|\bar{y}_k^{(0)}-y_k^*\|^2\right] + C_{k,T} \\
            \leq & (1-\beta_k\mu_g)^{T}\E\left[\left(1+\frac{\beta_k\mu_g}{2}\right)\|\bar{y}_k^{(0)}-y_{k-1}^*\|^2 + \left(1+\frac{2}{\beta_k\mu_g}\right)\|y_{k-1}^* - y_k^*\|^2\right] + C_{k,T}\\
            \leq & \left(1+\frac{\beta_k\mu_g}{2}\right)(1-\beta_k\mu_g)^{T}\E\left[\|\bar{y}_k^{(0)}-y_{k-1}^*\|^2\right] + \left(\frac{2\alpha_{k-1}^2}{\beta_k\mu_g}+\alpha_{k-1}^2\right)L_{y^*}^2\E\left[\|\bar{r}_{k-1}\|^2\right] + C_{k,T}\\
            \leq & \left(1-\frac{\beta_k\mu_g}{2}\right)\E\left[\|\bar{y}_k^{(0)}-y_{k-1}^*\|^2\right] + \left(\frac{2\alpha_{k-1}^2}{\beta_k\mu_g}+\alpha_{k-1}^2\right)L_{y^*}^2\E\left[\|\bar{r}_{k-1}\|^2\right] + C_{k,T},
        \end{aligned}
    \end{equation}
    where the third inequality holds since $(1+\frac{a}{2})(1-a)^T\leq (1-\frac{a}{2})$ for any $a>0$ and $T\geq 1$, and $y^*(x)$ is $L_{y^*}$-smooth. This implies
    \[
        \frac{\beta_k\mu_g}{2}\E\left[\|\bar{y}_k^{(0)} - y_{k-1}^*\|^2\right]\leq \E\left[\|\bar{y}_k^{(0)}-y_{k-1}^*\|^2\right] - \E\left[\|\bar{y}_{k+1}^{(0)}-y_{k}^*\|^2\right] +  \left(\frac{2\alpha_{k-1}^2}{\beta_k\mu_g}+\alpha_{k-1}^2\right)L_{y^*}^2\E\left[\|\bar{r}_{k-1}\|^2\right] + C_{k,T}.
    \]
    Taking summation on both sides, we have
    \begin{align*}
        \frac{\mu_g}{2}\sum_{k=1}^{K}\beta_k\E\left[\|\bar{y}_k^{(0)} - y_{k-1}^*\|^2\right]\leq \E\left[\|\bar{y}_{1}^{(0)}-y_{0}^*\|^2\right] + L_{y^*}^2\sum_{k=1}^{K}\left(\frac{2\alpha_{k-1}^2}{\beta_k\mu_g}+\alpha_{k-1}^2\right)\E\left[\|\bar{r}_{k-1}\|^2\right] + \sum_{k=1}^{K}C_{k,T}.
    \end{align*}

\end{proof}

\begin{lemma}\label{lem: higp_main}
    Suppose Assumptions \ref{assump: lip_convexity}, \ref{assump: W}, \ref{assump: similarity_of_g}, and \ref{assump: stoc_derivatives} hold. In Algorithm \ref{algo: HIGP_oracle} define
    \begin{align*}
        &H^{(k)} := \frac{1}{n}\sum_{i=1}^{n}\nabla_y^2g_i(\bar{x}_k,y_k^*),\ b^{(k)} := \frac{1}{n}\sum_{i=1}^{n}\nabla_y f_i(\bar{x}_k, y_k^*), \\
        &z_{*}^{(k)} := \left(H^{(k)}\right)^{-1}\cdot b^{(k)} = \left(\sum_{i=1}^{n}\nabla_y^2g_i(\bar{x}_k,y_k^*)\right)^{-1}\left(\sum_{i=1}^{n}\nabla_y f_i(\bar{x}_k, y_k^*)\right),
    \end{align*}
    If $\gamma$ satisfies \eqref{ineq: gamma_for_quadratics}, then we have
    \begin{align}
        &\E\left[\|\E\left[\bar{z}_{t}^{(k)}|\F_k\right] - z_{*}^{(k)}\|^2\right] \notag\\
        \leq &(1-\gamma\mu_g)^{N}\cdot\frac{L_{f,0}^2}{\mu_g^2} + 5\left(\frac{1}{\mu_g^2} + \frac{\gamma}{\mu_g}\right)\left(L_{g,2}^2\sigma_z^2 + L_{f,1}^2\right)\left(\E\left[\|\bar{y}_{k+1}^{(0)} - y_k^*\|^2\right] + \sigma_x^2\tilde{\alpha}_k^2 + T\tilde{\beta}_{k+1}^2\right) \notag \\
            + &90C_{\tilde{M}}L_{g,1}^2\left(\frac{1}{\mu_g^2} + \frac{\gamma}{\mu_g}\right)\left(\frac{L_{f,0}^2}{\mu_g^2} +L_{f,0}^2\right)\left(1-\frac{2\gamma\mu_g}{3}\right)^{N-1}. \label{ineq: z_conv_higp}
    \end{align}
\end{lemma}
\begin{proof}[Proof of Lemma~\ref{lem: higp_main}]
    Define 
    \[
        \dot{z}_{t,k} := \E\left[\bar{z}_{t}^{(k)}|\F_k\right],\ \dot{s}_{t,k} := \E\left[\bar{s}_{t}^{(k)}|\F_k\right].
    \]
    We know
    \begin{align*}
        &\dot{z}_{t+1,k} - z_*^{(k)} = \dot{z}_{t+1,k} - z_*^{(k)} = \E\left[\bar{z}_{t}^{(k)}|\F_k\right] - \gamma\E\left[\bar{s}_{t}^{(k)}|\F_k\right] - z_*^{(k)} \\
        = &\E\left[\bar{z}_{t}^{(k)}|\F_k\right] - \gamma\left(H^{(k)}\E\left[\bar{z}_{t}^{(k)}|\F_k\right] - b^{(k)}\right) - z_*^{(k)} - \gamma\left(\E\left[\bar{s}_{t}^{(k)}|\F_k\right] - \left(H^{(k)}\E\left[\bar{z}_{t}^{(k)}|\F_k\right] - b^{(k)}\right)\right) \\
        = & \dot{z}_{t,k} - \gamma\left(H^{(k)}\dot{z}_{t,k} - b^{(k)}\right) - z_*^{(k)} - \gamma\left(\dot{s}_{t,k} - \left(H^{(k)}\dot{z}_{t,k} - b^{(k)}\right)\right).
    \end{align*}
    Hence we know
    \begin{align}
        &\|\dot{z}_{t+1,k} - z_*^{(k)}\|^2 \notag \\
        = &\|\dot{z}_{t,k} - \gamma\left(H^{(k)}\dot{z}_{t,k} - b^{(k)}\right) - z_*^{(k)} - \gamma\left(\dot{s}_{t,k} - \left(H^{(k)}\dot{z}_{t,k} - b^{(k)}\right)\right)\|^2 \notag \\
        \leq & (1+\gamma\mu_g)\|\dot{z}_{t,k} - \gamma\left(H^{(k)}\dot{z}_{t,k} - b^{(k)}\right) - z_*^{(k)}\|^2 + (1 + \frac{1}{\gamma\mu_g})\gamma^2\|\dot{s}_{t,k} - \left(H^{(k)}\dot{z}_{t,k} - b^{(k)}\right)\|^2 \notag \\
        \leq & (1+\gamma\mu_g)(1 - \gamma\mu_g)^2\|\dot{z}_{t,k} - z_*^{(k)}\|^2 + \left(\frac{\gamma}{\mu_g} + \gamma^2\right)\|\dot{s}_{t,k} - \left(H^{(k)}\dot{z}_{t,k} - b^{(k)}\right)\|^2 \notag\\
        \leq &(1-\gamma\mu_g)\|\dot{z}_{t,k} - z_*^{(k)}\|^2 + \left(\frac{\gamma}{\mu_g} + \gamma^2\right)\|\dot{s}_{t,k} - \left(H^{(k)}\dot{z}_{t,k} - b^{(k)}\right)\|^2, \label{ineq: dot_z_t}
    \end{align}
    where the second inequality uses Lemma \ref{lem: gd_decrease}. For $\dot{s}_{t,k} - \left(H^{(k)}\dot{z}_{t,k} - b^{(k)}\right)$ we have
    \begin{align*}
        &\|\dot{s}_{t,k} - \left(H^{(k)}\dot{z}_{t,k} - b^{(k)}\right)\|^2 \\
        = &\frac{1}{n^2}\|\sum_{i=1}^{n}\left(\nabla_y^2g_i(x_{i,k},y_{i,k}^{(T)})\E\left[z_{i,t}^{(k)}|\F_k\right] - \nabla_y^2g_i(\bar{x}_k,y_k^*)\E\left[\bar{z}_t^{(k)}|\F_k\right] + \nabla_y f_i(\bar{x}_k, y_k^*) - \nabla_y f_i(x_{i,k}, y_{i,k}^{(T)})\right)\|^2 \\
        = &\frac{1}{n^2}\|\sum_{i=1}^{n}\left[\nabla_y^2g_i(x_{i,k},y_{i,k}^{(T)})\E\left[z_{i,t}^{(k)} - \bar{z}_t^{(k)}|\F_k\right] - \left(\nabla_y^2g_i(x_{i,k},y_{i,k}^{(T)}) - \nabla_y^2g_i(\bar{x}_k, \bar{y}_k^{(T)}) \right)\dot{z}_{t,k}\right]\\
        + &\sum_{i=1}^{n}\left[\left(\nabla_y^2g_i(\bar{x}_k, \bar{y}_k^{(T)}) - \nabla_y^2g_i(\bar{x}_k, y_k^*)\right)\dot{z}_{t,k} + \nabla_y f_i(\bar{x}_k, y_k^*) - \nabla_y f_i(\bar{x}_k, \bar{y}_k^{(T)}) + \nabla_y f_i(\bar{x}_k, \bar{y}_k^{(T)}) - \nabla_y f_i(x_{i,k}, y_{i,k}^{(T)})\right]\|^2 \\
        \leq & \frac{5}{n}\sum_{i=1}^{n}\left[L_{g,1}^2\|\E\left[z_{i,t}^{(k)} - \bar{z}_t^{(k)}|\F_k\right]\|^2 + (\|x_{i,k}-\bar{x}_k\|^2 + \|y_{i,k}^{(T)} - \bar{y}_k^{(T)}\|^2 + \|\bar{y}_k^{(T)} - y_k^*\|^2)(L_{g,2}^2\|\dot{z}_{t,k}\|^2 + L_{f,1}^2) \right] \\
        = & \frac{5L_{g,1}^2}{n}\|\E\left[Z_{t}^{(k)} - \bar{z}_{t}^{(k)}\bfonet|\F_k\right]\|^2 + \frac{5\left(L_{g,2}^2\sigma_z^2 + L_{f,1}^2\right)}{n}\left(n\|\bar{y}_k^{(T)} - y_k^*\|^2 + \|X_k - \bar{x}_k\bfonet\|^2 + \|Y_k^{(T)} - \bar{y}_k^{(T)}\bfonet\|^2\right).
    \end{align*}
    The above inequality and \eqref{ineq: dot_z_t} imply
    \begin{align}
        \E\left[\|\dot{z}_{N,k} - z_*^{(k)}\|^2\right]\leq & (1-\gamma\mu_g)\E\left[\|\dot{z}_{N-1,k} - z_*^{(k)}\|^2\right] + \left(\frac{\gamma}{\mu_g} + \gamma^2\right)\E\left[\|\dot{s}_{N-1,k} - \left(H^{(k)}\dot{z}_{N-1,k} - b^{(k)}\right)\|^2\right]\notag\\
            \leq &(1-\gamma\mu_g)^N\E\left[\|z_*^{(k)}\|^2\right] +  \frac{5L_{g,1}^2\left(\frac{\gamma}{\mu_g} + \gamma^2\right)}{n}\sum_{t=0}^{N-1}(1-\gamma\mu_g)^{N-1-t}\E\left[\|\E\left[Z_{t}^{(k)} - \bar{z}_{t}^{(k)}\bfonet|\F_k\right]\|^2\right] \notag\\
            +&\frac{\frac{\gamma}{\mu_g} + \gamma^2}{1-(1-\gamma\mu_g)}\cdot \frac{5\left(L_{g,2}^2\sigma_z^2 + L_{f,1}^2\right)}{n}\E\left[n\|\bar{y}_k^{(T)} - y_k^*\|^2 + \|X_k - \bar{x}_k\bfonet\|^2 + \|Y_k^{(T)} - \bar{y}_k^{(T)}\bfonet\|^2\right]\notag\\
            \leq &(1-\gamma\mu_g)^{N}\cdot\frac{L_{f,0}^2}{\mu_g^2} + 5\left(\frac{1}{\mu_g^2} + \frac{\gamma}{\mu_g}\right)\left(L_{g,2}^2\sigma_z^2 + L_{f,1}^2\right)\left(\E\left[\|\bar{y}_k^{(T)} - y_k^*\|^2\right] + \sigma_x^2\tilde{\alpha}_k^2 + T\tilde{\beta}_{k+1}^2\right) \notag\\
            + &5L_{g,1}^2\left(\frac{\gamma}{\mu_g} + \gamma^2\right)\sum_{t=0}^{N-1}\left(1-\frac{\gamma\mu_g}{2}\right)^{N-1-t}\left(3C_{\tilde{M}}\left(1-\frac{2\gamma\mu_g}{3}\right)^t\left(\frac{L_{f,0}^2}{\mu_g^2} + L_{f,0}^2\right)\right) \notag\\
            \leq& (1-\gamma\mu_g)^{N}\cdot\frac{L_{f,0}^2}{\mu_g^2} + 5\left(\frac{1}{\mu_g^2} + \frac{\gamma}{\mu_g}\right)\left(L_{g,2}^2\sigma_z^2 + L_{f,1}^2\right)\left(\E\left[\|\bar{y}_k^{(T)} - y_k^*\|^2\right] + \sigma_x^2\tilde{\alpha}_k^2 + T\tilde{\beta}_{k+1}^2\right) \notag \\
            + &90C_{\tilde{M}}L_{g,1}^2\left(\frac{1}{\mu_g^2} + \frac{\gamma}{\mu_g}\right)\left(\frac{L_{f,0}^2}{\mu_g^2} +L_{f,0}^2\right)\left(1-\frac{2\gamma\mu_g}{3}\right)^{N-1},\label{ineq: dot_z_N}
    \end{align}
    where the second inequality uses Lemma \ref{lem: useful_ineq}, the third inequality uses \eqref{ineq: z_exp_cons}, and the fourth inequality holds since
    \[
        \sum_{t=0}^{N-1-t}\left(1-\frac{\gamma\mu_g}{2}\right)^{N-1-t}\left(1-\frac{2\gamma\mu_g}{3}\right)^t = \left(1-\frac{2\gamma\mu_g}{3}\right)^{N-1}\sum_{t=0}^{N-1}\frac{1-\frac{2\gamma\mu_g}{3}}{1-\frac{\gamma\mu_g}{2}} < \left(1-\frac{2\gamma\mu_g}{3}\right)^{N-1}\cdot \frac{6}{\gamma\mu_g}.
    \]
\end{proof}
\begin{lemma}\label{lem: all_const_sum}
    If $0< \beta_k\leq 1$ and $\alpha_k > 0$ for any $k\geq 0$, then the parameters $\tilde{\alpha}_k,\ \tilde{\beta}_k,$ and $C_{k,T}$ defined in Lemmas \ref{lem: XY_cons} and \ref{lem: y_error} satisfy
    \begin{align*}
        &\sum_{k=0}^{K}\tilde{\alpha}_k^2\leq \frac{2}{1-\rho^2}\sum_{i=0}^{K-1}\alpha_i^2 = \cO\left(\sum_{k=0}^{K}\alpha_k^2\right)\\
        &\sum_{k=0}^{K}\tilde{\beta}_{k+1}^2\leq \frac{4(1+\rho^2)}{(1-\rho^2)^2}\left[\left(10\sigma_{g,1}^2 + 5\delta^2\right)\sum_{i=0}^{K}\beta_i^2 + \frac{20L_{g,1}^2\sigma_x^2}{1-\rho^2}\sum_{i=0}^{K-1}\alpha_i^2\right] = \cO\left(\sum_{k=0}^{K}(\alpha_k^2+\beta_k^2)\right) \\
        &\sum_{k=1}^{K}C_{k,T}\leq \left(\frac{1}{\mu_g} + 1\right)L_{g,1}^2\left[T\sigma_x^2\sum_{k=1}^{K}\tilde{\alpha}_k^2 + 2T^2\sum_{k=0}^{K}\tilde{\beta}_{k+1}^2\right] + \frac{T\sigma_{g,1}^2}{n}\sum_{k=1}^{K}\beta_k^2 = \cO\left(\sum_{k=0}^{K}(\alpha_k^2+\beta_k^2)\right).
    \end{align*}
\end{lemma}
\begin{proof}[Proof of Lemma~\ref{lem: all_const_sum}]
    The first inequality holds due to $\tilde{\alpha}_0 = 0$ and 
    \begin{align*}
        \sum_{k=0}^{K-1}\tilde{\alpha}_{k+1}^2 = \sum_{k=0}^{K-1}\sum_{i=0}^{k}\alpha_i^2\left(\frac{1+\rho^2}{2}\right)^{k-i} = \sum_{i=0}^{K-1}\sum_{k=i}^{K-1}\alpha_i^2\left(\frac{1+\rho^2}{2}\right)^{k-i}\leq \frac{2}{1-\rho^2}\sum_{i=0}^{K-1}\alpha_i^2.
    \end{align*}
    Similarly, we have
    \begin{align*}
        &\sum_{k=0}^{K}\tilde{\beta}_{k+1}^2 = \frac{1+\rho^2}{1-\rho^2}\sum_{k=0}^{K}\sum_{i=0}^{k}\beta_i^2(10\sigma_{g,1}^2 + 10L_{g,1}^2\sigma_x^2\tilde{\alpha}_i^2 + 5\delta^2)\left(\frac{3+\rho^2}{4}\right)^{k-i} \\
        \leq & \frac{4(1+\rho^2)}{(1-\rho^2)^2}\sum_{i=0}^{K}\beta_i^2(10\sigma_{g,1}^2 + 10L_{g,1}^2\sigma_x^2\tilde{\alpha}_i^2 + 5\delta^2) \leq \frac{4(1+\rho^2)}{(1-\rho^2)^2}\left[(10\sigma_{g,1}^2 + 5\delta^2)\sum_{i=0}^{K}\beta_i^2 + \frac{20L_{g,1}^2\sigma_x^2}{1-\rho^2}\sum_{i=0}^{K-1}\alpha_i^2\right].
    \end{align*}
    Lastly, we know
    \begin{align*}
        \sum_{k=1}^{K}C_{k,T} = &\sum_{k=1}^{K}\sum_{l=0}^{T-1}\left[\left(\frac{\beta_k}{\mu_g} + \beta_k^2\right)L_{g,1}^2\left(\sigma_x^2\tilde{\alpha}_k^2 + \left[\left(\frac{3+\rho^2}{4}\right)^{l}T - l\left(\frac{3+\rho^2}{4}\right)\right]\tilde{\beta}_k^2 + l\tilde{\beta}_{k+1}^2\right) + \frac{\beta_k^2\sigma_{g,1}^2}{n}\right] \\
        \leq &\sum_{k=1}^{K}\left(\frac{\beta_k}{\mu_g} + \beta_k^2\right)L_{g,1}^2\left(T\sigma_x^2\tilde{\alpha}_k^2 + T^2\tilde{\beta}_k^2 + T^2\tilde{\beta}_{k+1}^2\right) + \sum_{k=1}^{K}T\frac{\beta_k^2\sigma_{g,1}^2}{n} \\
        \leq &\left(\frac{1}{\mu_g} + 1\right)L_{g,1}^2\left[T\sigma_x^2\sum_{k=1}^{K}\tilde{\alpha}_k^2 + 2T^2\sum_{k=0}^{K}\tilde{\beta}_{k+1}^2\right] + \frac{T\sigma_{g,1}^2}{n}\sum_{k=1}^{K}\beta_k^2,
    \end{align*}
    where the last inequality uses $0< \beta_k\leq 1$.
\end{proof}

Now we are ready to give the proof of Theorem \ref{thm: main}. 

\begin{lemma}\label{lem: r_bar_sum}
    Suppose Assumptions \ref{assump: lip_convexity}, \ref{assump: W}, \ref{assump: similarity_of_g}, and \ref{assump: stoc_derivatives} hold. For Algorithm \ref{algo:DSBO} we have
    \begin{equation}\label{ineq: main_of_MA}
        \sum_{k=0}^{K}\left(\frac{\alpha_k}{2} - \frac{L_{\Phi}\alpha_k^2}{2}\right)\E\left[\|\bar{r}_k\|^2\right] \leq \frac{1}{2}\sum_{k=0}^{K}\alpha_k\E\left[\|\E\left[\bar{u}_k|\F_k\right] - \nabla\Phi(\bar{x}_k)\|^2\right] + 2\sigma_u^2\sum_{k=0}^{K}\alpha_k^2 + \Phi(0) - \inf_x\Phi(x) + \frac{1}{2}\E\left[\|\bar{r}_0\|^2\right].
    \end{equation}
\end{lemma}

\begin{proof}[Proof of Lemma~\ref{lem: r_bar_sum}]
    The $L_{\Phi}$-smoothness of $\Phi$ indicates that
\begin{equation}\label{ineq: l_smooth_ineq}
    \Phi(\bar{x}_{k+1}) - \Phi(\bar{x}_k) \leq \nabla\Phi(\bar{x}_k)\T(-\alpha_k\bar{r}_k) + \frac{L_{\Phi}\alpha_k^2}{2}\|\bar{r}_k\|^2.
\end{equation}
Notice that we also have
\begin{equation}\label{eq: difference_of_r}
    \frac{1}{2}\E\left[\|\bar{r}_{k+1}\|^2|\F_k\right] - \frac{1}{2}\|\bar{r}_k\|^2 = -\alpha_k\|\bar{r}_k\|^2 +\alpha_k\E\left[\bar{u}_k|\F_k\right]\T\bar{r}_k+ \frac{1}{2}\E\left[\|\bar{r}_{k+1} - \bar{r}_k\|^2|\F_k\right].
\end{equation}
Hence we know
\begin{align*}
    &\Phi(\bar{x}_{k+1}) - \Phi(\bar{x}_k) + \frac{1}{2}\E\left[\|\bar{r}_{k+1}\|^2|\F_k\right] - \frac{1}{2}\|\bar{r}_k\|^2\\
    \leq &\alpha_k(\E\left[\bar{u}_k|\F_k\right] - \nabla\Phi(\bar{x}_k))\T\bar{r}_k + (\frac{L_{\Phi}\alpha_k^2}{2} - \alpha_k)\|\bar{r}_k\|^2 + \frac{1}{2}\E\left[\|\bar{r}_{k+1} - \bar{r}_k\|^2|\F_k\right] \\
    \leq & \frac{\alpha_k}{2}\left(\|\E\left[\bar{u}_k|\F_k\right] - \nabla\Phi(\bar{x}_k)\|^2 + \|\bar{r}_k\|^2\right) + (\frac{L_{\Phi}\alpha_k^2}{2} - \alpha_k)\|\bar{r}_k\|^2 + \frac{1}{2}\E\left[\|\bar{r}_{k+1} - \bar{r}_k\|^2|\F_k\right],
\end{align*}
which implies
\begin{equation}\label{ineq: sum_r_original}
    \begin{aligned}
        &\left(\frac{\alpha_k}{2} - \frac{L_{\Phi}\alpha_k^2}{2}\right)\E\left[\|\bar{r}_k\|^2\right]\\
    \leq &\frac{\alpha_k}{2}\E\left[\|\E\left[\bar{u}_k|\F_k\right] - \nabla\Phi(\bar{x}_k)\|^2\right] +   \frac{1}{2}\E\left[\|\bar{r}_{k+1} - \bar{r}_k\|^2\right] + \E\left[\Phi(\bar{x}_k) - \Phi(\bar{x}_{k+1})\right] +\frac{1}{2}\E\left[\|\bar{r}_k\|^2\right] - \frac{1}{2}\E\left[\|\bar{r}_{k+1}\|^2\right]\\
    \leq &\frac{\alpha_k}{2}\E\left[\|\E\left[\bar{u}_k|\F_k\right] - \nabla\Phi(\bar{x}_k)\|^2\right] + 2\alpha_k^2\sigma_u^2 + \E\left[\Phi(\bar{x}_k) - \Phi(\bar{x}_{k+1})\right] +\frac{1}{2}\E\left[\|\bar{r}_k\|^2\right] - \frac{1}{2}\E\left[\|\bar{r}_{k+1}\|^2\right],
    \end{aligned}
\end{equation}
where the second inequality holds since we know
\begin{align*}
    &\E\left[\|\bar{r}_k\|^2\right]\leq \max\left(\E\left[\|\bar{r}_{k-1}\|^2\right], \E\left[\|\bar{u}_k\|^2\right]\right)\leq \max_{0\leq i\leq k}\E\left[\|\bar{u}_i\|^2\right]\leq \sigma_u^2, \\
    &\E\left[\|\bar{r}_{k+1} - \bar{r}_k\|^2\right] = \alpha_k^2\E\left[\|\bar{r}_k - \bar{u}_k\|^2\right]\leq 2\alpha_k^2\E\left[\|\bar{r}_k\|^2 + \|\bar{u}_k\|^2\right]\leq 4\sigma_u^2.
\end{align*}
In these two conclusions $\E\left[\|\bar{u}_i\|^2\right]\leq \sigma_u^2$ is due to the first inequality in \eqref{ineq: XY_consensus}. Taking summation on both sides of \eqref{ineq: sum_r_original}, we have~\eqref{ineq: main_of_MA}.
\end{proof}

\begin{lemma}\label{lem: r_and_grad}
    For Algorithm \ref{algo:DSBO} we have
    \begin{align}\label{eq:temp}
    \begin{aligned}
        \sum_{k=0}^{K}\alpha_k\E\left[\|\bar{r}_k -\nabla\Phi(\bar{x}_k)\|^2\right]&\leq \E\left[\|\bar{r}_0 - \nabla\Phi(0)\|^2\right] + 2\sum_{k=0}^{K}\alpha_k\E\left[\|\E\left[\bar{u}_k|\F_k\right] - \nabla\Phi(\bar{x}_k)\|^2\right] + \\
        &\qquad \qquad 2\sum_{k=0}^{K}\alpha_k\E\left[\|\bar{r}_k\|^2\right] + \sigma_u^2\sum_{k=0}^{K}\alpha_k^2.
        \end{aligned}
    \end{align}
\end{lemma}

\begin{proof}[Proof of Lemma~\ref{lem: r_and_grad}]
    Recall that in Algorithm \ref{algo:DSBO} we know
    \[
        \bar{r}_{k+1} = (1-\alpha_k)\bar{r}_k + \alpha_k\bar{u}_k,
    \]
    which implies 
    \begin{align*}
        &\|\bar{r}_{k+1} - \nabla\Phi(\bar{x}_{k+1})\| \\
        = &\|(1-\alpha_k)(\bar{r}_k -\nabla\Phi(\bar{x}_k)) + \alpha_k(\E\left[\bar{u}_k|\F_k\right] - \nabla\Phi(\bar{x}_k)) + \nabla\Phi(\bar{x}_k) - \nabla\Phi(\bar{x}_{k+1}) + \alpha_k(\bar{u}_k -\E\left[\bar{u}_k|\F_k\right])\|.
    \end{align*}
    Hence we know
    \begin{align*}
        &\E\left[\|\bar{r}_{k+1} - \nabla\Phi(\bar{x}_{k+1})\|^2\right] \\
        =&\E\left[\|(1-\alpha_k)(\bar{r}_k -\nabla\Phi(\bar{x}_k)) + \alpha_k(\E\left[\bar{u}_k|\F_k\right] - \nabla\Phi(\bar{x}_k)) + \nabla\Phi(\bar{x}_k) - \nabla\Phi(\bar{x}_{k+1})\|^2\right] + \alpha_k^2\E\left[\|\bar{u}_k -\E\left[\bar{u}_k|\F_k\right]\|^2\right]\\
        \leq & (1-\alpha_k)\E\left[\|\bar{r}_k -\nabla\Phi(\bar{x}_k)\|^2\right] + \alpha_k\E\left[\|\E\left[\bar{u}_k|\F_k\right] - \nabla\Phi(\bar{x}_k) + \frac{1}{\alpha_k}(\nabla\Phi(\bar{x}_k) - \nabla\Phi(\bar{x}_{k+1})\|^2\right] + \alpha_k^2\sigma_u^2 \\
        \leq &(1-\alpha_k)\E\left[\|\bar{r}_k -\nabla\Phi(\bar{x}_k)\|^2\right] + 2\alpha_k\E\left[\|\E\left[\bar{u}_k|\F_k\right] - \nabla\Phi(\bar{x}_k)\|^2 + \|\bar{r}_k\|^2\right]+ \alpha_k^2\sigma_u^2.
    \end{align*}
    Taking summation on both sides, we obtain~\eqref{eq:temp}.
\end{proof}
The next lemma characterizes $\|\nabla\Phi(\bar{x}_k) - \E\left[\bar{u}_k|\F_k\right]\|^2$, which together with previous lemmas prove Theorem \ref{thm: main}.
\begin{lemma}\label{lem: last}
    In Algorithm \ref{algo:DSBO} if we define
    \begin{align*}
        &\alpha_k = \frac{\mu_g^4}{3L_{g,1}^2C_y}\cdot\beta_k \equiv \frac{1}{\sqrt{K}},\ \gamma \text{ such that \eqref{ineq: gamma_for_quadratics} holds},\  N = \Theta(\log K),\ T\geq 1, \\
        &C_y = 5\left(L_{f,1}^2 + \frac{L_{g,2}^2L_{f,0}^2}{\mu_g^2}\right) + 50L_{g,1}^2\left(\frac{1}{\mu_g^2} + \frac{\gamma}{\mu_g}\right)\left(L_{g,2}^2\sigma_z^2 + L_{f,1}^2\right).
    \end{align*}
    we have
    \begin{align*}
        \sum_{k=0}^{K}\alpha_k\|\E\left[\bar{u}_k|\F_k\right] - \nabla\Phi(\bar{x}_k)\|^2 &= C_y\sum_{k=0}^{K}\alpha_k\|\bar{y}_k^{(T)} - y_k^*\|^2 + \cO\left(1 + \left(1-\frac{\gamma\mu_g}{2}\right)^{N}\sum_{k=0}^{K}\alpha_k\right), \\ 
        \frac{1}{K}\sum_{k=0}^{K}\E\left[\|\nabla\Phi(\bar{x}_k)\|^2\right] &= \cO\left(\frac{1}{\sqrt{K}}\right).
    \end{align*}
\end{lemma}
\begin{proof}[Proof of Lemma~\ref{lem: last}]
    Notice that we have
    \begin{align*}
        \E\left[\bar{u}_{k}|\F_k\right] &= \frac{1}{n}\sum_{i=1}^{n}\nabla_x f_i(x_{i,k},y_{i,k}^{(T)}) - \frac{1}{n}\sum_{i=1}^{n}\nabla_{xy}^2 g_i(x_{i,k},y_{i,k}^{(T)})\E\left[z_{i,N}^{(k)}|\F_k\right], \\
        \nabla\Phi(\bar{x}_k) &= \frac{1}{n}\sum_{i=1}^{n}\nabla_x f_i(\bar{x}_k, y_k^*) - \left(\frac{1}{n}\sum_{i=1}^{n}\nabla_{xy}^2 g_i(\bar{x}_k, y_k^*)\right)\left(\frac{1}{n}\sum_{i=1}^{n}\nabla_y^2g_i(\bar{x}_k, y_k^*)\right)^{-1}\left(\frac{1}{n}\sum_{i=1}^{n}\nabla_y f_i(\bar{x}_k, y_k^*)\right), \\
        &=\frac{1}{n}\sum_{i=1}^{n}\nabla_x f_i(\bar{x}_k, y_k^*) - \frac{1}{n}\left(\sum_{i=1}^{n}\nabla_{xy}^2 g_i(\bar{x}_k, y_k^*)\right)\left(\sum_{i=1}^{n}\nabla_y^2g_i(\bar{x}_k, y_k^*)\right)^{-1}\left(\sum_{i=1}^{n}\nabla_y f_i(\bar{x}_k, y_k^*)\right) \\
        &= \frac{1}{n}\sum_{i=1}^{n}\nabla_x f_i(\bar{x}_k, y_k^*) - \frac{1}{n}\left(\sum_{i=1}^{n}\nabla_{xy}^2 g_i(\bar{x}_k, y_k^*)\right)z_*^{(k)}.
    \end{align*}
    Hence we know
    \begin{align*}
        &\|\E\left[\bar{u}_k|\F_k\right] - \nabla\Phi(\bar{x}_k)\| \\
        = &\frac{1}{n}\sum_{i=1}^{n}\left(\|\nabla_x f_i(x_{i,k},y_{i,k}^{(T)}) - \nabla_xf_i(\bar{x}_k,\bar{y}_k^{(T)})\| + \|\nabla_xf_i(\bar{x}_k,\bar{y}_k^{(T)}) - \nabla_x f_i(\bar{x}_k, y_k^*)\|\right) \\
        + & \frac{1}{n}\sum_{i=1}^{n}\left(\|\nabla_{xy}^2g_i(x_{i,k},y_{i,k}^{(T)})\left(\E\left[z_{i,N}^{(k)}|\F_k\right] - z_*^{(k)}\right)\| + \|\left(\nabla_{xy}^2g_i(x_{i,k},y_{i,k}^{(T)}) - \nabla_{xy}^2 g_i(\bar{x}_k, \bar{y}_k^{(T)})\right)z_*^{(k)}\|\right) \\
        + & \frac{1}{n}\sum_{i=1}^{n}\|\left(\nabla_{xy}^2 g_i(\bar{x}_k, \bar{y}_k^{(T)}) - \nabla_{xy}^2 g_i(\bar{x}_k, y_k^*)\right)z_*^{(k)}\|,
    \end{align*}
    which implies
    \begin{align*}
        &\|\E\left[\bar{u}_k|\F_k\right] - \nabla\Phi(\bar{x}_k)\|^2 \\
        \leq &\frac{5}{n}\sum_{i=1}^{n}\left[L_{f,1}^2\left(\|x_{i,k}-\bar{x}_k\|^2+ \|y_{i,k}^{(T)} - \bar{y}_k^{(T)}\|^2 + \|\bar{y}_k^{(T)} - y_k^*\|^2\right) + L_{g,1}^2\|\E\left[z_{i,N}^{(k)}|\F_k\right] - z_*^{(k)}\|^2\right] \\
        + &\frac{5}{n}\sum_{i=1}^{n}\left[\frac{L_{g,2}^2L_{f,0}^2}{\mu_g^2}\left(\|x_{i,k}-\bar{x}_k\|^2 + \|y_{i,k}^{(T)} - \bar{y}_k^{(T)}\|^2 + \|\bar{y}_k^{(T)} - y_k^*\|^2\right)\right] \\
        \leq & 5\left(L_{f,1}^2 + \frac{L_{g,2}^2L_{f,0}^2}{\mu_g^2}\right)\cdot \frac{1}{n}\left(\|X_k - \bar{x}_k\bfonet\|^2 + \|Y_k^{(T)} - \bar{y}_k^{(T)}\bfonet\|^2 + n\|\bar{y}_k^{(T)} - y_k^*\|^2\right) \\
        + &10L_{g,1}^2\cdot\frac{1}{n}\left(\|\E\left[Z_N^{(k)} - \bar{z}_N^{(k)}\bfonet|\F_k\right]\|^2\right) + 10L_{g,1}^2\|\E\left[\bar{z}_N^{(k)}|\F_k\right] - z_*^{(k)}\|^2 \\
        \leq & \left[5\left(L_{f,1}^2 + \frac{L_{g,2}^2L_{f,0}^2}{\mu_g^2}\right) + 50L_{g,1}^2\left(\frac{1}{\mu_g^2} + \frac{\gamma}{\mu_g}\right)\left(L_{g,2}^2\sigma_z^2 + L_{f,1}^2\right)\right]\cdot \left(\|\bar{y}_k^{(T)} - y_k^*\|^2 + \sigma_x^2\tilde{\alpha}_k^2 + T\tilde{\beta}_{k+1}^2\right) \\
        + & 30L_{g,1}^2\left(1-\frac{\gamma\mu_g}{2}\right)^N\left(\frac{L_{f,0}^2}{\mu_g^2} + L_{f,0}^2\right) \\
        +& 10L_{g,1}^2\left[(1-\gamma\mu_g)^{N}\cdot\frac{L_{f,0}^2}{\mu_g^2} + 90C_{\tilde{M}}L_{g,1}^2\left(\frac{1}{\mu_g^2} + \frac{\gamma}{\mu_g}\right)\left(\frac{L_{f,0}^2}{\mu_g^2} +L_{f,0}^2\right)\left(1-\frac{2\gamma\mu_g}{3}\right)^{N-1}\right],
    \end{align*}
    where the third inequality uses \eqref{ineq: XY_consensus}, \eqref{ineq: z_exp_cons} and \eqref{ineq: z_conv_higp}. Taking summation on both sides, we have
    \begin{equation}
        \begin{aligned}
            \sum_{k=0}^{K}\alpha_k\|\E\left[\bar{u}_k|\F_k\right] - \nabla\Phi(\bar{x}_k)\|^2 = C_y\sum_{k=0}^{K}\alpha_k\|\bar{y}_k^{(T)} - y_k^*\|^2 + \cO\left(\sum_{k=0}^{K}\alpha_k(\tilde{\alpha}_k^2 + \tilde{\beta}_k^2) + \left(1-\frac{\gamma\mu_g}{2}\right)^{N-1}\sum_{k=0}^{K}\alpha_k\right).
        \end{aligned}
    \end{equation}
    Setting for all $k$ that
    \[
        \alpha_k = C_{\alpha,\beta}\cdot\beta_k \equiv \frac{1}{\sqrt{K}},\ C_{\alpha,\beta} = \frac{\mu_g}{2\sqrt{3C_y}L_{y^*}},
    \]
    and using \eqref{ineq: bar_y_sum} and Lemma \ref{lem: all_const_sum}, we know
    \begin{equation}\label{ineq: r_and_grad_final}
        \begin{aligned}
            &\frac{1}{\sqrt{K}}\sum_{k=0}^{K}\E\left[\|\E\left[\bar{u}_k|\F_k\right] - \nabla\Phi(\bar{x}_k)\|^2\right]=C_yC_{\alpha,\beta}\sum_{k=0}^{K}\beta_k\|\bar{y}_k^{(T)} - y_k^*\|^2 + \cO\left(\frac{1}{\sqrt{K}} + \sqrt{K}\left(1-\frac{\gamma\mu_g}{2}\right)^{N-1}\right)\\
            = &C_yC_{\alpha,\beta}L_{y^*}^2\sum_{k=0}^{K}\left(\frac{4C_{\alpha,\beta}}{\sqrt{K}\mu_g^2}+\frac{2}{K\mu_g}\right)\E\left[\|\bar{r}_k\|^2\right] + \cO\left(1 + \sqrt{K}\left(1-\frac{\gamma\mu_g}{2}\right)^{N-1}\right)\\ 
            =&\sum_{k=1}^{K}\left(\frac{1}{3\sqrt{K}}+\frac{2C_yC_{\alpha,\beta}L_{y^*}^2}{K\mu_g}\right)\E\left[\|\bar{r}_k\|^2\right] + \cO\left(1 + \sqrt{K}\left(1-\frac{\gamma\mu_g}{2}\right)^{N-1}\right),
        \end{aligned}
    \end{equation}
    which together with \eqref{ineq: main_of_MA} and \eqref{eq: smoothness_const} imply 
    \begin{align*}
        &\left(\frac{1}{2\sqrt{K}} - \frac{L_{\Phi}}{2K}\right)\sum_{k=0}^{K}\E\left[\|\bar{r}_k\|^2\right] \\
        \leq &\frac{1}{2\sqrt{K}}\sum_{k=0}^{K}\E\left[\|\E\left[\bar{u}_k|\F_k\right] - \nabla\Phi(\bar{x}_k)\|^2\right] + 2\sigma_u^2\sum_{k=0}^{K}\frac{1}{K} + \Phi(0) - \inf_x\Phi(x) + \frac{1}{2}\E\left[\|\bar{r}_0\|^2\right] \\
        \leq &\frac{1}{2\sqrt{K}}\sum_{k=1}^{K}\left(\frac{1}{3}+\frac{2C_yC_{\alpha,\beta}L_{y^*}^2}{\sqrt{K}\mu_g}\right)\E\left[\|\bar{r}_k\|^2\right] + \cO\left(1 + \sqrt{K}\left(1-\frac{\gamma\mu_g}{2}\right)^{N-1}\right).
    \end{align*}
    Hence we know
    \begin{align*}
        \left(\frac{1}{3\sqrt{K}} - \frac{L_{\Phi}}{2K} - \frac{C_yC_{\alpha,\beta}L_{y^*}^2}{K\mu_g}\right)\sum_{k=0}^{K}\E\left[\|\bar{r}_k\|^2\right] = \cO\left(1 + \sqrt{K}\left(1-\frac{\gamma\mu_g}{2}\right)^{N-1}\right).
    \end{align*}
    Using the above expression, \eqref{ineq: r_and_grad_final} and Lemma \ref{lem: r_and_grad}, we know 
    \begin{align*}
        \frac{1}{\sqrt{K}}\sum_{k=0}^{K}\E\left[\|\nabla\Phi(\bar{x}_k)\|^2\right]\leq \frac{2}{\sqrt{K}}\sum_{k=0}^{K}\E\left[\|\bar{r}_k\|^2 + \|\bar{r}_k - \nabla\Phi(\bar{x}_k)\|^2\right] = \cO\left(1 + \sqrt{K}\left(1-\frac{\gamma\mu_g}{2}\right)^{N-1}\right),
    \end{align*}
    for sufficiently large $K$. Note that $\gamma$ is in a constant interval by \eqref{ineq: gamma_for_quadratics}, hence $\left(1-\frac{\gamma\mu_g}{2}\right)$ is a constant that is independent of $K$. Picking $N = \Theta(\log K)$ such that $\left(1-\frac{\gamma\mu_g}{2}\right)^{N-1} = \cO\left(\frac{1}{\sqrt{K}}\right)$, we know
    \[
        \frac{1}{K}\sum_{k=0}^{K}\E\left[\|\nabla\Phi(\bar{x}_k)\|^2\right] = \cO\left(\frac{1}{\sqrt{K}}\right).
    \]
    Moreover, from \eqref{ineq: XY_consensus} we know:
    \[
        \frac{1}{K}\sum_{k=0}^{K}\frac{\E\left[\|X_k - \bar{x}_k\bfonet\|^2\right]}{n} = \cO\left(\frac{1}{K}\sum_{k=0}^{K}\tilde{\alpha}_k^2\right) = \cO\left(\frac{1}{K}\right),
    \]
    where the second equality holds due to Lemma \ref{lem: all_const_sum}
    The above two equalities prove Theorem \ref{thm: main}. To find an $\epsilon$-stationary point, we may set $K = \Theta(\epsilon^{-2})$ and we know from $T\geq 1,\ N=\log K$ that the sample complexity will be
    $\tilde{\cO}(\epsilon^{-2})$.
\end{proof}

\section{Discussion}\label{sec: discussion}
We briefly discuss Assumption 3.4 (iv) and (v) in \citet{yang2022decentralized} and MDBO in \cite{gao2022stochastic} in this section.
\subsection{Assumption 3.4 (iv) and (v) in \citet{yang2022decentralized}}
\begin{itemize}
    \item Assumption 3.4 (iv) assumes bounded second moment of $\nabla_yg_i(x,y;\xi)$. It is stronger than our Assumption \ref{assump: similarity_of_g} as discussed right after Assumption \ref{assump: similarity_of_g}. \\
    As pointed out by one reviewer during the discussion period, bounded moment condition on $\nabla_y g_i(x,y;\xi)$ is also restrictive especially when $g_i$ is strongly convex in $y$. To see this, we notice that the unbiasedness of $\nabla_y g_i(x,y;\xi)$ and its bounded second moment imply
\[
    \|\nabla_y g(x,y)\|^2 = \E\left[\|\nabla_y g(x,y;\xi)\|^2\right] -  \E\left[\|\nabla g(x,y) - \E\left[\nabla_y g(x,y;\xi)\right]\|^2\right]\leq C_g^2
\]
for all $x, y$. Here $\nabla_y g(x,y;\xi):=\frac{1}{n}\sum_{i=1}^{n}\nabla_y g_i(x,y;\xi_i)$. Then for any $y_1, y_2$
    \[
        2C_g\geq \|\nabla_y g(x, y_1) - \nabla_y g(x, y_2)\| \geq \mu_g \|y_1 - y_2\|
    \]
where the second inequality uses the fact that $g(x,y)$ is $\mu_g$-strongly convex in $y$ for any $x$. However $\sup_{y_1, y_2}\|y_1 - y_2\|=+\infty$, which leads to the contradiction, meaning that there does not exist a function $g$ satisfying all the assumptions above. In short, a function cannot be strongly convex and have bounded gradient at the same time , but both assumptions are used in \citet{yang2022decentralized}.
    
    \item Assumption 3.4 (v) assumes each $I - \frac{1}{L_g}\nabla_y^2g_i(x,y;\xi)$ has bounded second moment such that
\[
    \E\left[\|I - \frac{1}{L_g}\nabla_y^2g_i(x,y;\xi)\|_2^2\right]\leq (1-\kappa_g)^2,
\]
for some constant $\kappa_g\in (0, \frac{\mu_g}{L_g})$, where $L_g = \sqrt{L_{g,2}^2 + \sigma_{g,2}^2}$. It serves as a key role in proving the linear convergence of the Hessian matrix inverse estimator (see Lemma A.2, A.3 and the definition of $b$ right under section B of the Supplementary Material). However, it is restrictive under certain cases. For any given $0<\mu_g<L_g$, consider $X\in\R^{2\times 2}$ to be a random matrix and
    \[
        X = 
        \begin{pmatrix}
            2L_g &0 \\
            0 &0 
        \end{pmatrix} \text{ or }
        \begin{pmatrix}
            0 &0 \\
            0 &2\mu_g 
        \end{pmatrix} \text{ with equal probability},
    \]
    then it is easy to verify that $X$ has bounded variance and in expectation equals $\text{diag}(L, \mu)$, but
    \[
        \E\left[\|I - \frac{1}{L_g}X\|_2^2\right]=1,
    \]
    and thus their Assumption 3.4 (v) does not hold in this example. 
\end{itemize}

\subsection{MDBO}\label{sec: MDBO_discussion}
Although \citet{gao2022stochastic} claims that they solve the G-DSBO problem, their hypergradient (see equations (2) and (3) of their paper accessed from arXiv at the time of the submission of our manuscript to ICML: \url{https://arxiv.org/abs/2206.15025v1})
is defined as
\[
    \nabla F(x) := \frac{1}{K}\sum_{k=1}^{K}\nabla F^{(k)}(x),
\]
where 
\[
    \nabla F^{(k)}(x) := \nabla_x f^{(k)}(x, y^*(x)) - \nabla_{xy}^2g^{(k)}(x, y^*(x))(\nabla_y^2 g^{(k)}(x,y^*(x)))^{-1}\nabla_y f^{(k)}(x,y^*(x)).
\]
Clearly, this is not the hypergradient of G-DSBO, unless $g^{(i)}(x,y)=g^{(j)}(x,y)$ for any $1\leq i<j \leq n$, which requires an additional assumption that the data distributions that generate the lower level function $g^{(i)}$ are the same.
Note that their algorithm cannot be classified as P-DSBO either, because $y^*(x)$ in the above expression is defined globally. Therefore, their algorithm is not designed for neither G-DSBO nor P-DSBO. It is not clear what problem that their algorithm is designed for. 

While we are preparing our camera-ready version, we find the latest version of \citet{gao2022stochastic} (which is \citet{gao2023convergence}), which implicitly uses the condition that all lower level functions are the same. See equation (2) on page 3 of \cite{gao2023convergence} and the description right above it: ``Then, according to Lemma 1 of (Gao, 2022a), we can compute the gradient of $F^{(k)}(x)$ as follows:", where ``(Gao, 2022a)" represents \citet{gao2022convergence_fed}, in which their Lemma 1 explicitly states ``When the data distributions across all devices are homogeneous". However, all assumptions about MDBO in \citet{gao2022stochastic} do not mention anything about the data distributions of the lower level functions $g^{(i)}$. It should be noted that once all lower level functions $g^{(i)}$ are the same then their problem setup is one special case of ours in \eqref{eq: dsbo_opt} (i.e., when $g^{(i)} = g^{(j)}$ for any $i\neq j$), and it does not need to tackle the major challenge discussed in \eqref{neq: mismatch}.

\subsection{Computational complexity}\label{sec: comp}
Assume that computing a stochastic derivative with size $m$ requires $\cO(m)$ computational complexity. For example the complexity of computing a stochastic Hessian matrix $\nabla_y^2g_i(x,y;\xi)$ is $\cO(q^2)$ and the complexity of computing a stochastic gradient $\nabla_x f(x,y;\phi)$ is $\cO(p)$. Note that computing a Hessian-vector product (or Jacobian-vector product) is as cheap as computing a gradient \citep{pearlmutter1994fast, bottou2018optimization}. FEDNEST \citep{tarzanagh2022fednest}, SPDB \citep{lu2022decentralized}, and our Algorithm \ref{algo:DSBO} MA-DSBO only require stochastic first order and matrix-vector product oracles and thus the computational complexity is $\tilde \cO(d\epsilon^{-2})$, where $d:=\max(p, q)$. Note that DSBO-JHIP \citep{chen2022decentralized} requires computing full Jacobian matrices which lead to $\tilde \cO(pq\epsilon^{-3})$ complexity. GBDSBO \citep{yang2022decentralized} computes full Hessian matrices in the Hessian inverse estimation inner loop (Line 10-13 of Algorithm 1 in \citet{yang2022decentralized}), and full Jacobian matrices in the outer loop (Line 8 of Algorithm 1 in \citet{yang2022decentralized}), and thus their computational cost is $\cO((q^2\log(\frac{1}{\epsilon}) + pq)n^{-1}\epsilon^{-2})$.

\end{document}